\documentclass[a4paper]{amsart}
\usepackage{amsmath,amsthm,amssymb,mathtools,mathrsfs}
\usepackage{bm}
\usepackage{tikz-cd}
\usepackage[all]{xy}
\usepackage{hyperref}
    \definecolor{urlcolor}{rgb}{0,0,0}
    \definecolor{linkcolor}{rgb}{.7,0.10,0.2}
    \definecolor{citecolor}{rgb}{.12,.54,.11}
\hypersetup{
      breaklinks=true, 
      colorlinks=true,
      urlcolor=urlcolor,
      linkcolor=linkcolor,
      citecolor=citecolor,
}

\usepackage{setspace}
\usepackage{xcolor}
\usepackage[shortlabels]{enumitem}
\usepackage[utf8]{inputenc}
\numberwithin{equation}{section}
\usepackage[top=2cm,
bottom=2cm,
left=2cm,
right=2cm]{geometry}

\usepackage[
style=alphabetic,
doi=false,
isbn=false,
maxnames=8,
url=false,
eprint=false,
block = space,
sorting=anyt,
maxcitenames=10]{biblatex} 
\newtheorem{theorem}{Theorem}[section]
\newtheorem{corollary}[theorem]{Corollary}
\newtheorem{proposition}[theorem]{Proposition}
\newtheorem{proposition-definition}[theorem]{Proposition-Definition}
\newtheorem{lemma}[theorem]{Lemma}

\newtheorem{conjecture}[theorem]{Conjecture}
\newtheorem{question}[theorem]{Question}

\theoremstyle{definition} 
\newtheorem{definition}[theorem]{Definition}

\newtheorem{theorem-definition}[theorem]{Theorem-Definition}

\theoremstyle{remark} 
\newtheorem{remark}[theorem]{Remark}

\renewcommand{\geq}{\geqslant}
\renewcommand{\leq}{\leqslant}
\renewcommand{\subset}{\subseteq}
\renewcommand{\setminus}{\smallsetminus}
\renewcommand{\tilde}{\widetilde}

\newcommand{\BD}{\mathbb{D}}

\newcommand{\BH}{\mathbb{H}}

\newcommand{\BT}{\mathbb{T}}

\newcommand{\BoA}{\mathbf{A}}
\newcommand{\BoC}{\mathbf{C}}

\newcommand{\BoL}{\mathbf{L}}

\newcommand{\BoN}{\mathbf{N}}
\newcommand{\BoP}{\mathbf{P}}
\newcommand{\BoQ}{\mathbf{Q}}

\newcommand{\BoZ}{\mathbf{Z}}

\newcommand{\CA}{\mathcal{A}}
\newcommand{\CB}{\mathcal{B}}
\newcommand{\CC}{\mathcal{C}}
\newcommand{\CD}{\mathcal{D}}

\newcommand{\CF}{\mathcal{F}}
\newcommand{\CG}{\mathcal{G}}
\newcommand{\CH}{\mathcal{H}}

\newcommand{\CK}{\mathcal{K}}

\newcommand{\CM}{\mathcal{M}}

\newcommand{\CO}{\mathcal{O}}

\newcommand{\CS}{\mathcal{S}}
\newcommand{\CT}{\mathcal{T}}

\newcommand{\CX}{\mathcal{X}}

\newcommand{\FM}{\mathfrak{M}}

\newcommand{\FX}{\mathfrak{X}}

\newcommand{\Fn}{\mathfrak{n}}

\newcommand{\SA}{\mathscr{A}}

\newcommand{\SC}{\mathscr{C}}

\newcommand{\SF}{\mathscr{F}}
\newcommand{\SG}{\mathscr{G}}
\newcommand{\SH}{\mathscr{H}}

\DeclareMathOperator{\Dol}{Dol}

\DeclareMathOperator{\ICA}{IH}

\newcommand{\Alg}{\mathrm{Alg}}
\newcommand{\an}{\mathrm{an}}
\newcommand{\rmB}{\mathrm{B}}
\newcommand{\rmBM}{\mathrm{BM}}
\newcommand{\rmBPS}{\mathrm{BPS}}

\newcommand{\ch}{\mathrm{ch}}
\newcommand{\cl}{\mathrm{cl}}

\newcommand{\cone}{\mathrm{cone}}

\newcommand{\dR}{\mathrm{dR}}
\newcommand{\DR}{\mathbf{DR}}
\newcommand{\rmDR}{\mathrm{DR}}
\newcommand{\dd}{\mathbf{d}}
\newcommand{\dSt}{\mathrm{dSt}}

\newcommand{\dg}{\mathrm{dg}}

\newcommand{\et}{\mathrm{\acute{e}t}}

\newcommand{\forg}{\operatorname{forg}}
\newcommand{\Free}{\operatorname{Free}}
\newcommand{\GL}{\mathrm{GL}}

\newcommand{\GSp}{\mathrm{GSp}}
\newcommand{\gr}{\mathrm{gr}}

\newcommand{\Hod}{\mathrm{Hod}}
\newcommand{\Hom}{\operatorname{Hom}}
\newcommand{\IC}{\mathcal{IC}}

\newcommand{\id}{\operatorname{id}}

\newcommand{\intHom}{\mathcal{H}om}

\newcommand{\JH}{\mathtt{JH}}

\newcommand{\Lie}{\mathrm{Lie}}

\newcommand{\Map}{\mathrm{Map}}
\newcommand{\MHM}{\mathrm{MHM}}
\newcommand{\MMHM}{\mathrm{MMHM}}

\newcommand{\mon}{\mathrm{mon}}

\newcommand{\Perf}{\operatorname{Perf}}
\newcommand{\Perv}{\operatorname{Perv}}

\newcommand{\pH}[1]{{^\mathfrak{p}\CH^{#1}}}
\newcommand{\pSH}[2]{{^{\mathfrak{p}/#1}\!\CH^{#2}}}
\newcommand{\pYH}[1]{{^{\mathfrak{p}/Y}\!\CH^{#1}}}
\newcommand{\pAH}[2]{{^{\mathfrak{p}/#1}\!\CH^{#2}}}
\newcommand{\pSD}[2]{^{\mathfrak{p}/{#1}}\!\mathcal{D}^{#2}_{\rmc}}

\newcommand{\pDc}[1]{{^{\mathfrak{p}}\mathcal{D}_{\rmc}^{#1}}}
\newcommand{\pr}{\mathrm{pr}}

\newcommand{\ptau}[1]{{^\mathfrak{p}\tau^{#1}}}
\newcommand{\pStau}[2]{{^{\mathfrak{p}/#1}\!\tau^{#2}}}
\newcommand{\pYtau}[1]{{^{\mathfrak{p}/Y}\!\tau^{#1}}}

\newcommand{\rank}{\operatorname{rank}}
\newcommand{\rat}{\textbf{rat}}

\newcommand{\rmd}{\text{d}}

\newcommand{\Sh}{\mathrm{Sh}}

\newcommand{\red}{\mathrm{red}}

\newcommand{\RHom}{\mathrm{RHom}}
\newcommand{\rmb}{\mathrm{b}}
\newcommand{\rmc}{\mathrm{c}}
\newcommand{\qcoh}{\mathrm{qcoh}}
\newcommand{\Spec}{\operatorname{Spec}}

\newcommand{\sing}{\mathrm{sing}}
\newcommand{\sm}{\mathrm{sm}}

\newcommand{\Tan}{\operatorname{T}}

\newcommand{\Tot}{\operatorname{Tot}}

\newcommand{\vir}{\mathrm{vir}}

\newcommand{\vrank}{\operatorname{vrank}}
\DeclareMathOperator{\Betti}{Betti}

\newcommand{\VHS}{\operatorname{VHS}}

\newcommand{\BPS}{\mathcal{BPS}} 
\DeclareMathOperator{\Sym}{Sym}
\newcommand{\pt}{\mathrm{pt}}
\DeclareMathOperator{\HO}{H}
\newcommand{\cms}{/\!\!/}

\title{Nonabelian Hodge isomorphisms for stacks and cohomological Hall algebras}
\date{\today}

\author{Lucien Hennecart}

\address{Laboratoire Ami\'enois de Math\'ematique Fondamentale et Appliqu\'ee, CNRS UMR 7352, Universit\'e de Picardie Jules Verne, 33 rue Saint Leu, 80000 Amiens, France}
\email{lucien.hennecart@u-picardie.fr}

\begin{document}

\begin{abstract}
In this paper, we complete the nonabelian Hodge theory (NAHT) triangle of isomorphisms for stacks between the Borel--Moore homologies of the Dolbeault, Betti, and de Rham moduli stacks. We first explain how to realise the category of connections on a smooth projective curve as a subcategory of a $2$-Calabi--Yau dg-category satisfying some appropriate geometric conditions. Then, we define a cohomological Hall algebra (CoHA) product on the Borel--Moore homology of the stack of connections on a smooth projective curve. This allows us to not only compare the Borel--Moore homologies of the stacks at the relative and absolute levels for the three sides of NAHT, but also to compare their CoHA structures: they all coincide. To compare the Dolbeault and de Rham sides, we define a relative CoHA for the Hodge--Deligne moduli space parametrising $\lambda$-connections. The Betti and de Rham sides are compared using the (derived) Riemann--Hilbert correspondence. The comparison of the Borel--Moore homologies of the Dolbeault and Betti moduli stacks was previously considered by the author with Davison and Schlegel Mejia (without taking the cohomological Hall algebra structures into account). This paper completes this study and provides a CoHA enhancement of the classical NAHT for curves.
\end{abstract}

\maketitle

\vspace{0.5cm}

\setcounter{tocdepth}{1}
\tableofcontents

\section{Introduction}

\subsection{Nonabelian Hodge correspondence}

Nonabelian Hodge theory (NAHT) is a deep and elegant non-algebraic relationship between three kinds of algebraic objects on a smooth projective variety: local systems, semistable Higgs bundles with vanishing Chern classes (of slope $0$), and integrable (flat) connections \cite{hitchin1987self,simpson1988constructing,corlette1988flat,donaldson1987twisted}. The non-algebraicity of this correspondence witnesses the rich structure of a (noncompact) hyperk\"ahler manifold. It takes its roots in \cite{narasimhan1965stable}, the Narasimhan--Seshadri theorem establishing a bijection between irreducible unitary representations of the fundamental group of the curve and stable vector bundles. This bijection was then extended to all semisimple representations of the fundamental group and all semistable Higgs bundles and this correspondence lies at the heart of NAHT. It involves harmonic bundles, which serve as a bridge between Higgs bundles and flat connections via deep analytic methods. The parabolic cases were also considered and still constitute a subject of active research with for example wild character varieties and wild NAHT \cite{biquard2004wild}.

Nonabelian Hodge theory should be considered as an analogue in higher dimensions (i.e. for nonabelian structure groups) of the natural comparison isomorphisms between the cohomological degree one Betti (=singular), de Rham and Dolbeault cohomologies of a smooth projective complex manifold (the so-called de Rham and Dolbeault theorems, see \cite{simpson1992higgs}, and \cite{simpson1990nonabelian} for an overview), which corresponds to the rank one case (i.e. to the structure group $\mathbf{G}_{\mathrm{m}}$).

This theory deals with three kinds of fundamental objects (bearing the name of the corresponding cohomology theories):
\begin{enumerate}
 \item semistable Higgs bundles of slope $0$ (with vanishing Chern classes), forming the \emph{Dolbeault side} and giving the Dolbeault moduli space $\CM^{\Dol}(X)$,
 \item local systems, forming the \emph{Betti side}, with the Betti moduli space $\CM^{\Betti}(X)$,
 \item flat connections, forming the \emph{de Rham side} to which one associates the de Rham moduli space $\CM^{\dR}(X)$.
\end{enumerate}

In this paper, we study a version of the nonabelian Hodge theory/correspondence for smooth projective curves involving the \emph{moduli stacks} rather than the moduli spaces. Moduli stacks were already present in the work of Simpson \cite{simpson1996hodge}, and lead to questions related to upgrades of the classical nonabelian Hodge isomorphisms. The interest in considering stacks is twofold. First, one recovers the corresponding moduli spaces by taking the good moduli spaces of the stacks. The stacks can therefore be viewed as enrichments of the moduli spaces. Second, the stacks enjoy more symmetries and structure than the moduli spaces. This point manifests itself through the cohomological Hall algebras one can built from them (\cite{minets2020cohomological,sala2020comological} for Higgs sheaves; \cite{davison2016cohomological,mistry2022cohomological,davison2022BPS} for local systems, and the present work for connections, and also \cite{porta2022two} for categorified versions). The constructions of these stacks appeared (even if sometimes implicitly as a scheme acted on by an algebraic group) in Simpson's foundational work \cite{simpson1994moduli}. The natural maps between the set of closed $\BoC$-points of these stacks coming from various equivalences of categories are not all continuous, although they become homeomorphisms after taking the good moduli spaces (see \eqref{equation:diagramNAHT}). Despite this disappointing fact, we proved in \cite{davison2022BPS} that (when $X$ is a smooth projective curve) the Borel--Moore homologies of the Betti and Dolbeault moduli stacks coincide, extending Davison's result in genera $\leq 1$ \cite{davison2023nonabelian} to curves of any genus. This is a very surprising fact given the lack of continuity of the natural maps between these spaces (in some sense, we constructed an isomorphism in Borel--Moore homology from non-continuous maps), but still natural to expect from the viewpoint of (dimensionally reduced) Donaldson--Thomas theory (given the strong structural results in this theory, \cite{davison2023nonabelian,davison2022BPS,davison2023bps}).

One may choose a point $p\in X$ on the smooth projective variety and consider framed objects (which is a way of rigidifying the moduli stacks). In this case, the moduli problems admit fine moduli spaces in the category of (finite type, separated, complex) schemes. This is the approach developped by Simpson in \cite{simpson1994moduli}. He obtains the \emph{representation spaces} $R^{\Dol}_r(X), R^{\Betti}_r(X)$ and $R^{\dR}_r(X)$. The change of framing induces $\GL_r$-actions on these three algebraic variety. Nonabelian Hodge theory gives $\GL_r$-equivariant set-theoretic bijections between the three spaces $R^{\sharp}(X)$, $\sharp\in\{\Dol,\Betti,\dR\}$ and after taking the categorical quotient (with respect to some natural linearisations), these bijections induce homeomorphisms. Simpson also proved that the $\GL_r$-equivariant map induced by the Riemann--Hilbert correspondence is a homeomorphism, already at the level of the representation spaces $R_r^{\dR}(X)$ and $R_r^{\Betti}(X)$. One can summarise this NAHT picture by the diagram \eqref{equation:diagramNAHT}.

In this diagram, all maps in the external triangle are $\GL_r$-equivariant, and the maps between $R^{\Dol}_r(X)$ and $R^{\dR}_r(X)$ on the one hand, and between $R^{\Dol}_r(X)$ and $R_r^{\Betti}(X)$ on the other hand, are only bijections between the sets of $\BoC$-points (not continuous in general). A counterexample to the continuity of these maps was explained by Simpson in \cite[Counterexample p.38]{simpson1994moduliII}.

From the Riemann--Hilbert correspondence, one obtains isomorphisms in equivariant cohomology and equivariant Borel--Moore homology
\[
 \HO^*_{\GL_r}(R^{\Betti}_r(X))\cong\HO^*_{\GL_r}(R^{\dR}_r(X)), \quad \HO^{\rmBM}_{*,\GL_r}(R^{\Betti}_r(X))\cong\HO^{\rmBM}_{*,\GL_r}(R^{\dR}_r(X)).
\]
A natural question, despite the lacking continuity of some maps in the outer triangle of \eqref{equation:diagramNAHT}, is then the following.

\begin{question}
\label{question:comparisonBMstacks}
 Can one compare the equivariant cohomologies $\HO^*_{\GL_r}(R_r^{\dR}(X))$ and $\HO^*_{\GL_r}(R_r^{\Dol}(X))$ or the equivariant Borel--Moore homologies $\HO^{\rmBM}_{*,\GL_r}(R^{\dR}_r(X))$ and $\HO^{\rmBM}_{*,\GL_r}(R^{\Dol}_r(X))$?
\end{question}
\begin{equation}
\label{equation:diagramNAHT}
 \begin{tikzcd}
	&&& {R^{\dR}_r(X)} \\
	&& {\CM^{\dR}_r(X)} \\
	{R^{\Dol}_r(X)} & {\CM^{\Dol}_r(X)} \\
	&& {\CM^{\Betti}_r(X)} \\
	&&& {R^{\Betti}_r(X)}\\
	\arrow["{\substack{\cong\\ \text{complex-analytic}}}", from=2-3, to=4-3]
	\arrow["{\substack{\cong\\ \text{real-analytic}}}"', from=2-3, to=3-2,<->]
	\arrow["{\substack{\cong\\ \text{real-analytic}}}"', from=3-2, to=4-3,<->]
	\arrow[from=3-1, to=3-2]
    \arrow["{\text{$1:1$ on sets of $\BoC$-points}}", from=3-1, to=1-4,bend left,<->]
    \arrow["{\substack{\text{Riemann--Hilbert}\\ \text{complex analytic}\\ \text{isomorphism}}}", from=1-4, to=5-4,bend left]
	\arrow["{\text{$1:1$ on sets of $\BoC$-points}}", from=5-4, to=3-1,bend left,<->]
	\arrow[from=5-4, to=4-3]
	\arrow[from=1-4, to=2-3]
\end{tikzcd}
\end{equation}
This is a very challenging problem, on which no progress has been made since Simpson formulated closely related questions \cite{simpson1994moduli,simpson1994moduliII} regarding the topological quotients $R^{\sharp}(X)/\GL_r$ (in general non-separated topological spaces). For the simplest of the nontrivial algebraic varieties, smooth projective curves, the NAHT correspondence is already highly nontrivial and gave rise to the celebrated $\mathrm{P}=\mathrm{W}$ conjecture \cite{de2012topology}, now proven thrice \cite{maulik2022p,hausel2022p,maulik2023perverse}. See \cite{felisetti2023p,hoskins2023twoproofs} for an exposition of these ideas. This conjecture relies on the homeomorphism between the Betti and Dolbeault moduli spaces to compare two filtrations one can define on the cohomology vector space: the perverse filtration, defined using the Hitchin map from the Dolbeault moduli space to the Hitchin base (a complete integrable system) and the weight filtration, defined using the mixed Hodge structure on the cohomology of the character variety (following the foundational work of Deligne on mixed Hodge structures, \cite{deligne1974theorie}). First stated for the smooth moduli spaces (coprime rank and degree, and twisted character varieties), it was then extended to the singular moduli spaces (involving the intersection cohomology) \cite{de2018perverse}, and then to the moduli stacks (with the BPS cohomology playing a central role) \cite{davison2023nonabelian}. It is worth noticing here that the $\mathrm{P}=\mathrm{W}$ conjectures for stacks and intersection cohomology in the case of singular moduli spaces are not yet settled, the crucial point being the $\chi$-independence for the BPS cohomology of the character variety.

The purpose of this paper is to answer positively the part of Question \ref{question:comparisonBMstacks} concerning Borel--Moore homology, when $X$ is a smooth projective curve, in a stronger sense: In this situation we have cohomological Hall algebra structures on the Borel--Moore homologies (when we take the direct sum over all ranks $r\in\BoN$), and we prove that through the nonabelian Hodge homeomorphisms, the (relative) cohomological Hall algebras all coincide. The starting point of this paper is the remark that the comparison between the Betti and de Rham moduli spaces is natural via the Riemann--Hilbert correspondence, and the comparison between the Dolbeault and de Rham moduli spaces is natural via the Hodge--Deligne moduli space (parametrising $\lambda$-connections), while the comparison between the Dolbeault and Betti moduli spaces is more indirect  (obtained by composing the two homeomorphisms $\Dol\rightarrow\dR\rightarrow\Betti$). This completes the picture initiated in \cite{davison2022BPS} concerning nonabelian Hodge isomorphisms for stacks on a smooth projective curve. This also completes the results of \emph{loc. cit.} by adding the rich cohomological Hall algebra structures into the game.

\subsection{Main results}

\subsubsection{CoHA structure for the de Rham stack}
Let $C$ be a smooth projective curve. We let $\JH^{\dR}\colon\mathfrak{M}^{\dR}(C)\rightarrow\CM^{\dR}(C)$ be the de Rham moduli stack, de Rham moduli space and the Jordan--H\"older map. The de Rham moduli stack parametrises (flat) connections on $C$ and the points of the good moduli space parametrise semisimple flat connections on $C$. We refer to \S\ref{section:connections} for more information on these objects.

As usual, $\underline{\BoQ}^{\vir}_{\mathfrak{M}^{\dR}(C)}$ denotes the constant mixed Hodge module on the stack $\FM^{\dR}(C)$, with some Tate twists by the virtual dimension (depending on the connected component). Namely, the restriction to the stack of rank $r$ connections is $(\underline{\BoQ}^{\vir}_{\mathfrak{M}^{\dR}(C)})_{|\mathfrak{M}_r^{\dR}(C)}=\underline{\BoQ}_{\mathfrak{M}_r^{\dR}(C)}\otimes\BoL^{(1-g)r^2}$ where $\BoL=\HO^*_{\rmc}(\BoA^1,\BoQ)$ is a pure mixed Hodge structure of weight $2$ placed in cohomological degree $2$ (seen as a complex of mixed Hodge modules, it is pure of weight zero).

\begin{theorem}[$\subset$ Theorem \ref{theorem:CoHAdeRham}]
\label{theorem:CoHAdeRhamintro}
 We have a relative cohomological Hall algebra structure on $\underline{\SA}^{\dR}(C)\coloneqq\JH_*^{\dR}\BD\underline{\BoQ}_{\FM^{\dR}(C)}^{\vir}$. It induces an absolute cohomological Hall algebra structure on $\HO^{\rmBM}_*(\FM^{\dR}(C),\BoQ^{\vir})$ by taking derived global sections.
\end{theorem}
To prove this theorem, we rely on the construction of relative cohomological Hall algebra products in \cite{davison2022BPS} using $3$-term complexes of vector bundles over the product $\FM^{\dR}(C)\times\FM^{\dR}(C)$. The $3$-term complexes allow us to define the virtual pullback from this product of stacks to the stack of short exact sequences. The application of such pullbacks in the study of cohomological Hall algebras was initiated by Kapranov and Vasserot in \cite{kapranov2019cohomological}. These pullbacks (when forgetting the mixed Hodge module strutures) appear to be particular cases of pullbacks by quasi-smooth morphisms as in \cite{khan2019virtual}. The methods of \cite{porta2022two} together with Khan's formalism give an other approach to Theorem \ref{theorem:CoHAdeRhamintro}, as explained to us by Francesco Sala. In \cite{davison2022BPS}, $3$-term complexes of vector bundles are used to define the CoHAs in categories of mixed Hodge modules, which the derived algebraic geometry approach does not seem to provide yet. One particular feature is that for $3$-term complexes, the virtual pullbacks may be constructed without any reference to derived algebraic geometry, in a purely classical way involving refined Euler classes. To carry over this construction, we need to check the technical Assumptions 1--3 of \cite{davison2022BPS}, which are respectively (1) the properness of the map from the stack of short exact sequence to the stack of objects keeping the middle term (i.e. properness of Quot schemes for connections); (2) the resolution property for the stack $\mathfrak{M}^{\dR}(C)$ (i.e. any coherent sheaf is quotient of a vector bundle); which follows from the existence of natural linearisations on the representation spaces; and the morphism from the stack of short exact sequences to the square of the stack of objects forgetting the middle term is globally presented (see \cite[Definition 4.5]{davison2022BPS}; in rough terms, this means that $\mathfrak{Exact}^{\dR}(C)$ is the total space of a complex of vector bundles over $\FM^{\dR}(C)\times\FM^{\dR}(C)$); and (3) the compatibility of global presentations with the stack of 2-step filtrations of objects (crucial in the proof of associativity of relative CoHA products in categories of mixed Hodge modules in \cite{davison2022BPS}).

Following the strategy developped in \cite{davison2022BPS}, we obtain the following structural results regarding the BPS algebra and the full CoHA.

We let $\underline{\BPS}_{\Alg}^{\dR}(C)\coloneqq\CH^0(\underline{\SA}^{\dR}(C))$. This is a mixed Hodge module on $\CM^{\dR}(C)$. It has an induced algebra structure (Theorem \ref{theorem:CoHAdeRham}) and is called \emph{relative BPS (associative) algebra}.

\begin{theorem}[=Theorem \ref{theorem:CoHAdeRham}]
\label{theorem:CoHAderhamstructure}
Let $C$ be a smooth projective curve of genus $g\geq 2$.
 \begin{enumerate}
  \item The relative BPS algebra is isomorphic to the free algebra generated by the intersection complexes of the good moduli space:
  \[
   \underline{\BPS}^{\dR}_{\Alg}(C)\cong\Free_{\boxdot-\Alg}\left(\bigoplus_{r\geq 1}\underline{\IC}(\CM_r^{\dR}(C))\right)
  \]
  \item We have a relative PBW isomorphism:
  \[
   \Sym_{\boxdot}\left(\underline{\BPS}_{\Lie}^{\dR}(C)\otimes\HO^*_{\BoC^*}\right)\rightarrow\underline{\SA}^{\dR}(C)
  \]
  where the relative BPS Lie algebra is defined as $\underline{\BPS}_{\Lie}^{\dR}(C)\coloneqq\Free_{\boxdot-\Lie}\left(\bigoplus_{r\geq 1}\underline{\IC}(\CM_r^{\dR}(C))\right)$ and we write $\HO^*_{\BoC^*}\coloneqq \HO^*_{\BoC^*}(\pt)$ to lighten the notations.
\end{enumerate}
\end{theorem}
One can obtain absolute versions of these isomorphisms by taking derived global section (Corollary \ref{corollary:absolutedR}). We also say a word on the case of genus one curves (Theorem \ref{theorem:genusone}).

\subsubsection{Hodge--Deligne moduli stack}
Let $X$ be a smooth projective variety. To interpolate between the de Rham and Dolbeault moduli spaces, Deligne suggested to consider flat $\lambda$-connections on the smooth projective variety $X$. The moduli space of semistable $\lambda$-connections was then used successfully by Simpson \cite{simpson1996hodge} to provide a construction of the twistor space describing the Hyperk\"ahler structure of the Betti/de Rham/Dolbeault moduli space. A $\lambda$-connection on $X$ is a pair $(\CF,\nabla)$ where $\CF$ is a vector bundle over $X$ and
\[
 \nabla\colon\CF\rightarrow\CF\otimes_{\CO_X}\Omega^1_X
\]
is a morphism of sheaves satisfying the $\lambda$-twisted Leibniz identity
\[
 \nabla(f\cdot s)=\lambda(s\otimes\rmd f)+f\nabla(s)
\]
for $s$ and $f$ local sections of $\CF$ and $\CO_X$ respectively. A $\lambda$-connection is said to be \emph{flat} is $\nabla\circ\nabla=0$ (This condition is empty when $\dim X=1$).
Let
\[
 \begin{tikzcd}
	{\FM^{\Hod}(X)} & {} & {\CM^{\Hod}(X)} \\
	& {\BoA^1}
	\arrow["\pi_{\FM^{\Hod}}"',from=1-1, to=2-2]
	\arrow["\pi_{\CM^{\Hod}}",from=1-3, to=2-2]
	\arrow["{\JH^{\Hod}}", from=1-1, to=1-3]
\end{tikzcd}
\]
be the Hodge--Deligne moduli stack and the Hodge-Deligne moduli space respectively parametrising all and semisimple flat semistable $\bm{\lambda}$-connections on $X$ and the Jordan--H\"older map between them (on $\BoC$-points, it sends a $\bm{\lambda}$-connection to its semisimplification, and we write $\bm{\lambda}$ instead of $\lambda$ when we do not want to specify the parameter). The map to $\BoA^1$ just remembers the parameter $\lambda$ of the $\bm{\lambda}$-connection. There are $\BoC^*$-actions on $\FM^{\Hod}(X)$ and $\CM^{\Hod}(X)$ covering the natural $\BoC^*$-action on $\BoA^1$ and the map $\JH^{\Hod}$ is $\BoC^*$-equivariant.

We have Cartesian squares
\[
 \begin{tikzcd}
	& {} \\
	{\FM^{\Dol}(X)} & {\CM^{\Dol}(X)} & {\{0\}} \\
	{\FM^{\Hod}(X)} & {\CM^{\Hod}(X)} & {\BoA^1}
	\arrow[from=2-3, to=3-3]
	\arrow["\imath_{\CM^{\Dol}}"',from=2-2, to=3-2]
	\arrow[from=3-2, to=3-3]
	\arrow[from=2-2, to=2-3]
	\arrow[from=2-1, to=3-1]
	\arrow[from=3-1, to=3-2]
	\arrow[from=2-1, to=2-2]
	\arrow["\lrcorner"{anchor=center, pos=0.125}, draw=none, from=2-1, to=3-2]
	\arrow["\lrcorner"{anchor=center, pos=0.125}, draw=none, from=2-2, to=3-3]
\end{tikzcd},
\quad\quad
 \begin{tikzcd}
	& {} \\
	{\FM^{\dR}(X)} & {\CM^{\dR}(X)} & {\{1\}} \\
	{\FM^{\Hod}(X)} & {\CM^{\Hod}(X)} & {\BoA^1}
	\arrow[from=2-3, to=3-3]
	\arrow["\imath_{\CM^{\dR}}"',from=2-2, to=3-2]
	\arrow[from=3-2, to=3-3]
	\arrow[from=2-2, to=2-3]
	\arrow[from=2-1, to=3-1]
	\arrow[from=3-1, to=3-2]
	\arrow[from=2-1, to=2-2]
	\arrow["\lrcorner"{anchor=center, pos=0.125}, draw=none, from=2-1, to=3-2]
	\arrow["\lrcorner"{anchor=center, pos=0.125}, draw=none, from=2-2, to=3-3]
\end{tikzcd}.
\]
For $\lambda\in\BoA^1$, we let $\JH_{\lambda}^{\Hod}\colon\FM_{\lambda}^{\Hod}(X)\rightarrow\CM_{\lambda}^{\Hod}(X)$ be the fiber of $\JH^{\dR}$ over $\lambda$. This is a good moduli space for the stack $\FM_{\lambda}^{\Hod}(X)$.

\begin{theorem}[=Theorem \ref{theorem:CoHAHodgeDeligne}]
\label{theorem:CoHAHodge}
 Let $C$ be a smooth projective curve. The complex of mixed Hodge modules $\underline{\SA}^{\Hod}(C)\coloneqq\JH^{\Hod}_*\BD\underline{\BoQ}_{\FM^{\Hod}(C)}^{\vir}$ has a relative cohomological Hall algebra structure for the relative monoidal product on $(\CD^+(\MHM(\CM^{\Hod}(C))),\boxdot_{\BoA^1})$ coming from the morphism $\CM^{\Hod}(C)\rightarrow\BoA^1$ (\S\ref{subsubsection:monoidalstroveralgvar}). It induces an absolute cohomological Hall algebra structure on $\HO^*(\underline{\SA}^{\Hod}(C))=(\pi_{\FM^{\Hod}(C)})_*\BD\underline{\BoQ}_{\FM^{\Hod}(C)}^{\vir}\in(\CD^+(\MHM(\BoA^1)),\stackrel{!}{\otimes}$).
 
 Moreover, we have canonical isomorphisms of algebra objects $\imath_{\CM^{\Dol}}^!\underline{\SA}^{\Hod}(C)\cong\underline{\SA}^{\Dol}(C)$ and $\imath_{\CM^{\dR}}^!\underline{\SA}^{\Hod}(C)\cong\underline{\SA}^{\dR}(C)$ in $(\CD^+(\MHM(\CM^{\Dol}(C))),\boxdot)$ and $(\CD^+(\MHM(\CM^{\dR}(C))),\boxdot)$ respectively.
\end{theorem}

For any $r\geq 1$, the morphism $\CM^{\Hod}_r(C)\rightarrow\BoA^1$ is a trivial topological fibration (it is real-analytically trivialisable, \S\ref{subsubsection:preferredtrivialization}). Therefore, there exists a constructible complex $\IC(\CM^{\Hod}_r(C)/\BoA^1)$ on $\CM_r^{\Hod}(C)$ such that for any $\imath_{\lambda}\colon\{\lambda\}\rightarrow \BoA^1$, $\imath_{\lambda}^!\IC(\CM_r^{\Hod}(C)/\BoA^1)\cong \IC(\CM_{\lambda}^{\Hod}(C))$ (by \cite[Proposition 2.5.1]{de2018compactification}). We actually have $\IC(\CM_r^{\Hod}/\BoA^1)=\IC(\CM_{r}^{\Hod})[1]$. We also define the (Tate twisted) mixed Hodge module $\underline{\IC}(\CM_r^{\Hod}(C)/\BoA^1)\coloneqq(\jmath_r)_{!*}\underline{\BoQ}_{\CM_{r}^{\Hod,\circ}}\otimes\BoL^{(1-g)r^2-2}$ where $\jmath_r\colon\CM_r^{\Hod,\circ}\rightarrow\CM_r^{\Hod}$ is the inclusion of the smooth locus. It is pure of weight zero and satisfies $\imath_{\CM_{\lambda}^{\Hod}}^!\underline{\IC}(\CM_r^{\Hod}(C)/\BoA^1)= \underline{\IC}(\CM_{\lambda}^{\Hod}(C))$ (this is easy to check on the smooth locus). We let
\[
 \begin{matrix}
  \Delta_r&\colon&\CM_1^{\Hod}(C)&\rightarrow&\CM_r^{\Hod}(C)\\
  &&x&\mapsto&x^{\oplus r}
 \end{matrix}
\]
be the small diagonal in the Hodge moduli space. This is a closed immersion.

As usual, we let $\HO^*_{\BoC^*}\coloneqq \HO^*(\pt/\BoC^*)$.

\begin{theorem}[=Theorem \ref{theorem:relativBPSalg}]
\label{theorem:CoHAHodgestructure}
\begin{enumerate}
 \item Let $C$ be a curve of genus $g=1$. We have a canonical isomorphism of complexes of mixed Hodge modules
 \[
 \Sym_{\boxdot_{\BoA^1}}\left(\left(\bigoplus_{r\geq 1}(\Delta_r)_*\underline{\IC}(\CM_1^{\Hod}(C)/\BoA^1)\right)\otimes\HO^*_{\BoC^*}\right)\rightarrow\underline{\SA}^{\Hod}(C)\in\CD^+(\MHM(\CM^{\Hod}(C))).
 \]
 \item  Let $C$ be a curve of genus $g\geq 2$. We have a canonical isomorphism of complexes of mixed Hodge modules
 \[
  \Sym_{\boxdot_{\BoA^1}}\left(\Free_{\boxdot_{\BoA^1}-\Lie}\left(\bigoplus_{r\geq 0}\underline{\IC}(\CM^{\Hod}_r(C)/\BoA^1)\right)\otimes\HO^*_{\BoC^*}\right)\rightarrow\underline{\SA}^{\Hod}(C)\in\CD^+(\MHM(\CM^{\Hod}(C))).
 \]
\end{enumerate}
\end{theorem}
\begin{corollary}[$\subset$ Proposition \ref{proposition:trivializationCoHA}]
 The complex of constructible sheaves $(\pi_{\FM^{\Hod}})_*\BD\BoQ_{\FM^{\Hod}(C)}^{\vir}\in\CD^+_{\rmc}(\BoA^1)$ is locally constant.
\end{corollary}

\begin{corollary}
Let $C$ be a smooth projective curve of genus $\geq 2$. The relative BPS algebra and relative BPS Lie algebra of the Hodge moduli stack can be upgraded to (shifted) mixed Hodge modules. Namely, we set
\[
 \underline{\BPS}_{\Alg}^{\Hod}(C)\coloneqq \Free_{\boxdot_{\BoA^1}-\Alg}\left(\bigoplus_{r\geq 1}\underline{\IC}(\CM_r^{\Hod}(C)/\BoA^1)\right), \quad \underline{\BPS}_{\Lie}^{\Hod}(C)\coloneqq \Free_{\boxdot_{\BoA^1}-\Lie}\left(\bigoplus_{r\geq 1}\underline{\IC}(\CM_r^{\Hod}(C)/\BoA^1)\right)
\]
Both belong to $\MHM(\CM^{\Hod}(C))[2]$ (complexes of mixed Hodge modules concentrated in cohomological degree $-2$).
\end{corollary}
For curves of genus one, given that the corresponding category is totally isotropic (\S\ref{subsection:Eulerformconnections}), we may also upgrade the relative BPS algebra and BPS Lie algebra to shifted mixed Hodge modules by setting
\[
  \underline{\BPS}_{\Alg}^{\Hod}(C)\coloneqq \Sym_{\boxdot_{\BoA^1}}\left(\bigoplus_{r\geq 1}(\Delta_r)_*\underline{\IC}(\CM_1^{\Hod}(C)/\BoA^1)\right), \quad \underline{\BPS}_{\Lie}^{\Hod}(C)\coloneqq \bigoplus_{r\geq 1}(\Delta_r)_*\underline{\IC}(\CM_1^{\Hod}(C)/\BoA^1),
\]
both being commutative. We refer to \cite[\S3.3]{davison2023bps} for an extended discussion on Lie algebras in monoidal categories of perverse sheaves/MHM.

\subsubsection{Compatibility of the CoHA structures}

We let $\Psi\colon \CM^{\dR}(C)\rightarrow\CM^{\Dol}(C)$ be the homeomorphism given by nonabelian Hodge theory and $\Phi\colon \CM^{\dR}(C)\rightarrow \CM^{\Betti}(C)$ be the homeomorphism provided by the Riemann--Hilbert correspondence.

\begin{theorem}[=Corollary \ref{corollary:reldRdolcoincide}+Theorem \ref{theorem:reldRBetticoincide}]
\label{theorem:stackyNAHiso}
We have canonical isomorphisms of algebra objects in $\CD^+_{\rmc}(\CM^{\Dol}(C))$ and $\CD^+_{\rmc}(\CM^{\Betti}(C))$:
\[
 \Phi_*\JH^{\dR}_*\BD\BoQ_{\FM^{\dR}(C)}^{\vir}\cong \JH^{\Betti}_*\BD\BoQ_{\FM^{\Betti}(C)}^{\vir}, \quad \Psi_*\JH^{\dR}_*\BD\BoQ_{\FM^{\dR}(C)}^{\vir}\cong \JH^{\Dol}_*\BD\BoQ_{\FM^{\Dol}(C)}^{\vir},
\]
canonical isomorphisms of algebra objects in $\Perv(\CM^{\Dol}(C))$ and $\Perv(\CM^{\Betti}(C))$:
\[
 \Phi_*\BPS_{\Alg}^{\dR}(C)\cong\BPS^{\Betti}(C), \quad \Psi_*\BPS_{\Alg}^{\dR}(C)\cong\BPS^{\Dol}(C),
\]
and canonical isomorphisms of Lie algebra objects in $\Perv(\CM^{\Dol}(C))$ and $\Perv(\CM^{\Betti}(C))$
\[
  \Phi_*\BPS_{\Lie}^{\dR}(C)\cong\BPS^{\Betti}_{\Lie}(C), \quad \Psi_*\BPS_{\Lie}^{\dR}(C)\cong\BPS^{\Dol}_{\Lie}(C).
\]
\end{theorem}
This theorem immediately translates to the absolute level, answering positively \cite[Conjecture 1.5]{sala2020comological}.
\begin{corollary}
\label{corollary:NAHTabsintro}
 We have canonical isomorphisms of algebras
\[
 \HO^*\!\!\SA^{\dR}(C)\cong  \HO^*\!\!\SA^{\Betti}(C)\cong \HO^*\!\!\SA^{\Dol}(C),
\]
\[
 \rmBPS_{\Alg}^{\dR}(C)\cong  \rmBPS_{\Alg}^{\Betti}(C)\cong \rmBPS_{\Alg}^{\Dol}(C),
\]
and a canonical isomorphism of Lie algebras
\[
 \rmBPS_{\Lie}^{\dR}(C)\cong  \rmBPS_{\Lie}^{\Betti}(C)\cong \rmBPS_{\Lie}^{\Dol}(C).
\]
\end{corollary}
Note that both the BPS algebra and the BPS Lie algebra are free (Lie) algebras generated by the intersection cohomology of the moduli spaces (when $C$ has genus $\geq 2$) and so the second and third series of isomorphisms in Corollary \ref{corollary:NAHTabsintro} are easy to see. The first one requires the strong machineries provided by the derived Riemann--Hilbert correspondence and the Hodge--Deligne moduli space of $\bm{\lambda}$-connections, together with the convenient tool of relative perverse $t$-structures.

Corollary \ref{corollary:NAHTabsintro} therefore shows that the cohomological nonabelian Hodge isomorphisms can be lifted from the moduli spaces to the stacks, and the lifts are compatible with the cohomological Hall algebra structures. Moreover, we have nonabelian Hodge isomorphisms for the BPS cohomologies  (algebra and Lie algebra versions) of the respective moduli stacks.

\begin{remark}
 The isomorphism between the Betti and Dolbeault cohomological Hall algebra cannot be lifted in the category of mixed Hodge structures as the Betti CoHA is not pure while the Dolbeault CoHA is. This is the purpose of the $\mathrm{P}=\mathrm{W}$ conjecture to find which filtration on the Dolbeault side corresponds to the weight filtration on the Betti side.
\end{remark}

\subsubsection{Affinized BPS Lie algebras}
We can also compare the \emph{affinized} BPS Lie algebras for the Betti, Dolbeault and de Rham moduli spaces. The definition of this affinization relies on work in progress of Davison and Kinjo, in which a coproduct on the cohomological Hall algebras of $2$-Calabi--Yau categories $\CA$ is constructed. This coproduct is first defined on the CoHA of the $3$-Calabi--Yau completion of the category considered (vanishing cycle CoHA) and then, by dimensional reduction, a coproduct for the CoHA of $\CA$ is obtained. This coproduct
\[
 \Delta\colon\HO^*\SA_{\CA}\otimes\HO^*\SA_{\CA}\rightarrow\HO^*\SA_{\CA}
\]
is cocommutative and so the space of primitive elements
\[
 \hat{\Fn}^{+,\rmBPS}_{\CA}\coloneqq \{x\in\HO^*\SA_{\CA}\mid\Delta(x)=x\otimes 1+1\otimes x\}
\]
has an induced Lie algebra structure, i.e. is closed under the Lie bracket induced by the product on $\HO^*\SA_{\CA}$.

This sub-Lie algebra of the cohomological Hall algebra is called \emph{affinized BPS Lie algebra}. It is an object of deep interest that is understood fully in only one case (the $2$d BPS algebra associated with the Jordan quiver \cite{davison2022affine}, or equivalently, the critical cohomological Hall algebra of the tripled Jordan quiver with its canonical potential), and it shows a highly nontrivial behaviour, related to the $W$-algebra $W_{1+\infty}$. In this paper, rather than describing $\hat{\Fn}_{\CA}^{+,\rmBPS}$ when $\CA$ is the category of semistable Higgs sheaves, local systems or connections on a smooth projective curve, we compare these three Lie algebras and show that they are isomorphic (Corollary \ref{corollary:actiondetlbabsolute}). We let $\Fn_{\CA}^{+,\rmBPS}\coloneqq \HO^*(\BPS_{\CA,\Lie})$ be the (absolute) BPS Lie algebra for the category $\CA$ \cite{davison2023bps}. For the sake of the presentation, we assume that for the categories $\CA$ considered, the moduli stack of objects of the $3$-Calabi--Yau completion $\tilde{\CA}$ is a global critical locus and that a critical cohomological Hall algebra can be defined. We assume that there is a dimensional reduction isomorphism connecting the critical CoHA of $\tilde{\CA}$ and the $2$D CoHA of $\CA$ (i.e. we assume that the situation is as good as in the quiver case, preprojective algebra and quiver with potential, so that all the results of \cite{davison2020cohomological} and \cite{davison2020bps} can be adapted). Our results Theorem \ref{theorem:actiondetlinebundle} and Corollary \ref{corollary:actiondetlbabsolute} are formulated independently of this assumption.

By the support property for the BPS sheaf for $3$-Calabi--Yau completions (as in the quiver case, \cite[\S3.10]{davison2022affine}, since we assume our situations are as good as the quiver situation), the subspace
\[
 \Fn_{\CA}^{+,\rmBPS}\otimes\HO^*_{\BoC^*}\subset \HO^*\SA_{\CA}
\]
is contained in $\hat{\Fn}^{+,\rmBPS}_{\CA}$. By the PBW theorem for $2$-Calabi--Yau categories \cite[Corollary 1.8]{davison2023bps}, and the Milnor--Moore theorem, the graded dimensions of $\Fn_{\CA}^{+,\rmBPS}\otimes\HO^*_{\BoC^*}$ and $\hat{\Fn}^{+,\rmBPS}_{\CA}$ must coincide, leading to an identification of vector spaces
\[
 \hat{\Fn}^{+,\rmBPS}_{\CA}=\Fn_{\CA}^{+,\rmBPS}\otimes\HO^*_{\BoC^*}.
\]

Since the definition of the Davison--Kinjo coproduct does not appear yet in the literature, and the critical CoHA for the $3$-CY completions of the categories involved are not fully worked out yet, we formulate the following theorem independently of this coproduct and without considering any $3$-Calabi--Yau completions.

\begin{theorem}[=Theorem \ref{theorem:actionChernclassdet}]
\label{theorem:actiondetlinebundle}
 Let $C$ be a smooth projective curve. The action of the first Chern classes of the determinant line bundles on $\SA^{\dR}(C)$, $\SA^{\Dol}(C)$ and $\SA^{\Betti}(C)$ coincide with each other through the stacky nonabelian Hodge isomorphisms (Theorem \ref{theorem:stackyNAHiso}).
\end{theorem}
For $\sharp\in\{\Betti,\Dol,\dR\}$, we let $\Fn^{\sharp,\rmBPS,+}(C)\coloneqq\HO^*(\BPS_{\Lie}^{\sharp}(C))$ be the $\sharp$-BPS Lie algebra.
\begin{corollary}
\label{corollary:actiondetlbabsolute}
 The subspaces $\Fn^{\dR,\rmBPS,+}(C)\otimes\HO^*_{\BoC^*}\subset\HO^*\SA^{\dR}(C)$, $\Fn^{\Betti,\rmBPS,+}(C)\otimes\HO^*_{\BoC^*}\subset\HO^*\SA^{\Betti}(C)$ and $\Fn^{\Dol,\rmBPS,+}(C)\otimes\HO^*_{\BoC^*}\subset\HO^*\SA^{\Dol}(C)$ coincide with each other under the isomorphisms given by the stacky nonabelian Hodge correspondence.
\end{corollary}
The existence of the coproduct would imply that the subspaces $\Fn^{\dR,\rmBPS,+}(C)\otimes\HO^*_{\BoC^*}\subset\HO^*\SA^{\dR}(C)$, $\Fn^{\Betti,\rmBPS,+}(C)\otimes\HO^*_{\BoC^*}\subset\HO^*\SA^{\Betti}(C)$ and $\Fn^{\Dol,\rmBPS,+}(C)\otimes\HO^*_{\BoC^*}\subset\HO^*\SA^{\Dol}(C)$ are closed under their respective Lie brackets coming from the ambient CoHA. Using the NAHT isomorphisms (Corollary \ref{corollary:NAHTabsintro}), these \emph{affinized} BPS Lie algebra would then be isomorphic. The full description of the affinized Lie algebras, or equivalently of the full CoHA, by generators and relations is not yet known but constitutes a challenging and crucial question in the theory of cohomological Hall algebras.

\subsubsection{Non-zero degree versions}
\label{subsubsection:parabolicversions}
We presented our results for degree zero Higgs bundles and local systems on the compact Riemann surface. In this section, we sketch how one can adapt the main results of this paper to Higgs bundles of non-zero degree or parabolic Higgs bundles and connections.

The strategies developped in this paper work in the parabolic case, but the constructions appear to be more subtle than in the compact case. For example, it is harder to construct the Hodge moduli space (see \cite{simpson2022twistor}).

It is also harder and requires extra care to construct moduli stack/spaces of connections on noncompact curves: this appears in \cite{nitsure1993moduli}, see also the literature on wild nonabelian Hodge theory (\cite{biquard1991stable} and the subsequent works).

Let $C$ be a smooth projective curve and $p\in C$ a fixed point. We let $\theta=d'/r'\in\BoQ$. We let $\zeta\coloneqq \exp\left(\frac{2\pi i}{r'}\right)$ be a primitive $r'$th root of unity and $\zeta_{\theta}\coloneqq \zeta^{d'}$. As in \cite{davison2022BPS}, we define $\BoC[\pi_1(C\setminus\{p\}),\lambda]\coloneqq\BoC\langle x_i^{\pm 1},y_i^{\pm 1}\colon 1\leq i\leq g\rangle/\langle\prod_{i=1}^gx_iy_ix_i^{-1}y_i^{-1}=\lambda\rangle$. We let $\JH^{\Betti}\colon\FM_{\theta}^{\Betti}\rightarrow\CM_{\theta}^{\Betti}$ be the stack of representations of the twisted fundamental group algebra $\BoC[\pi_1(C\setminus\{p\}),\zeta_{\theta}]$ and its good moduli space.

We let $\FM_{\theta}^{\dR}$ be the stack of connections on $C\setminus\{p\}$, with first order pole at $p$ and monodromy around $p$ given by multiplication by $\zeta_{\theta}$ \cite{nitsure1993moduli}. It has a good moduli space $\JH^{\dR}\colon\FM_{\theta}^{\dR}\rightarrow\CM_{\theta}^{\dR}$. The Riemann--Hilbert correspondence identifies the stacks and moduli spaces $\JH^{\sharp}\colon \FM_{\theta}^{\sharp}\rightarrow\CM_{\theta}^{\sharp}$ for $\sharp\in\{\Betti,\dR\}$. It may be constructed using the ring of logarithmic differential operators $\Lambda^{\dR,\log(p)}$ on the open curve $C\setminus\{p\}$ as defined in \cite[p.87]{simpson1994moduli}.

On the Dolbeault side, we consider the stack of slope $\theta$ semistable Higgs bundles on $C$, $\FM_{\theta}^{\Dol}$, together with its good moduli space $\JH^{\Dol}\colon \FM_{\theta}^{\Dol}\rightarrow\CM_{\theta}^{\Dol}$.

Last, we have the Hodge--Deligne moduli stack and space over $\BoA^1$, $\JH^{\Dol}\colon \FM_{\theta}^{\Hod}\rightarrow\CM_{\theta}^{\Hod}$ interpolating between the de Rham and Dolbeault moduli spaces/stacks \cite[\S2.4]{migliorini2017recent}. It can be constructed using the deformation to the associated graded of the ring of logarithmic differential operators $\Lambda^{\dR,\log(p)}$, as in \cite[p.86]{simpson1994moduli}. The strategy developed in this paper yields the analogue of Theorems \ref{theorem:CoHAdeRhamintro}, \ref{theorem:CoHAderhamstructure}, \ref{theorem:CoHAHodge}, \ref{theorem:CoHAHodgestructure}, \ref{theorem:stackyNAHiso}, \ref{theorem:actiondetlinebundle} and of their corollaries for a slope $\theta\in\BoQ$ possibly nonzero.

\subsubsection{$\chi$-independence phenomena}
We finish the outline of our main results by stating some $\chi$-independence results involving the cohomological Hall algebra structures. We state more conjectures in the last section of this paper \S\ref{section:chiindependenceph}. Let $\mu=d/r$ and $\mu'=r/d'$ where $r,d,r',d'\in\BoZ$ and $\gcd(r,d)=\gcd(r,d')=1$. Let $\zeta_r$ be a primitive $r$th root of unity. There is an automorphism of $\BoC$ sending $\zeta_r^d$ to $\zeta_r^{d'}$.

\begin{theorem}
The Galois conjugation induces an isomorphism $\gamma\colon\HO_*^{\rmBM}(\FM_{r,d}^{\Betti})\rightarrow\HO_*^{\rmBM}(\FM_{r,d'}^{\Betti})$ respecting the weight filtrations between the Borel--Moore homologies of different Betti stacks. If $\mu=r/d, \mu'=r/d'\in\BoQ$ with $\gcd(r,d)=\gcd(r,d')=1$. Then, we have an isomorphism of algebras induced by Galois conjugation
\[
 \HO^*(\SA_{\mu}^{\Betti})\cong\HO^*(\SA_{\mu'}^{\Betti}).
\]
\end{theorem}

\begin{corollary}
If $\mu=r/d, \mu'=r/d'\in\BoQ$ with $\gcd(r,d)=\gcd(r,d')=1$. Then, we have an isomorphism of algebras
\[
 \HO^*(\SA_{\mu}^{\Dol})\cong\HO^*(\SA_{\mu'}^{\Dol}).
\]
\end{corollary}

\subsection{More questions}
\subsubsection{Stacky nonabelian Hodge isomorphisms for higher dimensional varieties}
Our methods to prove nonabelian Hodge isomorphisms for stacks are very specific to smooth projective curves, which is the setting allowing us to take advantage of the cohomological Hall algebra structures. It is not known how to defined such structures for Higgs sheaves, local systems or connections on higher dimensional varieties. It is a nevertheless extremely motivating question, in view of the positive results obtained in the case of curves, to try to see to what extent one can compare the Borel--Moore homologies of the Betti, Dolbeault and de Rham stacks for higher dimensional algebraic varieties. At the level of moduli spaces, and in connection with the $\mathrm{P}=\mathrm{W}$ conjecture, this question was formulated in \cite{de2018perverse}. Such questions were also raised by Simpson \cite{simpson1994moduliII,simpson1996hodge}.

\subsubsection{Stacky nonabelian Hodge isomorphisms in cohomology}
The nonabelian Hodge isomorphisms for stacks we provide between the Betti, Dolbeault and de Rham moduli stacks concern the \emph{Borel--Moore homologies} of these stacks. This is justified by the fact that it is understood since at least the book \cite{chriss1997representation} that Borel--Moore homology plays in representation theory a more prominent role than cohomology. The stacks we consider in this paper are neither smooth nor proper, and so there is no comparison between Borel--Moore homology and cohomology. This seems therefore a legitimate question to try to compare the \emph{cohomologies} of these three stacks (for curves but also for higher dimensional smooth projective varieties).

\subsubsection{Categorification of stacky NAHT}
A reasonable categorification of the stacky NAHT should provide a way to compare some categories of coherent sheaves on the Dolbeault, Betti and de Rham stacks. This question has been studied by Porta--Sala in \cite{porta2022two}, building on \cite{porta2017derived} and the categorified Hall algebras defined there. They explain that one should consider the derived moduli stacks and the corresponding categories of perfect complexes, and how the derived Riemann--Hilbert correspondence induces an equivalence of categorified Hall algebras \cite[Theorem 7.19]{porta2022two}. The comparison between the categorified Dolbeault and de Rham Hall algebras involves the Hodge moduli space; the former is expected to be the associated graded of the latter with respect to a filtration arising from the stack of $\lambda$-connections \cite[Conjecture 6.7]{porta2022two}.

For cohomological Hall algebras as considered in the present paper, we are able to compare the objects without taking any associated graded. This provides strong support for \cite[Conjecture 6.7]{porta2022two} but also an improvement of this conjecture in the world of CoHAs.

\subsection{Cohomological Hall algebras in families}
One ingredient in this paper is formed by cohomological Hall algebras in families, which we combined with the powerful tool of relative perverse $t$-structures \cite{hansen2021relative}. The author together with collaborators are developping this formalism and use it to draw several other consequences, among which
\begin{enumerate}
 \item The positivity conjecture of Schiffmann concerning the positivity of Kac polynomials of smooth projective curves, counting absolutely indecomposable vector bundles \cite{schiffmann2016indecomposable}. This conjecture may be formulated as follows.
 \begin{conjecture}
  Let $(g,r,d)\in\BoN\times\BoZ_{\geq 1}\times\BoZ$ be a triple of a genus $g$, a rank $r$ and a degree $d$. Then there exists a representation of $\GSp(2g,\BoC)$ whose character is given by $A_{g,r,d}(z_1,\hdots,z_{2g})$.
 \end{conjecture}
This is a strict strengthening of \cite[Corollary 1.5.]{schiffmann2016indecomposable}. The representation one has to consider is given by the relative BPS cohomology of the Dolbeault moduli space over the moduli stack of smooth genus $g$ curves.
 \item The $\lambda$-independence conjecture of the BPS sheaf for deformed preprojective algebras.
 \begin{conjecture}
  Let $Q=(Q_0,Q_1)$ be a quiver, $\lambda\in\BoC^{Q_0}$ and $\Pi_Q^{\lambda}$ the $\lambda$-deformed preprojective algebra. We let $\BPS_{\Pi_Q^{\lambda},\Lie}$ be the $BPS$ sheaf for $\Pi_Q^{\lambda}$ (defined as in \cite{davison2022BPS} or via deformed dimensional reduction \cite{davison2022deformed}). Then, for any $\lambda,\mu\in\BoC^{Q_0}$ and $\dd\in\BoN^{Q_0}$ such that $\lambda\cdot\dd=\mu\cdot\dd=0$, we have an isomorphism $\BPS_{\Pi_Q^{\lambda},\dd}\cong\BPS_{\Pi_Q^{\mu},\dd}$. In particular, for any $\lambda\in\BoC^{Q_0}$, and $\dd\in\BoN^{Q_0}$ such that $\lambda\cdot\dd=0$,
  \[
   \ch\HO^*\BPS_{\Pi_Q^{\lambda},\dd}=A_{Q,\dd}(t^{-2}).
  \]
  Moreover, the BPS cohomology gives a local system on the affine space $\{\lambda\in\BoC^{Q_0}\mid\lambda\cdot\dd=0\}$.
 \end{conjecture}
 This is a generalisation of \cite{crawley2004absolutely}, where deformed preprojective algebras were used to prove Kac's positivity conjecture in the indivisible case.
 \item $\chi$-independence in families (for the BPS cohomology of the moduli space of Higgs sheaves or pure one-dimensional sheaves on K3 or Abelian surfaces). The $\chi$-independence for Higgs sheaves was proven in \cite{de2021cohomology,kinjo2021cohomological}. For sheaves on K3 or Abelian surfaces, no ``3d'' definition of the BPS Lie algebra sheaf is available yet but there is a purely 2d definition one can provide using \cite{davison2022BPS} and \cite{davison2023bps}. The $\chi$-independence for a given K3 of Abelian surface is not yet known, although expected. This would have to be considered first, before the version in families. The $\chi$-independence conjecture in families for Higgs bundles can be written as follows.
 \begin{conjecture}
  Let $f\colon\CC\rightarrow S$ be a family of smooth projective curves of genus $g$ over a smooth connected base $S$. We let $\FM_{r,d}^{\Dol}(\CC)$ be the moduli stack of semistable rank $r$ and degree $d$ Higgs bundles over $\CC$, $\CM_{r,d}^{\Dol}(\CC)$ the good moduli space and $B_S$ the Hitchin base. We have morphisms
  \[
   \FM_{r,d}^{\Dol}(\CC)\xrightarrow{\JH}\CM_{r,d}^{\Dol}(\CC)\xrightarrow{h}B_S\xrightarrow{\pr} S.
  \]
Then, there exists a BPS sheaf in families $\BPS_{r,d}(\CC)\in\MHM(\CM_{r,d}^{\Dol}(\CC))[\dim S]$ such that for any $s\in S$, if $\imath_{\CM,s}\colon\CM_{r,d}^{\Dol}(\CC_s)\rightarrow\CM_{r,d}^{\Dol}(\CC)$ is the inclusion of the fiber of $\pr\circ h$ over $s$, then $\imath_{\CM,s}^!\BPS_{r,d}^{\Dol}(\CC)\cong\BPS_{r,d}^{\Dol}(\CC_s)$. Moreover, the complex $h_*\BPS_{r,d}^{\Dol}(\CC)$ has locally constant cohomology sheaves and is independent of $d$.
 \end{conjecture}

\end{enumerate}

\subsection*{Acknowledgements}
I would like to thank Philip Boalch, Ben Davison, Tasuki Kinjo and Marco Robalo for useful email exchanges and generous discussions. I am grateful to Francesco Sala for pointing out connections between \cite{porta2022two} and \cite{porta2022non} and various definitions of virtual pullbacks. I would like to thank the organisers of the workshop on ``Moduli spaces, virtual invariants and shifted symplectic structures'' in Seoul in June 2023 where I learned a lot, met experts in derived algebraic geometry and made progress on the writing of this paper.

\subsection{Conventions and notations}
\begin{enumerate}[$\bullet$]
 \item If $X$ is an algebraic variety with an action of an algebraic group $G$, we denote by $X/G$ the quotient stack. In the literature, the consacred notation is sometimes $[X/G]$ and we drop the brackets, to lighten the notation.
 \item (Complexes of) Mixed Hodge modules (MHM) are underlined $\underline{\CF}$; while (complexes of) constructible sheaves are not: $\CF$.
 \item Most of the time, we will forget to write the analytification functor $(-)^{\an}$ and denote by $X$ at the same time the complex algebraic variety and the complex analytic variety. The context should help to choose between the two when it matters.
 \item The $\BoC^*$-equivariant cohomology of the point is denoted by $\HO^*_{\BoC^*}$.
 \item To lighten the notation, we will often forget the underlying smooth projective curve in the notation: for example, $\FM_r^{\Dol}=\FM_r^{\Dol}(C)$.
 \item In this paper, the word \emph{algebra} is used to refer to associative algebras. Other types of algebras (e.g. Lie algebras) are called by their full name.
 \item If $X$ is a complex algebraic variety, we denote by $\IC(X)\in\Perv(X)$ the intersection cohomology complex on $X$. We denote by $\ICA(X)\coloneqq\HO^*(\IC(X))$ its derived global sections. This may vary from conventions used in the literature: if $\dim X=n$, $\ICA(X)$ is concentrated in degrees $[-n,n]$.
 \item If $X$ is even-dimensional, $\underline{\IC}(X)$ denotes the MHM upgrade of the intersection complex.
\end{enumerate}

\section{Perverse sheaves and mixed Hodge modules}

\subsection{Perverse sheaves}
Let $X$ be a complex algebraic variety. We let $\CD(X,\BoQ)$ be the category of complexes of sheaves of $\BoQ$-vector spaces on $X$ with quasi-isomorphisms inverted. This is the derived category of the category of sheaves of $\BoQ$-vector spaces $\Sh(X,\BoQ)$. A sheaf $\SF\in\Sh(X,\BoQ)$ is called \emph{constructible} if there exists an algebraic stratification $X=\bigsqcup_{i\in I}X_i$ by locally closed smooth subvarieties such that the restrictions $\SF_{X_i}$ are locally constant sheaves with finite dimensional fibers. Constructible sheaves form a full Abelian subcategory $\Sh_{\rmc}(X,\BoQ)\subset\Sh(X,\BoQ)$. A complex of sheaves $\SF\in\CD(X,\BoQ)$ is called constructible if its cohomology sheaves are constructible. We let $\CD_{\rmc}(X,\BoQ)$ be the triangulated category of constructible complexes. By \cite{beilinson2018faisceaux}, it has the \emph{perverse $t$-structure} $(\pDc{\leq 0}(X,\BoQ),\pDc{\geq 0}(X,\BoQ))$. Its heart is denoted $\Perv(X,\BoQ)\coloneqq\pDc{\leq 0}(X,\BoQ)\cap\pDc{\geq 0}(X,\BoQ)$ and its objects are called perverse sheaves. We refer to \cite{beilinson2018faisceaux} for more properties of this very well-behaved finite length Abelian category. Since the ring of coefficients will be $\BoQ$ throughout this paper, we sometimes drop it from the notation: $\CD_{\rmc}(X), \Perv(X)$. Moreover, we often consider the bounded or semi-bounded versions $\CD^{\rmb}_{\rmc}(X)$ and $\CD^+_{\rmc}(X)$.

\subsection{$\ell$-adic constructible sheaves and perverse sheaves}
\label{subsection:elladic}
In the last section of this paper, \S\ref{section:chiindependenceph}, for an algebraic variety $X$ defined over an algebraically closed field $k$, we will consider at some places $\ell$-adic constructible derived categories $\CD_{\rmc}(X,\overline{\BoQ}_{\ell})$, $\ell$-adic Borel--Moore homology, and $\ell$-adic intersection cohomology, where $\ell$ is a prime number invertible in $k$. For us, $k$ will be a field of characteristic $0$. We refer to \cite{beilinson2018faisceaux,bhatt2015proetale} for definitions and properties. We will also need the versions for Artin stacks over $k$. We refer to \cite{behrend1993lefschetz,behrend2003derived,laszlo2008six,laszlo2008sixb,laszlo2009perverse} for background and extension of the definitions to this context. Since all the stacks appearing in this paper have an explicit presentation as quotient stacks, the more elementary formalism of equivariant derived categories of \cite{bernstein2006equivariant} is sufficient for us. The general formalism is convenient to avoid referring to explicit presentations of quotient stacks.

\subsection{Relative perverse $t$-structures}
\label{subsection:relativeptstructure}
Let $f\colon X\rightarrow S$ be a morphism between complex algebraic varieties. Relative perverse $t$-structures were introduced in \cite{hansen2021relative} to interpolate the perverse $t$-structures on the constructible derived categories of the fibers of the morphism $f$ to a global $t$-structure on $X$, but earlier work around the Langlands program and the affine or Beilinson--Drinfeld Grassmannians hinted at the existence of such a $t$-structure, see \emph{loc. cit.}.  For us, relative perverse $t$-structures will prove to be the right tool to define and study the BPS algebras for the Hodge moduli space (\S\ref{section:CoHAHodge}). One of the main results of \cite{hansen2021relative} reads as follows.
\begin{theorem}[{\cite[Theorem~1.1 + Remark~1.3]{hansen2021relative}}]
\label{theorem:relativeperversetstructure}
 There exists a $t$-structure $(\pSD{S}{*\leq 0}(X),\pSD{S}{*\geq 0}(X))$ on $\CD_{\rmc}(X)$, called the \emph{$*$-relative perverse $t$-structure}, such that an object $\SF\in\CD_{\rmc}(X)$ is in $\pSD{S}{*\leq 0}(X)$ (resp. $\pSD{S}{*\geq 0}(X)$) if and only if for any geometric point $s\in S$, and $X_s$ the fiber of $f$ over $s$:
\[
 \begin{tikzcd}
	{X_s} & X \\
	{\{s\}} & S
	\arrow["{\imath_{X_s}}", from=1-1, to=1-2]
	\arrow["f", from=1-2, to=2-2]
	\arrow["{f_s}"', from=1-1, to=2-1]
	\arrow[from=2-1, to=2-2]
	\arrow["\lrcorner"{anchor=center, pos=0.125}, draw=none, from=1-1, to=2-2]
\end{tikzcd}
\]
one has $\imath_{X_s}^*\SF\in\CD^{\leq 0}_{\rmc}(X_s)$ (resp. $\imath_{X_s}^*\SF\in\CD^{\geq 0}_{\rmc}(X_s)$).
\end{theorem}
Although the authors of \cite{hansen2021relative} work with $\infty$-categories, $t$-structures are defined on the homotopy category, see \cite[Definition~1.2.1.4]{lurie2017higher}. By considering the (semi-)bounded constructible derived categories, we obtain $t$-structures on $\CD^{\rmb}_{\rmc}(X)$ and $\CD^+_{\rmc}(X)$. The truncation functors associated with this $t$-structure are denoted $\pStau{S}{*\leq i}$ and $\pStau{S}{*\geq i}$ for $i\in\BoZ$. The cohomology functors are $\pSH{S}{*i}$, $i\in\BoZ$. The heart of the $*$-relative perverse $t$-structure is denoted by $\Perv^*(X/S)$.

One can dualise Theorem \ref{theorem:relativeperversetstructure} using the Verdier duality functor and consider instead of $*$-pullbacks the $!$-pullbacks.
\begin{corollary}
\label{corollary:shriekreltstructure}
 There exists a $t$-structure $(\pSD{S}{!\leq 0}(X),\pSD{S}{!\geq 0}(X))$ on $\CD_{\rmc}(X)$ called the \emph{$!$-relative perverse $t$-structure} such that an object $\SF\in\CD_{\rmc}(X)$ is in $\pSD{S}{!\leq 0}(X)$ (resp. $\pSD{S}{!\geq 0}(X)$) if and only if for any geometric point $s\in S$, and $X_s$ the fiber of $f$ over $s$, with same Cartesian diagram as in Theorem \ref{theorem:relativeperversetstructure}, one has $\imath_{X_s}^!\SF\in\CD^{\leq 0}_{\rmc}(X_s)$ (resp. $\imath_{X_s}^!\SF\in\CD^{\geq 0}_{\rmc}(X_s)$).
\end{corollary}
We denote by $\pStau{S}{!\leq i}$, $\pStau{S}{!\geq i}$, $i\in\BoZ$ the truncation functors. The cohomology functors are $\pSH{S}{!i}$, $i\in\BoZ$. The heart of the $!$-relative perverse $t$-structure is denoted by $\Perv^!(X/S)$.

In this paper, we will use exclusively the $!$-relative perverse $t$-structure. Therefore, we may drop the symbol ``!'' from the notation and write $\pStau{S}{\leq i}$, $\pStau{S}{\geq i},\pSH{S}{i}$ and $\Perv(X/S)$.

An example of relative perverse sheaf playing an important role in this paper (for the Hodge--Deligne moduli space) is given by the following lemma.
\begin{lemma}
\label{lemma:relativeICps}
 If $f\colon X\rightarrow S$ is a morphism between complex algebraic varieties with $S$ smooth and $f$ topologically locally trivial, then $\IC(X)[\dim S]\in\Perv^!(X/S)$.
\end{lemma}
\begin{proof}
 By shrinking $S$, we may assume that $f$ is topologically locally trivial and work on the topological space $f^{-1}(s)\times S=X_s\times S$ for $s\in S$ and perverse sheaves for stratified topological spaces as in \cite{beilinson2018faisceaux}. Then, $\IC(X_s\times U)\cong \IC(X_s)\boxtimes\BoQ_S[\dim S]$. If $\imath_s\colon \{s\}\rightarrow S$ is the inclusion of $s$, then $\imath_s^!\BoQ_S\cong \BoQ_{\{s\}}[-2\dim S]$. The lemma follows.
\end{proof}

The following lemma gives some compatibilities between the relative and absolute perverse $t$-structures for certain constructible complexes when the morphism at hand is a topologically locally trivial fibration.
\begin{lemma}
\label{lemma:propertyrelative}
 Let $f\colon X\rightarrow Y$ be a morphism of complex algebraic varieties, such that $Y$ is smooth and $f$ is a topologically trivial fibration. For $y\in Y$, we let $X_y=f^{-1}(y)$ be the fiber of $f$ over $y$ and $g\colon X\rightarrow Y\times X_y$ be a homeomorphism. We let $\jmath\colon U\rightarrow X$ be a locally closed subset such that $g\circ\jmath$ is identified with $(\id_Y\times\jmath_{U_y})$ through $g$, where $\jmath_{U_y}\colon U_y\rightarrow X_y$ is the open subset defined by the Cartesian diagram
 \[
  \begin{tikzcd}
	{U_y} & U \\
	y & Y
	\arrow["{f_{U}}", from=1-2, to=2-2]
	\arrow[from=1-1, to=2-1]
	\arrow[from=2-1, to=2-2]
	\arrow[from=1-1, to=1-2]
	\arrow["\lrcorner"{anchor=center, pos=0.125}, draw=none, from=1-1, to=2-2]
\end{tikzcd}
 \]
In this lemma, we use the $!$-relative perverse $t$-structure given by Corollary \ref{corollary:shriekreltstructure}.

We let $\flat\in\{!,*\}$. We have the following properties.
\begin{enumerate}
 \item Let $\SF\in\CD^{\rmb}_{\rmc}(U_y,\BoQ)$ be a constructible complex. then, $g^*(\id_Y\times\jmath_{U_y})_{\flat}(\BoQ_Y\boxtimes\SF)\cong g^*(\BoQ_Y\boxtimes(\jmath_{U_y})_{\flat}\SF)$,
 \item For any constructible complex $\SG\in\CD^{\rmb}_{\rmc}(X_y,\BoQ)$, $\pYtau{\leq j}(g^*(\BoQ_Y\boxtimes\SG))\cong g^*(\BoQ_Y\boxtimes(\ptau{\leq j-2\dim Y}\SG))$ and so, in particular, $\pYtau{\leq j}(g^*(\id_Y\times\jmath_{U_y})_{\flat}(\BoQ_Y\boxtimes\SF))\cong g^*(\BoQ_Y\boxtimes(\ptau{\leq j-2\dim Y}(\jmath_{U_y})_{\flat}\SF))$,
 \item For any constructible complex $\SG\in\CD^{\rmb}_{\rmc}(X_y,\BoQ)$, $\pYtau{\geq j}(g^*(\BoQ_Y\boxtimes\SG))\cong g^*(\BoQ_Y\boxtimes(\ptau{\geq j-2\dim Y}\SG))$ and so, in particular, $\pYtau{\geq j}(g^*(\id_Y\times\jmath_{U_y})_{\flat}(\BoQ_Y\boxtimes\SF))\cong g^*(\BoQ_Y\boxtimes(\ptau{\geq j-2\dim Y}(\jmath_{U_y})_{\flat}\SF))$,
 \item For any constructible complex $\SG\in\CD^{\rmb}_{\rmc}(X_y,\BoQ)$, $\pYH{j}(g^*(\BoQ_Y\boxtimes\SG))\cong g^*(\BoQ_Y[2\dim Y]\boxtimes(\pH{j-2\dim Y}(\SG))$ and so, in particular, $\pYH{j}(g^*(\id_Y\times\jmath_{U_y})_{\flat}(\BoQ_Y\boxtimes\SF))\cong g^*(\BoQ_Y[2\dim Y]\boxtimes(\pH{j-2\dim Y}(\jmath_{U_y})_{\flat}\SF))$.
 \item  If $\SF_y\in\Perv(U_y)$ is a perverse sheaf, then
 \[
  (\id_Y\times\jmath_{U_y})_{!*}(\BoQ_Y[\dim Y]\boxtimes\SF)\cong\BoQ_{Y}[\dim Y]\boxtimes((\jmath_{U_y})_{!*}\SF)
 \]
 is an isomorphism of perverse sheaves on $X$.
 \end{enumerate}
\end{lemma}
\begin{proof}
 The point $(1)$ follows from the compatibility of pushforwards with products: if $a\times b\colon A\times B\rightarrow C\times D$ is a morphism between products of algebraic varieties, $(a\times b)_{\flat}\cong a_{\flat}\boxtimes b_{\flat}$ (see \cite[\S4.2.7 (a)]{beilinson2018faisceaux}). $(4)$ is a combination of $(2)$ and $(3)$. $(2)$ and $(3)$ are dual to each other and the second part of $(2)$ (resp. $(3)$) is obtained by applying the first part of $(2)$ (resp. $(3)$) to $\SG=(\jmath_{U_y})_{\flat}(\SF)$, using also $(1)$. We may therefore only prove the first part of $(2)$. We have a distinguished triangle \cite[Proposition 1.3.3]{beilinson2018faisceaux}
\[
 \ptau{\leq j-2\dim Y}\SG\rightarrow\SG\rightarrow\ptau{> j-2\dim Y}\SG\rightarrow.
\]
Since $\BoQ_{Y}\boxtimes -$ and $g^*$ are triangulated functors, we obtain a distinguished triangle
\[
 g^*(\BoQ_Y\boxtimes\ptau{\leq j-2\dim Y}\SG)\rightarrow g^*(\BoQ_Y\boxtimes\SG)\rightarrow g^*(\BoQ_Y\boxtimes\ptau{> j-2\dim Y}\SG)\rightarrow.
\]
For $y'\in Y$, we let $\imath_{X_{y'}}\colon X_{y'}\rightarrow X$ be the inclusion of the fiber of $f$ over $y'$. We have $\imath_{X_{y'}}^!g^*(\BoQ_Y\boxtimes\ptau{\leq j-2\dim Y}\SG)\cong (\ptau{\leq j-2\dim Y}\SG)[-2\dim Y]$ (resp. $\imath_{X_{y'}}^!g^*(\BoQ_Y\boxtimes\ptau{> j-2\dim Y}\SG)\cong (\ptau{> j-2\dim Y}\SG)[-2\dim Y]$) via the homeomorphism $X_y\cong X_{y'}$ induced by $g$ (and since $Y$ being smooth, we have $\imath_y^!\BoQ_Y\cong \BoQ_{\{y\}}[-2\dim Y]$). Consequently, $\imath_{X_{y'}}^!g^*(\BoQ_Y\boxtimes\ptau{\leq j-2\dim Y}\SG)$ (resp. $\imath_{X_{y'}}^!g^*(\BoQ_Y\boxtimes\ptau{> j-2\dim Y}\SG)$) is in perverse degrees $\leq j$ (resp. $>j$). By the characterization of the $!$-relative perverse $t$-structure (Corollary \ref{corollary:shriekreltstructure}), this implies that $g^*(\BoQ_Y\boxtimes\ptau{\leq j-2\dim Y}\SG)\in{\pSD{Y}{\leq j}(X)}$ (resp. $g^*(\BoQ_Y\boxtimes\ptau{> j-2\dim Y}\SG)\in {\pSD{Y}{>j}(X)}$. By unicity of distinguished triangles associated to truncation of objects up to isomorphism \cite[Proposition 1.3.3]{beilinson2018faisceaux}, we obtain isomorphisms
\[
 \pYtau{\leq j}g^*(\BoQ_Y\boxtimes\SG)\cong g^*(\BoQ_Y\boxtimes\ptau{\leq j-2\dim Y}\SG) \quad\text{ and }\quad \pYtau{> j}g^*(\BoQ_Y\boxtimes\SG)\cong g^*(\BoQ_Y\boxtimes\ptau{> j-2\dim Y}\SG).
\]
$(5)$ is rather straightforward when working on the trivialization $Y\times X_y$ of $X$ and with perverse sheaves on stratified spaces as in \cite[\S2.1]{beilinson2018faisceaux}.
\end{proof}

\begin{corollary}
\label{corollary:simplesummands}
 Let $f\colon X\rightarrow Y$ be a morphism between algebraic varieties such that $Y$ is smooth and $f$ is a topologically trivial fibration. We let $y\in Y$ and $g\colon X\rightarrow Y\times X_y$ be a homeomorphism. We assume given an algebraic stratification $X=\sqcup_{i\in I}X_i$ and that it induces via $g$ an algebraic stratification $\CS_y$: $X_{y}=\bigsqcup_{i\in I}X_{y,i}$. Then, for any constructible complexe $\SG\in\CD^{\rmb}_{\rmc}(X_y)$ constructible for the stratification $\CS_y$, the simple direct summands of $g^*(\BoQ_Y\boxtimes \SG)$ are of the form $g^*(\id_Y\times \jmath_{F_y})_{!*}(\BoQ_Y[\dim Y]\boxtimes \SG')$ for $\jmath_F\colon F\rightarrow X$ a locally closed immersion of a smooth variety $F$ such that $g\circ \jmath_F$ is identified with $\id_Y\times\jmath_{F_y}\colon Y\times F_y\rightarrow Y\times X_y$ (where $F_y\coloneqq F\cap f^{-1}(y)$) and $\SG'$ a simple local system on $F_y$.
\end{corollary}
\begin{proof}
 This is immediate when working with perverse sheaves for stratified spaces as in \cite[\S2.1]{beilinson2018faisceaux} and using $(5)$ in Lemma \ref{lemma:propertyrelative}. The locally closed subset $F$ is one of the $X_i$'s.
\end{proof}

\begin{remark}
 To the knowledge of the author, the literature does not contain the notion of relative $t$-structure for the derived category of mixed Hodge modules on an algebraic variety $X$ with a morphism $f\colon X\rightarrow S$. It is highly possible that the sheaf of relative differential operators on $f$ would appear when attempting to construct such a $t$-structure. It would be desirable to obtain mixed Hodge modules enhancements of our constructions in a more straightforward way. 
\end{remark}

\subsubsection{A useful lemma}
We now formulate an easy but convenient lemma.

\begin{lemma}
\label{lemma:morphismonfiber}
 Let $f\colon X\rightarrow Y$ be a morphism between complex algebraic varieties. We assume that $f$ is a topologically trivial fibration and that $\RHom_{\CD^{\rmb}(Y)}(\BoQ_Y,\BoQ_Y)\cong\BoQ$ (for example, $Y$ is contractible). Then, for any $y\in Y$ and $\SF,\SG\in\CD^{\rmb}_{\rmc}(X_y,\BoQ)$, by topologically identifying $X\cong X_y\times Y$ and $f$ with the second projection, the restriction map (given by the $*$-restriction to $X_y$)
 \[
  \Hom_{\CD^{\rmb}_{\rmc}(X)}(\SF\boxtimes \BoQ_Y,\SG\boxtimes\BoQ_Y)\rightarrow \Hom_{\CD^{\rmb}_{\rmc}(X_y)}(\SF,\SG)
 \]
is an isomorphism. Similarly, the $!$-restriction to the fiber $X_y$ between the same $\Hom$ spaces is an isomorphism.

\end{lemma}
\begin{proof}
By the compatibility between the external tensor product and the external Hom functor \cite[\S4.2.7 (b)]{beilinson2018faisceaux}, we have
\[
 \intHom_{\CD^{\rmb}_{\rmc}(X)}(\SF\boxtimes\BoQ_Y,\SG\boxtimes\BoQ_Y)\cong \intHom_{\CD^{\rmb}_{\rmc}(X_y)}(\SF,\SG)\boxtimes\intHom_{\CD^{\rmb}_{\rmc}(Y)}(\BoQ_Y,\BoQ_Y).
\]
By compatibility of the external tensor product with pushforwards, we obtain
\[
  \RHom_{\CD^{\rmb}_{\rmc}(X)}(\SF\boxtimes\BoQ_Y,\SG\boxtimes\BoQ_Y)\cong \RHom_{\CD^{\rmb}_{\rmc}(X_y)}(\SF,\SG)\otimes\RHom_{\CD^{\rmb}_{\rmc}(Y)}(\BoQ_Y,\BoQ_Y).
\]
Under our assumptions, the right-hand-side reduces to $\RHom_{\CD^{\rmb}_{\rmc}(X_y)}(\SF,\SG)$ and the map is given by the $*$-restriction.

The case of the $!$-restriction follows by combining the $*$ case with the Verdier duality.
\end{proof}

\subsection{Mixed Hodge modules}
We work with mixed Hodge modules (MHM) over the spaces considered. The body of this paper only uses MHM but for later use, we give a short account on monodromic mixed Hodge modules in Appendix \ref{section:purityforMMHM} (for results regarding purity). We refer to \cite{davison2020cohomological} or \cite[\S2.2]{davison2021purity} for more details and background on (monodromic) mixed Hodge modules and their use in representation and Donaldson--Thomas theories.

\subsubsection{Mixed Hodge modules}
Let $X$ be a complex algebraic variety. Saito defined an Abelian category $\MHM(X)$, the category of mixed Hodge modules on $X$ (see \cite{saito1990mixed}, the references therein and the subsequent literature). This can be seen as an enhancement of the category of perverse sheaves $\Perv(X)$ in the sense that there is an exact functor $\rat_X\colon\MHM(X)\rightarrow\Perv(X)$. This functor can be derived to a functor $\CD^{\rmb}(\MHM(X))\rightarrow\CD^{\rmb}(\Perv(X))\simeq \CD^{\rmb}_{\rmc}(X)$ (the last equivalence being Beilinson's theorem, \cite{beilinson2006derived}). It intertwines the natural six functors on these categories: $f^*, f_*, f^!, f_!, \otimes$ and  $\intHom$.

\subsubsection{Tate twist}
We let $\BoL\coloneqq \HO^*_{\rmc}(\BoA^1,\BoQ)$ be the compactly supported cohomology of $\BoA^1$, with its natural mixed Hodge structure. This is a complex of mixed Hodge structures concentrated in cohomological degree $2$ and pure of weight $0$.

\subsubsection{Intersection complexes}
Let $X$ be a complex algebraic variety and $\jmath\colon X^{\sm}\rightarrow X$ its smooth locus. If $X$ is irreducible even dimensional, the constant mixed Hodge module $\underline{\BoQ}_{X^{\sm}}\otimes\BoL^{-\dim(X)/2}$ is pure of weight zero. It admits a unique extension $\underline{\IC}(X)\coloneqq\jmath_{!*}\underline{\BoQ}_{X^{\sm}}\otimes\BoL^{-\dim(X)/2}$ to a pure weight $0$ mixed Hodge module on $X$. It is the \emph{intersection complex mixed Hodge module} of $X$. We have $\rat_X(\underline{\IC}(X))\cong \IC(X)$.

\begin{lemma}
\label{lemma:relativeICmhm}
 Let $f\colon X\rightarrow S$ be a morphism between algebraic varieties such that $f$ is a topologically locally trivial  fibration, $\dim X-\dim S$ is even and the smooth locus $\jmath\colon X^{\sm}\rightarrow X$ of $X$ intersects any fiber of $f$ in a dense subset. Then, the pure weight zero complex of mixed Hodge modules $\underline{\IC}(X/S)\coloneqq\jmath_{!*}\underline{\BoQ}_{X^{\sm}}\otimes\BoL^{(\dim S-\dim X)/2-2\dim S}$ is such that for any $s\in S$, $\imath_{X_s}^!\underline{\IC}(X/S)\cong \underline{\IC}(X_s)$.
\end{lemma}
\begin{proof}
 Let $s\in S$ be a geometric point. We let $X_s^{\sm,\sm}\coloneqq X_s^{\sm}\cap X^{\sm}$, where $X_s=f^{-1}(s)$. By assumption, $X_s^{\sm,\sm}$ is a dense open subset of $X_s$. If $\jmath_{X_s^{\sm,\sm}}\colon X_s^{\sm,\sm}\rightarrow X$ denotes the inclusion, we have $\jmath_{X_s^{\sm,\sm}}^!\underline{\IC}(X/S)\cong \underline{\BoQ}_{X_s^{\sm,\sm}}\otimes\BoL^{(\dim S-\dim X)/2}$. The lemma then follows from $\rat_X(\underline{\IC}(X/S))\cong \IC(X/S)$ by construction of $\underline{\IC}(X/S)$.
\end{proof}

We call $\underline{\IC}(X/S)$ the \emph{relative intersection complex} of $X/S$.

\subsubsection{Mixed Hodge modules on stacks}
We will only deal with mixed Hodge modules on global quotient stacks. Mixed Hodge modules on arbitrary Artin stacks are harder to work with -- one classically assume that there is an exhaustion by global quotient stacks (see \cite{davison2021purity} for example). On a global quotient stack, one can work with equivariant mixed Hodge modules and their derived category, as explained by Achar in \cite{achar2013equivariant}. Moreover, we will need to consider direct images of mixed Hodge modules for a morphism from a global quotient stack to an algebraic variety. For this, we refer to the detailed treatment in \cite[\S2.3]{davison2021purity}. One general caveat is that the six-functor formalism is subtle and not defined for general stacks.

\subsection{Monoidal categories}

\subsubsection{Monoidal structures: monoid over a point}
\label{subsubsection:monoidalstrucpoint}
Let $\CM$ be a monoid in the category of complex schemes. We assume that each connected component of $\CM$ is a finite type separated scheme. The monoid structure is given by the map $\oplus\colon\CM\times\CM\rightarrow\CM$ and the unit by $0_{\CM}\colon\pt\rightarrow\CM$ where $\pt=\Spec(\BoC)$. The formula
\[
 \SF\boxdot\SG\coloneqq \oplus_*(\SF\boxtimes\SG)
\]
defines monoidal structures on the categories $\CD^+_{\rmc}(\CM,\BoQ)$ and $\CD^+(\MHM(\CM))$, and assuming that $\oplus$ is a finite morphism, on $\Perv(\CM)$ and $\MHM(\CM)$. We refer to \cite[\S3.1]{davison2022BPS} for more details on this monoidal category.

\subsubsection{Monoidal structures: monoid over an algebraic variety}
\label{subsubsection:monoidalstroveralgvar}
Let $S$ be a finite type separated $\BoC$-scheme. Let $f\colon\CM\rightarrow S$ be a monoid in the category of $S$-schemes. We assume that $f$ is a finite type separated morphism. In this section, we explain how to endow the categories $\CD^+_{\rmc}(\CM,\BoQ)$, $\CD^+(\MHM(\CM))$ and $\Perv(\CM/S)$ (relative perverse sheaves, \S\ref{subsection:relativeptstructure}) with monoidal structures. We let $\oplus\colon \CM\times_S\CM\rightarrow\CM$ be the monoid structure on $\CM$ and $0_{\CM}\colon S\rightarrow \CM$ be the unit. It is a closed immersion. We let $\iota\colon\CM\times_S\CM\rightarrow\CM\times\CM$ be the closed immersion (by separatedness of $S$). For a closed point $s\in\CM$, we let $\imath_{\CM_s}\colon\CM_s\rightarrow\CM$ be the closed immersion of the fiber of $f$ over $s$. The base-change over $s$ of the monoid structure $\oplus$ gives a monoid $(\CM_s,\oplus)$ as in \S\ref{subsubsection:monoidalstrucpoint}. We let $\boxdot$ be the associated monoidal structure. For any $\SF, \SG\in\CD^+_{\rmc}(\CM,\BoQ)$ (resp. $\CD^+(\MHM(\CM))$), we let
\begin{equation}
\label{equation:monoidalproduct}
 \SF\boxdot_S\SG\coloneqq \oplus_*\iota^!(\SF\boxtimes\SG).
\end{equation}
We also let $\SF\boxtimes_S\SG\coloneqq \imath^!(\SF\boxtimes\SG)$ to shorten the notation.

\begin{proposition}
\label{proposition:compexternalS}
 For any closed point $s\in S$, and any $\SF,\SG\in\CD^+(\MHM(\CM))$ (resp. $\CD^+_{\rmc}(\CM)$), $\imath_{\CM_s}^!(\SF\boxdot_S\SG)\cong(\imath_{\CM_s}^!\SF)\boxdot(\imath_{\CM_s}^!\SG)$.
\end{proposition}
\begin{proof}
This follows by base-change in the Cartesian square
\[
 \begin{tikzcd}
	{\CM_s\times\CM_s} & {\CM_s} \\
	{\CM\times_S\CM} & \CM
	\arrow["\oplus", from=2-1, to=2-2]
	\arrow["{\imath_{\CM_s}}", from=1-2, to=2-2]
	\arrow["{\imath_{\CM_s}\times_S\imath_{\CM_s}}"', from=1-1, to=2-1]
	\arrow["\oplus", from=1-1, to=1-2]
	\arrow["\lrcorner"{anchor=center, pos=0.125}, draw=none, from=1-1, to=2-2]
\end{tikzcd}
\]
using the fact that we have a commutative triangle
\[
 \begin{tikzcd}
	{\CM_s\times\CM_s} && {\CM\times_S\CM} \\
	& \CM\times\CM
	\arrow["{\imath_{\CM_s}\times\imath_{\CM_s}}"'{pos=0.5}, from=1-1, to=2-2]
	\arrow["{\imath_{\CM_s}\times_S\imath_{\CM_s}}", from=1-1, to=1-3]
	\arrow["\imath", from=1-3, to=2-2]
\end{tikzcd}
\]
and the compatibility of the external product $\boxtimes$ with exceptional pullbacks.
\end{proof}

\begin{proposition}
\label{proposition:monoidalstructures}
The formula \eqref{equation:monoidalproduct} defines monoidal structures on $\CD^+_{\rmc}(\CM,\BoQ)$, $\CD^+(\MHM(\CM))$. If $\oplus$ is finite on fibers (i.e. $\oplus\colon\CM_s\times\CM_s\rightarrow\CM_s$ is finite for any closed point $s\in S$), it induces a monoidal structure on $\Perv(\CM/S)$. The monoidal unit is $(0_{\CM})_*\BoQ[2\dim S]\in\Perv(\CM/S)$ and if the monoid $(\CM,\oplus)$ is commutative, the monoidal structures are symmetric.
\end{proposition}
\begin{proof}
 We first need to prove the associativity of $\boxdot_S$. Since $\oplus$ is associative (being a monoid structure on $\CM$), we just need to prove that for any $\SF,\SG,\SH\in\CD^+_{\rmc}(\CM,\BoQ)$ (resp. $\CD^+(\MHM(\CM))$), $(\SF\boxtimes_S\SG)\boxtimes_S\SH\cong\SF\boxtimes_S(\SG\boxtimes_S\SH)$. This follows from the commutativity of the diagram
 \[
  \begin{tikzcd}
	{X\times_SX\times_SX} & {(X\times_SX)\times X} \\
	{X\times (X\times_SX)} & {X\times X\times X}
	\arrow["\id\times\imath"', from=2-1, to=2-2]
	\arrow["\imath\times\id", from=1-2, to=2-2]
	\arrow["{\imath''}"', from=1-1, to=2-1]
	\arrow["{\imath'}", from=1-1, to=1-2]
\end{tikzcd}
 \]
where $\imath'\colon X\times_SX\times_SX\rightarrow (X\times_SX)\times X$ is obtained via the associativity $X\times_SX\times_SX\cong (X\times_SX)\times_SX$ and the maps $\id\colon X\times_SX\rightarrow X\times_S X$ and $\id\colon X\rightarrow X$, and similarly for the map $\imath''$. If $(\CM,\oplus)$ is commutative, this monoidal structure is easily seen to be symmetric.

To show this induces a monoidal structure on $\Perv(\CM/S)$ when $\oplus$ is finite on fibers, we need to prove that the operation $\boxdot_S$ preserves relative perverse sheaves on $\CM$. This is a consequence of Proposition \ref{proposition:compexternalS}. Indeed, if $\SF, \SG\in\Perv(\CM/S)$ are relative perverse sheaves, then $\imath^!_{\CM_s}(\SF\boxdot_S\SG)\cong (\imath_{\CM_s}^!\SF)\boxdot(\imath_{\CM_s}^!\SG)=\oplus_* (\imath_{\CM_s}^!\SF)\boxtimes(\imath_{\CM_s}^!\SG)$ where the first isomorphism is Proposition~\ref{proposition:compexternalS}. By assumption, $\imath_{\CM_s}^!\SF$ and $\imath_{\CM_s}^!\SG$ are perverse sheaves on $\CM_s$. The external tensor product $\boxtimes$ and the pushforward $\oplus_*$ for the finite morphism $\oplus $ are perverse $t$-exact and so $\oplus_* (\imath_{\CM_s}^!\SF)\boxtimes(\imath_{\CM_s}^!\SG)$ is perverse. This proves that $\SF\boxdot_S\SG$ is indeed relative perverse following the definition of the $!$-relative perverse $t$-structure in Corollary~\ref{corollary:shriekreltstructure}.

Since $0_{\CM}$ is a closed immersion and $S$ is smooth, $(0_{\CM})_*\BoQ_S[2\dim S]\in \Perv(\CM/S)$ (Corollary \ref{corollary:shriekreltstructure}). The unitality comes from the fact that in the diagram
\begin{equation}
\label{equation:monoidalunit}
 \begin{tikzcd}
	{S\times_S\CM\cong\CM} & S\times\CM \\
	{\CM\times_S\CM} & \CM\times\CM \\
	\CM
	\arrow["\oplus"', from=2-1, to=3-1]
	\arrow["{0_{\CM}\times_S\id}"', from=1-1, to=2-1]
	\arrow["{0_{\CM}\times\id}", from=1-2, to=2-2]
	\arrow["\imath"', from=2-1, to=2-2]
	\arrow["{\imath_0=f\times\id}", from=1-1, to=1-2]
	\arrow["\lrcorner"{anchor=center, pos=0.125}, draw=none, from=1-1, to=2-2]
	\arrow["\id"', bend right=80, from=1-1, to=3-1]
\end{tikzcd},
\end{equation}
we have $\oplus\circ (0_{\CM}\times\id)=\id_{\CM}$ and so
\[
 ((0_{\CM})_*\BoQ_S[2\dim  S])\boxdot_S\SF\cong\oplus_*\imath^!(0_{\CM}\times\id)_*(\BoQ_S[2\dim S]\boxtimes\SF)\cong \oplus_*(0_{\CM}\times_S\id)_*\imath_0^!(\BoQ_S[2\dim S]\boxtimes\SF)\cong \SF
\]
where the first isomorphism comes from the compatibility of the external tensor product with pushforwards \cite[\S4.2.7 (a)]{beilinson2018faisceaux}, the second isomorphism is the base-change map in the diagram \eqref{equation:monoidalunit}, and the third isomorphism comes from $\oplus\circ (0_{\CM}\times_S\CM)=\id$ together with $\imath_0^!(\BoQ_S[2\dim S])\boxtimes\SF\cong \imath_0^!\pr_2^!\SF\cong\SF$. We used the smooth projection $\pr_2\colon \CM\times S\rightarrow\CM$, which gives $\pr_2^!\SF\cong (\BoQ_S[2\dim S])\boxtimes\SF$.
\end{proof}

\subsubsection{Monoidal structures: for an algebraic variety}
Let $S$ be an algebraic variety. It can be seen as a monoid over itself, where $f=\id_S\colon S\rightarrow S$ and the monoid structure $\oplus\colon S\times_SS\rightarrow S$ is also given by the identity morphism. Then, the monoidal structures on $\CD^+(\MHM(S))$ and $\CD^+_{\rmc}(S)$ obtained from \S\ref{subsubsection:monoidalstroveralgvar} are nothing else than the $!$-derived tensor product $\stackrel{!}{\otimes}$:
\[
 \SF\boxdot_S\SG=\imath^!(\SG\boxtimes\SG)
\]
where $\imath\colon S\rightarrow S\times S$ is the diagonal.

\section{Higgs sheaves and local systems}
In this section, we concentrate ourselves on curves, although the definitions make sense or can be adapted for any algebraic variety. We recall the main features of the categories of local systems on a closed Riemann surface and (semistable) Higgs bundles on a smooth projective curve. These are the properties that allowed the author together with Davison and Schlegel Mejia to carry a thorough study of the Borel--Moore homologies of the associated stacks, culminating in a stacky nonabelian Hodge isomorphism between the Borel--Moore homologies of the Betti and Dolbeault stacks \cite[Theorem 1.7]{davison2022BPS}. In this paper, we consider the untwisted versions (i.e. no punctures on the Riemann surface and Higgs bundles of degree $0$). We refer to \cite{davison2022BPS} for more details or for the twisted cases. Our methods adapt the the twisted situations and nonzero degree Higgs bundles, see \S\ref{subsubsection:parabolicversions}, although the constructions of the moduli spaces and stacks bear some subtleties but were considered in \cite{nitsure1993moduli}.

\subsection{Local systems}
\subsubsection{Categories of local systems}
\label{subsubsection:categorieslocsys}
Let $\Sigma_g$ be a closed Riemann surface of genus $g\geq 1$ and $p\in \Sigma_g$ a fixed point. We let
\[
 \pi_1(\Sigma_g,p)\cong \left\langle a_i,b_i\colon 1\leq i\leq g\mid \prod_{i=1}^ga_ib_ia_i^{-1}b_i^{-1}=1\right\rangle
\]
be its fundamental group based at $p$.

The fundamental group algebra, $\BoC[\pi_1(\Sigma_g,p)]$ is a $2$-Calabi--Yau algebra (\cite[Corollary 6.2.4]{davison2012superpotential} and also \cite[Corollary 7.3]{davison2022BPS}). In \cite{davison2022BPS}, the case of \emph{twisted} fundamental group algebras is dealt with using multiplicative preprojective algebras and \cite{kaplan2023multiplicative}.

\subsubsection{Moduli stacks of local systems}
Connected components $\FM_{g,r}^{\Betti}$ ($r\in\BoN$) of the moduli stack of local systems can be described as global quotient stacks. Namely, the moduli stack of rank $r$ local systems on $\Sigma_g$ is
\[
 \mathfrak{M}^{\Betti}_{g,r}\coloneqq X_{g,r}^{\Betti}/\GL_r
\]
where $X_{g,r}^{\Betti}=\{(M_i,N_i)_{1\leq i\leq g}\in\GL_{r}^{2g}\mid \prod_{i=1}^gM_iN_iM_i^{-1}N_i^{-1}=1\}$ is an affine algebraic variety and $\GL_r$ acts on it by simultaneous conjugation. We let $\FM_g^{\Betti}\coloneqq\bigsqcup_{r\in\BoN}\FM_{g,r}^{\Betti}$, where is is implicitly understood that $\FM_{g,0}^{\Betti}=\pt$.

The good moduli space of $\FM_g^{\Betti}$ is denoted by $\CM_g^{\Betti}$. It can be decomposed as $\CM_g^{\Betti}=\bigsqcup_{r\in\BoN}\CM_{g,r}^{\Betti}$ where $\CM_{g,r}^{\Betti}\coloneqq X_{g,r}^{\Betti}\cms\GL_r$ is the affine GIT quotient. We let $\JH^{\Betti}\colon \FM_g^{\Betti}\rightarrow\CM_g^{\Betti}$ be the Jordan--H\"older map, which at the level of $\BoC$-points sends a representation of $\pi_1(\Sigma_g,p)$ to its associated graded with respect to some Jordan--H\"older filtration.

\subsubsection{The RHom complex}

The RHom complex on $\FM_g^{\Betti}\times\FM_g^{\Betti}$ admits a resolution by a $3$-term complex of vector bundles. Its description is recalled in \cite[\S7.2.3]{davison2022BPS}, to which we refer for more details. It is obtained from the standard resolution of $A\coloneqq \BoC[\pi_1(\Sigma_g,p)]$ by projective $A$-bimodules
\[
 0\rightarrow K\rightarrow A\otimes A\otimes A\rightarrow A\otimes A\rightarrow A\rightarrow 0.
\]
It is directly given as a $3$-term complex of vector bundles on $\FM_g^{\Betti}\times\FM_g^{\Betti}$ (\cite[\S7.2.3]{davison2022BPS}, formulated in the closely related context of multiplicative preprojective algebras). We denote a shift of this complex by $\CC^{\Betti}\coloneqq \RHom[1]$.

By now classical facts concerning stacks of short exact sequences, their derived enhancements and cotangent complexes, the classical truncation $t_0(\Tot(\CC^{\Betti}))$ is isomorphic to $\mathfrak{Exact}_A\coloneqq\mathfrak{Exact}^{\Betti}$, the stack of short exact sequences of representations of $\pi_1(\Sigma_g,p)$, see for example \cite[Proposition 3.6, Corollary 3.7]{porta2022two}.

\subsubsection{The Euler form}
\label{subsubsection:theeulerformfdga}
The Euler form of the category of local systems is given by $\chi(M,N)=2(1-g)\rank(M)\rank(N)$ where $M,N$ are two local systems on $C$. Indeed $\chi(M,N)$ is given by the cohomology of $C$ with coefficients in the local system $M^{\vee}\otimes N$, $\chi(\HO^*(C,M^{\vee}\otimes N))=\rank(M)\rank(N)\chi(\HO^*_{\sing}(C))$ and $\chi_{\sing}(C)=2(1-g)$. This Euler form may as well be computed using the realisation of fundamental group algebras as localised multiplicative preprojective algebras, as in \cite[\S 7.2.4]{davison2022BPS}.

\subsubsection{The direct sum map}
\label{subsubsection:directsumBetti}
The direct sum of representations of $\pi_1(\Sigma_g)$ induces a morphism of stacks $\oplus\colon\FM_g^{\Betti}\times\FM_g^{\Betti}\rightarrow\FM_g^{\Betti}$. By functoriality, it induces a map at the level of the good moduli spaces $\oplus\colon\CM_g^{\Betti}\times\CM_g^{\Betti}\rightarrow\CM_g^{\Betti}$. It is a classical fact that this map is finite (see \cite[Proposition 6.15]{davison2022BPS}, but also \cite[Lemma 2.1]{meinhardt2019donaldson} for the general strategy).

\subsection{Higgs sheaves}

\subsubsection{Categories of Higgs sheaves}
Let $C$ be a smooth projective curve over $\BoC$. We denote by $K_C=\Omega^1_C$ its canonical line bundle/sheaf of algebraic one-forms. A Higgs sheaf on $C$ is a pair $(\CF,\theta)$ of a coherent sheaf $\CF$ on $C$ and a Higgs field $\theta\colon \CF\rightarrow \CF\otimes_{\CO_C} \Omega^1_C$ (an $\CO_C$-linear map). A Higgs sheaf is said \emph{semistable} if for any $0\neq\SG\subsetneq\SF$ such that $\theta(\SG)\subset\SG\otimes \Omega^1_C$, the inequality of slopes $\mu(\SG)\coloneqq \frac{\deg\SG}{\rank\SG}\leq \frac{\deg\SF}{\rank\SG}=\mu(\SF)$ is satisfied. If the inequality is always strict, the Higgs sheaf is said \emph{stable}. The category of semistable Higgs bundles of some fixed slope $\mu\in\BoQ$ is Abelian and of finite length.

\subsubsection{Moduli stacks of Higgs sheaves}
The moduli stack of (semistable) Higgs sheaves on a smooth projective curve can be constructed in various ways, using versions of Quot schemes. There are (at least) three equivalent but distinct ways to proceed.
\begin{enumerate}
 \item Use the Beauville--Narasimhan--Ramanan (BNR) correspondence and identify Higgs sheaves on $C$ with compactly supported coherent sheaves on $\Tan^*C$. Then, using a compactification $\widetilde{\Tan^*C}$ of $\Tan^*C$, use usual Quot schemes for $\widetilde{\Tan^*C}$. This is the approach recalled in \cite{davison2022BPS}, following the spectral correspondence of \cite{narasimhan1989spectral}. This approach is also explained in \cite[pp.18-19, The second construction]{simpson1994moduliII}.

 \item Use the Quot scheme description of open substacks of the stack of coherent sheaves on the smooth projective curve. Then, describe open substacks of the stack of Higgs sheaves using Hamiltonian reduction. This is the approach pursued in \cite{sala2020comological} to define the cohomological Hall algebra of Higgs sheaves on the curve. See also \cite{pouchin2013higgs} for earlier fruitful use of this description (related to a representation-theoretic study of the global nilpotent cone for \emph{weighted projective lines}).

 \item Use the approach explained in \cite{simpson1994moduli}: A Higgs sheaf is the same as an $\CO_C$-linear map $\mathcal{T}_C\otimes \CF\rightarrow\CF$ where $\mathcal{T}_C$ is the tangent line bundle to $C$. Therefore, a Higgs sheaf is a $\Sym(\CT_C)$-module, where $\Sym(\CT_C)$ is an $\CO_C$-algebra. This approach is the best suited for the analogy with the study of the Hodge moduli space. Indeed, this description allows a more straightforward comparison with the construction of moduli spaces for flat vector bundles and $\lambda$-connections.
\end{enumerate}
These constructions provide us with the Dolbeault stack $\FM^{\Dol}(C)$, parametrising slope $0$ semistable Higgs sheaves on $C$ and its good moduli space, $\CM^{\Dol}(C)$, whose $\BoC$-points are in bijection with the set of isomorphism classes of polystable Higgs sheaves of slope $0$. They are related by the Jordan--H\"older map $\JH^{\Dol}\colon \FM^{\Dol}(C)\rightarrow\CM^{\Dol}(C)$. At the level of $\BoC$-points, it sends a semistable Higgs sheaf to the polystable Higgs sheaf obtained by taking the associated graded with respect to a filtration with stable subquotients.

\subsubsection{The RHom complex}
\label{subsubsection:RHomHiggs}
A global presentation of the RHom complex over $\FM^{\Dol}(C)\times\FM^{\Dol}(C)$ is explained in \cite[Proposition 6.9]{davison2022BPS}. It is a classical construction that provides us with a $3$-term complex of vector bundles
\[
 \CC^{\Dol}=(V^{-1}\rightarrow V^0\rightarrow V^1)
\]
on $\FM^{\Dol}(C)\times\FM^{\Dol}(C)$ such that $t_0(\Tot(\CC^{\Dol}))\simeq \mathfrak{Exact}^{\Dol}(C)$ is canonically equivalent to the stack of short exact sequences of semistable Higgs bundles of slope $0$ on $C$. This presentation was crucial in \cite{davison2022BPS} in order to provide a construction of the relative CoHA product at the level of mixed Hodge modules for the stack of semistable Higgs bundles.

\subsubsection{The direct sum map}
The direct sum of semistable Higgs bundles of slope $0$ induces a morphism between stacks $\oplus\colon\FM^{\Dol}(C)\times\FM^{\Dol}(C)\rightarrow\FM^{\Dol}(C)$. By universality of the good moduli space, we obtain a map $\oplus\colon\CM^{\Dol}(C)\times\CM^{\Dol}(C)\rightarrow\CM^{\Dol}(C)$ at the level of the moduli space of polystable Higgs bundles of slope $0$. As proven in \cite[Proposition 6.14]{davison2022BPS}, this map is finite. 

\section{Connections}
\label{section:connections}
In this section, we develop the formalism necessary to define the CoHA product for the stack of connections on a smooth projective curve and to carry out the study of the BPS algebra as in \cite{davison2022BPS}.

For more details on the objects and constructions, we refer to \cite{hotta2007d,simpson1994moduli,pantev2019moduli}.

As in the introduction, we recall that an alternative approach to defining the CoHA for connections using Khan's formalism (in particular, the pullback in Borel--Moore homology for quasi-smooth morphisms of derived stacks, \cite{khan2019virtual}) together with \cite{porta2022two}.

\subsection{Categories of connections}
\subsubsection{Differential operators}
A vector bundle with connection is a pair $(\CF,\nabla)$ of a vector bundle $\CF$ on $C$ together with a map $\nabla\colon\CF\rightarrow\CF\otimes \Omega^1_C$ satisfying the Leibniz rule: $\nabla(f\cdot e)=e\otimes _{\CO_C}\mathrm{d}f+f\nabla(e)$ for $e,f$ local sections of $\CF, \CO_C$ respectively \cite[Definition 1.2.2]{hotta2007d}.

A vector bundle with connection can be seen as a module over the sheaf of algebras $\Lambda^{\dR}$ over $C$ that is coherent over $\CO_C$, the structure sheaf of $C$, where $\Lambda^{\dR}\coloneqq \CD_C$ is the sheaf of differential operators on $C$. We owe the notation $\Lambda^{\dR}$ to \cite[pp. 85-86]{simpson1994moduli}.

The sheaf of differential operators $\Lambda^{\dR}$ on $C$ is an increasing union of $\CO_C$-coherent submodules: $\Lambda_{\leq k}^{\dR}=\bigcup_{k\geq 0}\Lambda_{\leq k}^{\dR}$ where $\Lambda_{\leq k}^{\dR}$ is the $\CO_C$-coherent sheaf on $C$ of differential operators of order $\leq k$.

We denote by $\CD_{\qcoh}(\Lambda^{\dR})$ the dg-category of complexes of left $\Lambda^{\dR}$-modules localised at quasi-isomorphisms (i.e. morphisms between complexes which induce isomorphisms in cohomology) \cite[\S1.5]{hotta2007d}.

\subsubsection{The de Rham algebra}
\label{subsubsection:deRhamalgebra}
We recall an other viewpoint on connections over the curve $C$, without attempting to be exhaustive or complete. This viewpoint is sketched only for the sake of the derived enhancement of the moduli stack of connections \S\ref{subsubsection:conneections:thederivedmodstack} and can be safely ignored in the first lecture.

Recall that $\Omega^1_C$ denotes the sheaf of algebraic $1$-forms over the smooth projective curve $C$. The de Rham algebra of $C$ is defined as
\[
 \DR_C\coloneqq \Sym_{\CO_C}(\Omega^1_C[1]).
\]
This is a sheaf of graded mixed commutative dg-algebras over $C$ \cite[\S2.2]{pantev2019moduli}, \cite[\S1.2]{pantev2013shifted}. The term ``mixed'' refers to the fact that $\DR_C$ has the additional differential given by the de Rham differential.
\subsubsection{Graded mixed dg-modules over the de Rham algebra}
To construct a derived moduli stack of connections, Pantev and To\"en realise flat connections as \emph{graded mixed dg-modules} over the de Rham algebra \S\ref{subsubsection:deRhamalgebra}. To save space (as they will only play the role of an \emph{ambient} nicely behaved dg-category), we do not recall the precise definition of these objects here. There is a category $\DR_C-\dg^{\gr}_{\epsilon}$ of graded mixed $\DR_C-\dg$-modules (with inverted quasi-isomorphisms), \cite[\S2.2]{pantev2019moduli}. Moreover, there is a fully faithful functor $\DR_C\colon \CD_{\qcoh}(\Lambda^{\dR})\rightarrow\DR_C-\dg^{\gr}_{\epsilon}$ \cite[Proposition 2.5]{pantev2019moduli}. Its essential image is described as the full subcategory of free graded modules over $\DR_C$.

\begin{remark}
 The discussion in \cite[\S5]{pantev2019moduli}, see also \cite[Theorem A]{pantev2019moduli} should give a relative Calabi--Yau structure on the restriction functor from $\DR_X-\dg^{\gr}_{\epsilon}$ to category of objects on the boundary,
 and therefore a $2$-Calabi--Yau structure on the category $\DR_C-\dg^{\gr}_{\epsilon}$ when $X=C$ is a smooth projective curve. This should be related to the geometric Calabi--Yau completions which are being developed by Kinjo--Safronov.
\end{remark}

\subsection{The moduli stack of connections and the good moduli space}
\subsubsection{Classical construction}
\label{subsubsection:connectionsclassical}
The moduli stacks and moduli spaces of connections are well-studied objects in algebraic geometry. They are constructed for example in \cite[Theorem 3.8]{simpson1994moduli}. See also \cite[Theorem 4.10]{simpson1994moduli}. The construction provides for each rank $r\in\BoN$ a \emph{representation space} $R^{\dR}_r(C)$ parametrising pairs $(\CF,\beta)$ where $\CF$ is a $\Lambda^{\dR}$-modules (semistability is automatic, \cite[Theorem 6.13]{simpson1994moduliII}) and $\beta$ is the framing, i.e. $\beta\colon \CF_p\cong\BoC^{r}$ is an isomorphism for a fixed point $p\in C$. The general linear group $\GL_r$ acts on $R_r^{\dR}(C)$ by changing the framing. The stacky quotient $\FM_r^{\dR}(C)\coloneqq R_r^{\dR}(C)/\GL_r$ is the moduli stack of rank $r$ connections on $C$. Moreover, the $\GL_r$-action on $R_r^{\dR}(C)$ admits a linearisation $\mathscr{L}^{\dR}$ where $\mathscr{L}^{\dR}$ is very ample and the GIT quotient with respect to this linearisation is exactly the de Rham moduli space, parametrising semisimple connections on $C$. Note that the existence of this linearisation implies that the quotient stack $\FM^{\dR}_r(C)$ have the resolution property \cite[Theorem~1.1, (2)]{totaro2004resolution}. Indeed, the linearization $\mathscr{L}^{\dR}$ gives an embedding in an affine chart of a projective variety. The resolution property is something that is desirable when defining the associated CoHA \cite[Assumption 2]{davison2022BPS}. We let $\FM^{\dR}(C)\coloneqq \bigsqcup_{r\geq 0}\FM^{\dR}_r(C)$ and $\CM^{\dR}(C)\coloneqq\bigsqcup_{r\geq 0}\CM_r^{\dR}(C)$. We have the good moduli space map $\JH^{\dR}\colon \FM^{\dR}(C)\rightarrow\CM^{\dR}(C)$.

\subsubsection{The derived moduli stack}
\label{subsubsection:conneections:thederivedmodstack}
There is a way to upgrade the moduli stack of connections to a derived stack. We give some information on this stack, without attempting to be exhaustive. We let $\bm{\FM}^{\dR}(C)\coloneqq\Perf^{\nabla}(C)$ be the derived stack of graded mixed dg-modules over $\DR_C$ which are perfect of degree $0$ \cite[page 26]{pantev2019moduli}. An other description uses \emph{mapping stacks}. Namely, we have $\Perf^{\nabla}(C)=\Map_{\dSt}(C_{\dR},\Perf)$ where
\begin{enumerate}
 \item $C_{\dR}$ is the de Rham functor associated to $C$:
 \[
 \begin{matrix}
  C_{\dR}(A)=C(A_{\cl}^{\red})
 \end{matrix}
 \]
for $A$ a nonpositively graded commutative differential graded algebra, where $A_{\cl}=\HO^0(A)$ is the classical truncation and the superscript ``$\red$'' indicates the reduced quotient (quotient by the nilradical). We denote by $C(-)$ the functor of points of the curve $C$. See for example \cite[\S4.1]{naef2023torsion}.
 \item $\Perf$ is the derived stack of perfect complexes.
\end{enumerate}
See also \cite[p.32]{simpson1996hodge}.

We let $\BT$ be the tangent complex of $\Perf^{\nabla}(C)$. For a perfect graded mixed dg-module over $\DR_C$ of degree $0$, corresponding to the vector bundle with connection $(E,\nabla)$, the restriction of $\BT$ to the point $(E,\nabla)$ is given by
\[
 \BT_{E}=\BH_{\rmDR}(C,E\otimes E^{\vee}[1]),
\]
the de Rham cohomology with values in the perfect object $E\otimes E^{\vee}[1]$, see \cite[Proposition~4.8]{pantev2019moduli}. This de Rham cohomology can be computed as follows. The connection
\[
 \nabla\colon E\rightarrow E\otimes\Omega^1_C
\]
can be dualised in a connection
\[
 \nabla^{\vee}\colon E^{\vee}\rightarrow E^{\vee}\otimes\Omega^1_C
\]
on the dual complex $E^{\vee}$. The associated de Rham complex is
\[
 E\otimes E^{\vee}\xrightarrow{\nabla+\nabla^{\vee}} E\otimes E^{\vee}\otimes\Omega^1_C.
\]
The tangent complex $\BT_E$ is the hypercohomology of this complex, which is just the $E\otimes E^{\vee}$-twisted de Rham complex.

We let $\bm{\FM}^{\dR}(C)\subset\Perf^{\nabla}(C)$ be the open substack of perfect graded mixed dg-modules concentrated in degree $0$. This is a derived enhancement of the Artin stack $\FM^{\dR}(C)$ constructed in \S\ref{subsubsection:connectionsclassical}: $t_0(\bm{\FM}^{\dR}(C))\simeq\FM^{\dR}(C)$, where $t_0$ denotes the classical truncation.

\subsubsection{The stack of short exact sequences and the cotangent complex}
We let $\bm{\mathfrak{Exact}}^{\dR}(C)$ be the derived moduli stack of short exact sequences of connections. It may be constructed as in \cite{sala2020comological}. It admits a map
\[
 \bm{\mathfrak{Exact}}^{\dR}(C)\xrightarrow{\bm{q}^{\dR}}\bm{\FM}^{\dR}(C)\times\bm{\FM}^{\dR}(C),
\]
which on $\BoC$ points sends a short exact sequence to its extremal terms.

There is a natural map
\[
 \FM^{\dR}(C)\times\FM^{\dR}(C)\xrightarrow{i_{\oplus}}\mathfrak{Exact}^{\dR}(C)
\]
sending a pair of connections to the trivial extension. If $(E,\nabla_E)$ and $(F,\nabla_F)$ are connections on $C$, the tangent complex of $\bm{q}^{\dR}$ at the trivial extension (direct sum) of $E$ and $F$ is computed by
\begin{equation}
\label{equation:tangentextensions}
 \BH_{\rmDR}(C,E\otimes F^{\vee}[1]).
\end{equation}
The computation is analogous to that of the tangent complex (\S\ref{subsubsection:conneections:thederivedmodstack}).

\subsection{The Euler form}
\label{subsection:Eulerformconnections}
The Euler form for the category of connections on the smooth projective curve $C$ is given by $\chi((\CF,\nabla),(\CG,\nabla))=2(1-g)\rank\CF\rank\CG$. It is therefore locally constant on the product $\FM^{\dR}(C)\times\FM^{\dR}(C)$. A straightforward way to see this is through the Riemann--Hilbert correspondence which gives a equivalences of categories between the dg-category containing some simple connections and the dg-category containing the corresponding representations of the fundamental group of the curve (as a consequence of the fact that the Riemann--Hilbert correspondence lifts to the dg-categories). Then, we use the computation of the Euler form for the fundamental group algebra \S\ref{subsubsection:theeulerformfdga}.

A more intrinsic way to compute the Euler form (i.e. without appealing to the Riemann--Hilbert correspondence) follows from \cite{simpson1992higgs}. Indeed, the Euler form is given by $\chi_{\dR}(\CF^{\vee}\otimes\CG)=\rank(\CF)\rank(\CG)\chi_{\dR}(C)$ (the Euler characteristic of the hypercohomology \eqref{equation:tangentextensions}) and by the comparison between the de Rham cohomology and singular cohomology of the curve $C$, $\chi_{\dR}(C)=2(1-g)$.

\subsection{Resolution of the RHom complex over the classical moduli stack}
\label{subsection:resolutionRHomdR}
We explain how to resolve globally the cotangent complex (or equivalently, the tangent complex) by a $3$-term complex of vector bundles over $\FM^{\dR}(C)\times\FM^{\dR}(C)$. Such resolutions are used to define the CoHA product at the level of mixed Hodge modules \cite{davison2022BPS}.

\begin{proposition}
\label{proposition:resolutionRHomdR}
 Let $r,s\in \BoN$. Then, there exists a $3$-term complex of vector bundles over $\FM_r^{\dR}(C)\times\FM_{s}^{\dR}(C)$, concentrated in degrees $[-1,1]$, and quasi-isomorphic to $\RHom[1]$.
\end{proposition}
\begin{proof}
We give a sketch of the argument, as it is very similar to the case of sheaves on surfaces, \cite[\S6.1.3]{davison2022BPS} or \cite[\S4.3]{kapranov2019cohomological}.

Let $(\CF,\nabla)$ be a rank $r$ vector bundle with flat connection, which we see as an $\CO_C$-coherent $\Lambda^{\dR}=\CD_C$-module. We let $N_1\ll0$ be such that we have an epimorphism
\[
  \CO(N_1)\otimes_{\BoC}\Hom_{\CO_C}(\CO_{C}(N_1),\CF)\cong\CO(N_1)^{P(N_1)}\xrightarrow{\tilde{f}}\CF,
\]
where $P(N_1)\coloneqq \dim_{\BoC}\Hom_{\CO_C}(\CO_C(N_1),\CF)$. It induces an epimorphism $\Lambda^{\dR}\otimes_{\CO_C}\CO(N_1)^{P(N_1)}\xrightarrow{f}\CF$ using the $\Lambda^{\dR}$-structure of $\CF$. The locally projective $\Lambda^{\dR}$-module $\Lambda^{\dR}\otimes_{\CO_C}\CO(N_1)^{P(N_1)}$ is filtered by $\CO_C$-coherent submodules $\Lambda^{\dR}_{\leq k}\otimes_{\CO_C}\CO(N_1)^{P(N_1)}$. By local Noetherianity of $\Lambda^{\dR}$ \cite[Theorem 2.6.11 and Proof]{hotta2007d}, there is a $k\in\BoN$ such that $K_{\leq k}\coloneqq\ker(f_{\Lambda^{\dR}_{\leq k}\otimes_{\CO_C}\CO(N_1)^{P(N_1)}})$ generates $\ker(f)$ as a $\Lambda^{\dR}$-module, where $f_{\Lambda^{\dR}_{\leq k}\otimes_{\CO_C}\CO(N_1)^{P(N_1)}}$ denotes the restriction of $f$. By boundedness of the family of vector bundles underlying a flat rank $r$ connection \cite[Proposition 3.5]{simpson1994moduli}, we can choose $k$ uniform over the stack of rank $r$ connections. The coherent sheaf $K_{\leq k}$ has constant rank and degree when $(\CF,\nabla)$ varies over rank $r$ connection and we let $N_2\ll N_1$ such that we have an epimorphism
\[
 \CO(N_2)\otimes_{\BoC}\Hom_{\CO_C}(\CO(N_2),K_{\leq k})\cong\CO(N_2)^{Q(N_2)}\xrightarrow{\tilde{g}} K_{\leq k},
\]
where $Q(N_2)=\dim_{\BoC}(\CO(N_2),K_{\leq k})$.

This map induces an epimorphism, again using the $\Lambda^{\dR}$-structure on $\ker(f)$,
\[
 \Lambda^{\dR}\otimes_{\CO_C}\CO(N_2)^{Q(N_2)}\xrightarrow{g}\ker(f).
\]
Letting $K'=\ker(g)$, we have an exact sequence
\[
 0\rightarrow K'\rightarrow \Lambda^{\dR}\otimes\CO(N_2)^{Q(N_2)}\rightarrow\Lambda^{\dR}\otimes\CO(N_1)^{P(N_1)}\rightarrow\CF\rightarrow 0
\]
of $\Lambda^{\dR}$-modules giving a locally projective resolution of $\CF$.
By applying the functor $\Hom_{\Lambda^{\dR}}(-,\CG)$, we obtain the complex
\[
 \Hom_{\Lambda^{\dR}}(\Lambda^{\dR}\otimes\CO(N_1)^{P(N_1)},\CG)\rightarrow\Hom_{\Lambda^{\dR}}(\Lambda^{\dR}\otimes\CO(N_2)^{Q(N_2)},\CG)\rightarrow\Hom_{\Lambda^{\dR}}(K',\CG)
\]
which is quasi-isomorphic to the complex $\RHom(\CF,\CG)$. Moreover, we can rewrite this complex as
\[
 \Hom_{\CO_C}(\CO(N_1)^{P(N_1)},\CG)\rightarrow\Hom_{\CO_C}(\CO(N_2)^{Q(N_2)},\CG)\rightarrow\Hom_{\Lambda^{\dR}}(K',\CG)
\]
using that $\Hom_{\Lambda^{\dR}}(\Lambda^{\dR}\otimes\CO(N_1)^{P(N_1)},\CG)\cong \Hom_{\CO_C}(\CO(N_1)^{P(N_1)},\CG)$ (and similarly for $\CO(N_2)$). This construction is uniform over $(\CF,\CG)\in\CM^{\dR}_r\times\CM^{\dR}_s$ for $r,s\in\BoN$ and therefore provides us with a $3$-term complex over $\CM^{\dR}_r\times\CM^{\dR}_s$ quasi-isomorphic to the $\RHom$ complex. Since we can choose $N_1, N_2$ arbitrarily negative, for any connected component $\CM_r^{\dR}\times\CM_s^{\dR}$ of $\CM^{\dR}\times\CM^{\dR}$, the RHom complex is represented by a $3$-term complex of vector bundles. Indeed, for $N_1,N_2\ll0$, $\dim\Hom_{\CO_C}(\CO(N_1)^{P(N_1)},\CG)$ and $\dim\Hom_{\CO_C}(\CO(N_1)^{P(N_1)},\CG)$ do not depend on $\CG$ in a fixed component of $\FM^{\dR}$ and for $(\CF,\CG)\in\FM_r^{\dR}\times\FM_s^{\dR}$, the Euler characteristic of the complex is equal to the Euler form $\chi(\CF,\CG)=2(1-g)rs$, so $\dim_{\BoC} \Hom_{\Lambda^{\dR}}(K',\CG)$ does not depend on $\CF,\CG$ of respective ranks $r$ and $s$.
\end{proof}
We see that the discussion is very similar to that for sheaves on surfaces or Higgs sheaves \cite[\S6.1.3]{davison2022BPS}. The discussion of \cite[\S6.1.4]{davison2022BPS} extends to this situation and gives the necessary properties for the associativity of the CoHA product \cite[Assumption 3]{davison2022BPS}.

\subsection{Sigma-collections}
\label{subsection:Sigmacollections}
A $\Sigma$-object in a $\BoC$-linear dg-category $\SC$ is an object $\CF$ such that
\[
 \dim_{\BoC}\HO^i(\Hom^{\bullet}(\CF,\CF))=
 \left\{
 \begin{aligned}
  1\quad&\text{ if $i=0,2$}\\
  2g\quad& \text{ if $i=1$}\\
  0\quad&\text{ otherwise}
 \end{aligned}
 \right.
\]
for some $g\geq 0$.

A Sigma-collection is a finite collection of pairwise orthogonal $\Sigma$-objects: $(\CF_1,\hdots,\CF_t)$ with $\Hom_{\SC}^0(\CF_i,\CF_j)=0$ for $i\neq j$.

\begin{lemma}
\label{lemma:sigmacollection_connections}
 A collection of non-isomorphic irreducible connections on a smooth projective curve is a $\Sigma$-collection.
\end{lemma}
\begin{proof}
 This follows directly from the Riemann--Hilbert correspondence and the fact that the analogous statement is satisfied for collections of irreducible local systems on a smooth projective curve (by the 2CY property \S\ref{subsubsection:categorieslocsys}).
\end{proof}

\subsection{The direct sum map}
\begin{proposition}
 The direct sum map $\oplus\colon\CM^{\dR}(C)\times\CM^{\dR}(C)\rightarrow\CM^{\dR}(C)$ is finite.
\end{proposition}
\begin{proof}
We rely on the classical nonabelian Hodge isomorphism (more precisely, the Riemann--Hilbert correspondence), which gives a diagram where the vertical arrows are analytic isomorphisms
 \[
  \begin{tikzcd}
  	& {(\CM^{\dR}(C))^{\an}\times(\CM^{\dR}(C))^{\an}} & {(\CM^{\dR}(C))^{\an}} \\
	{} & {(\CM_g^{\Betti})^{\an}\times(\CM_g^{\Betti})^{\an}} & {(\CM_g^{\Betti})^{\an}}
	\arrow["\oplus", from=2-2, to=2-3]
	\arrow["\Phi", from=1-3, to=2-3]
	\arrow["\Phi\times\Phi"', from=1-2, to=2-2]
	\arrow["\oplus", from=1-2, to=1-3]
\end{tikzcd}.
 \]
 The morphism $\oplus\colon\CM^{\Betti}\times\CM^{\Betti}\rightarrow\CM^{\Betti}$ is finite (\S\ref{subsubsection:directsumBetti}), and therefore, $\oplus^{\an}\colon(\CM^{\dR}(C))^{\an}\times(\CM^{\dR}(C))^{\an}\rightarrow(\CM^{\dR}(C))^{\an}$ is finite. Consequently, so is $\oplus\colon\CM^{\dR}\times\CM^{\dR}\rightarrow\CM^{\dR}$. 
\end{proof}

\section{Cohomological Hall algebras for Betti, de Rham and Dolbeault moduli stacks}
In this section, we first recall the main result of \cite{davison2022BPS} concerning the BPS algebras and the cohomological Hall algebras of the Betti and de Rham moduli stacks (Theorem \ref{theorem:dhsBettiDolbeault}), and we carry out the analogous study for the de Rham moduli stack (Theorem \ref{theorem:CoHAdeRham}).

\subsection{The determinant line bundle on a quotient stack}
\label{subsection:thedetlb}
Let $\FX=X/G$ be a quotient stack where $G$ is a linear algebraic group and $X$ a $G$-variety. For any character $\chi\colon G\rightarrow\BoC^*$, we obtain a morphism
\[
 \tilde{\chi}\colon\FX\rightarrow\FX\times\rmB\BoC^*
\]
which is the identity on the first component and the morphism induced by $X\rightarrow \pt$, and $\chi$ on the second component. We have $\tilde{\chi}^*(\BD\BoQ_{\FX}^{\vir}\boxtimes\BoQ_{\rmB\BoC^*})\cong \BD\BoQ_{\FX}^{\vir}$ (where in this abstract context, $\vir$ is just a certain Tate twist). By adjunction and pushing forward to the good moduli space $f\colon \FX\rightarrow\CX$ (which we assume exists), we obtain the action
\[
 f_*\BD\BoQ_{\FX}^{\vir}\otimes \HO^*_{\BoC^*}\rightarrow f_*\BD\BoQ_{\FX}^{\vir}.
\]
In our cases of interest, we have $G=\GL_r$ for some $r\in\BoN$ and the choosen character is the determinant $\det\colon \GL_r\rightarrow\BoC^*$.

\subsection{CoHA and BPS algebra}
\label{subsection:CoHAandBPSalgebra}
In \cite{davison2020bps}, \cite{davison2021purity} and \cite{davison2022BPS}, it was explained how to fruitfully exploit the good moduli space morphism $\JH\colon\FM_{\CA}\rightarrow\CM_{\CA}$ of the stack of object in a $2$-Calabi--Yau category $\CA$ to study the CoHA of this category. For the sake of efficiency, we only recall the definitions for the Dolbeault and Betti stacks here (the case of general $2$-Calabi--Yau categories is treated in great detail in \cite{davison2022BPS}). We let $\sharp\in\{\Betti,\Dol\}$. The complexes of mixed Hodge modules $\underline{\SA}^{\sharp}\coloneqq\JH_*^{\sharp}\BD\underline{\BoQ}_{\FM^{\sharp}}^{\vir}\in\CD^+(\MHM(\CM^{\sharp}))$ admit a (relative) CoHA product \cite[Theorem 1.1]{davison2022BPS}. Here and in general, the superscript $\vir$ denotes a Tate twist related to the virtual dimension of the stack. More precisely, for $r\in\BoN$, $\underline{\BoQ}_{\FM^{\sharp}_r}^{\vir}=\underline{\BoQ}_{\FM^{\sharp}_r}\otimes\BoL^{(1-g)r^2}$. The complex of mixed Hodge modules $\underline{\SA}^{\sharp}$ is concentrated in nonnegative cohomological degrees \cite[Lemma 10.1]{davison2022BPS}, so that $\underline{\BPS}_{\Alg}^{\sharp}\coloneqq\CH^{0}(\underline{\SA}^{\sharp})\in\MHM(\CM^{\sharp})$ inherits a multiplication. It is called the \emph{relative BPS algebra}. By taking derived global sections, we obtain genuine algebras $\HO^*(\underline{\SA}^{\sharp})$ and $\rmBPS_{\Alg}^{\sharp}$, the latter being a subalgebra of the former. Moreover, the determinant line bundle on $\FM^{\sharp}$ (\S\ref{subsection:thedetlb}) gives an action of $\HO^*_{\BoC^*}$ on $\underline{\SA}^{\sharp}$. As it is now clear from \cite{hennecart2022geometric, davison2022BPS, davison2023bps}, it is crucial to twist the multiplication naturally defined on $\underline{\SA}^{\sharp}$ by a bilinear form $\Psi$. We recall for the reader's convenience that in the framework of quantum groups of Kac--Moody type, this twist exchanges the enveloping algebra and the specialisation at $q=-1$ of the quantum group \cite{hennecart2022geometric}. In our context, when one wants to describe the BPS algebra as the enveloping algebra of some Lie algebra, this is absolutely crucial (the specialisation at $-1$ of the quantum group is not naturally in general the enveloping algebra of any Lie algebra). We briefly recall how this twist works (in our very specific situations).

We choose a bilinar form
\[
 \Psi\colon\BoZ\times\BoZ\rightarrow\BoZ
\]
such that
\[
 \Psi(r,s)+\Psi(s,r)\equiv \chi(r,s)=2(1-g)rs\pmod{2}
\]
where $\chi$ is the Euler form of the category of local systems on the curve $C$ of genus $g$, or of semistable Higgs bundles of slope $0$. We may therefore very well take $\Psi=0$ here (i.e. forget about the twists); a more natural choice would be $\Psi(r,s)=(1-g)rs$ (the Euler form of the category of semistable slope zero vector bundles on $C$). If
\[
 m_{r,s}\colon\underline{\SA}^{\sharp}_r\boxdot\underline{\SA}^{\sharp}_s\rightarrow\SA_{r+s}^{\sharp}
\]
is the $(r,s)$ component of the multiplication of $\underline{\SA}^{\sharp}$, then the $\Psi$-twisted multiplication is given by $m^{\Psi}_{r,s}\coloneqq (-1)^{\Psi(r,s)}m_{r,s}$.

In this paper, a bilinear form $\Psi$ as above is fixed once and for all and although it does not appear in the notation, we consider the $\Psi$-twisted multiplication on $\underline{\SA}^{\sharp}$ (this is justified as $\Psi=0$ is a legitimate choice here, and the reader may wish to make this choice).

\subsection{Local systems and semistable Higgs bundles}
The (relative) cohomological algebras for the stacks of (possibly) twisted local system and Higgs bundles over a smooth projective curve were defined and studied in depth in \cite{davison2022BPS}. One of the motivations of \cite{davison2022BPS} was to obtain a nonabelian Hodge isomorphism between the Borel--Moore homologies of the Betti and Dolbeault moduli spaces, motivated by the $\mathrm{P}=\mathrm{W}$ conjecture and its variants for intersection cohomology of the moduli spaces and Borel--Moore homology for the stacks.

We recall here our main results.

\begin{theorem}[{\cite[Theorem 1.4+Theorem 1.7]{davison2022BPS}}]
\label{theorem:dhsBettiDolbeault}
Let $C$ be a curve of genus $\geq 2$.
\begin{enumerate}
 \item For $\sharp\in\{\Dol,\Betti\}$, we have an idenfication of the relative BPS algebra $\underline{\BPS}_{\Alg}^{\sharp}\cong\Free_{\boxdot-\Alg}\left(\bigoplus_{r\geq 1}\underline{\IC}(\CM_r^{\sharp})\right)$; a relative PBW isomorphism $\Sym_{\boxdot}\left(\underline{\BPS}_{\Lie}^{\sharp}\otimes\HO^*_{\BoC^*}\right)\cong \underline{\SA}^{\sharp}$ where $\underline{\BPS}_{\Lie}^{\sharp}\coloneqq\Free_{\boxdot-\Lie}\left(\bigoplus_{r\geq 1}\underline{\IC}(\CM_{r}^{\sharp})\right)$; and absolute versions of these two statements $\rmBPS_{\Alg}^{\sharp}\cong\Free_{\Alg}\left(\bigoplus_{r\geq 1}\ICA(\CM_r^{\sharp})\right)$ and  $\HO^*\SA^{\sharp}\cong\Sym\left(\rmBPS_{\Lie}^{\sharp}\otimes\HO^*_{\BoC^*}\right)$ where $\rmBPS_{\Lie}^{\sharp}\coloneqq \HO^*(\BPS_{\Lie}^{\sharp})$.
 \item Let $\Xi\colon \CM^{\Betti}\rightarrow\CM^{\Dol}$ be the homeomorphism given by nonabelian Hodge theory. We have isomorphisms of Lie algebra objects $\Xi_*\BPS_{\Lie}^{\Betti}\cong\BPS_{\Lie}^{\Dol}$, of constructible complexes $\Xi_*\SA^{\Betti}\cong\SA^{\Dol}$ and therefore isomorphisms of the BPS cohomologies and Borel--Moore homologies
 \[
  \rmBPS_{\Lie}^{\Betti}\cong \rmBPS_{\Lie}^{\Dol}, \quad \HO^{\rmBM}_*(\FM^{\Betti})\cong\HO^{\rmBM}_*(\FM^{\Dol}).
 \]
\end{enumerate}
\end{theorem}
The case of curves of genus $\leq 1$ was considered first in \cite{davison2023nonabelian}. Note that neither \cite{davison2023nonabelian} nor \cite{davison2022BPS} take the CoHA structures into account when comparing $\HO^{\rmBM}_*(\FM^{\Betti})$ and $\HO^{\rmBM}_*(\FM^{\Dol})$.

\subsection{Connections}
\label{subsection:connections}
In this section, we explain how to obtain the cohomological Hall algebra structure for the stack of connections over a smooth projective curve following the formalism developped in \cite{davison2022BPS}.

From our verification of the \cite[Assumptions 1-6]{davison2022BPS} in \S\ref{section:connections} regarding the stack of connections on a smooth projective curve, one can apply \cite[Theorem 1.4+Theorem 1.5]{davison2022BPS} to the category of connections. We then obtain the following (compare Theorem \ref{theorem:CoHAdeRham} with Theorem \ref{theorem:dhsBettiDolbeault} concerning Betti and Dolbeault stacks).

\begin{theorem}
\label{theorem:CoHAdeRham}
 The complex of mixed Hodge modules $\underline{\SA}^{\dR}\coloneqq \JH_*^{\dR}\BD\underline{\BoQ}_{\FM^{\dR}}^{\vir}\in\CD^+(\MHM(\CM^{\dR}))$ is pure and admits a relative CoHA product. This makes this complex an associative unital algebra object in $(\CD^+(\MHM(\CM^{\dR})),\boxdot)$. This complex is concentrated in nonnegative cohomological degrees and the degree $0$ cohomology mixed Hodge module $\underline{\BPS}_{\Alg}^{\dR}\coloneqq\CH^{0}\underline{\SA}^{\dR}$ has an induced algebra structure, making it an algebra object in the tensor category $(\MHM(\CM^{\dR}),\boxdot)$ (obtained as in \S\ref{subsubsection:monoidalstrucpoint} for the monoid $(\CM^{\dR},\oplus)$).
\end{theorem}
We adopt (implicitly) the same bilinear form $\Psi$ as in \S\ref{subsection:CoHAandBPSalgebra} to twist the product on $\underline{\SA}^{\dR}$. Again, $\Psi=0$ is a possible choice.
\begin{proof}
 The category of connections on a smooth projective curve satisfies the Assumptions 1--3 of \cite{davison2022BPS}, as explained in \S\ref{section:connections}. We write the following diagram for reference:
 \[
  \begin{tikzcd}
	{\FM^{\dR}\times\FM^{\dR}} & {\mathfrak{Exact}^{\dR}} & {\FM^{\dR}} \\
	{\CM^{\dR}\times\CM^{\dR}} && {\CM^{\dR}}
	\arrow["{q^{\dR}}"', from=1-2, to=1-1]
	\arrow["{p^{\dR}}", from=1-2, to=1-3]
	\arrow["{\JH^{\dR}\times\JH^{\dR}}"', from=1-1, to=2-1]
	\arrow["{\JH^{\dR}}", from=1-3, to=2-3]
	\arrow["\oplus"', from=2-1, to=2-3]
\end{tikzcd}
 \]
As in \cite{davison2022BPS}, one can construct a virtual pullback map
 \begin{equation}
 \label{equation:virpbdR}
  \BD\underline{\BoQ}_{\FM^{\dR}\times\FM^{\dR}}\rightarrow q^{\dR}_*\BD\underline{\BoQ}_{\mathfrak{Exact}}\otimes \BoL^{\vrank(\CC^{\dR})}
 \end{equation}
and a pushforward map
\begin{equation}
\label{equation:pfdR}
 p^{\dR}_*\BD\underline{\BoQ}_{\mathfrak{Exact}^{\dR}}\rightarrow\BD\underline{\BoQ}_{\FM^{\dR}}.
\end{equation}
By applying $\oplus_*(\JH^{\dR}\times\JH^{\dR})_*$ to \eqref{equation:virpbdR}, $\JH^{\dR}_*$ to \eqref{equation:pfdR} and composing both together, with a Tate twist, they combine together to give the CoHA product
\[
 m\colon \underline{\SA}^{\dR}\boxdot\underline{\SA}^{\dR}\rightarrow\underline{\SA}^{\dR}.
\]
The fact that $\underline{\SA}^{\dR}$ is concentrated in nonnegative cohomological degrees is a consequence of \S\ref{subsection:Sigmacollections} (more precisely, Lemma~\ref{lemma:sigmacollection_connections}) and the fact that $\JH^{\dR}$ is a good moduli space (so Assumptions 4 and 6 of \cite{davison2022BPS}) are satisfied). Indeed, we can use the neighbourhood theorem \cite[Theorem 5.11]{davison2021purity} and conclude (as in \cite[Lemma~10.1]{davison2022BPS}), using by the fact that the analogous statement for preprojective algebras of quivers is known to be true (see \cite[Theorem 6.1]{davison2021purity} for the proof). The fact that $\underline{\BPS}_{\Alg}^{\dR}$ has an induced algebra structure now follows formally, as in \cite{davison2022BPS}, or \cite[\S7.6.3]{davison2022BPS}, or \cite{davison2020bps} in the case of preprojective algebras. We only use that $\oplus$ is a finite morphism (and so $\oplus_*$ preserves the categories of mixed Hodge modules) and that $\underline{\SA}^{\dR}$ is concentrated in nonnegative cohomological degrees. The comparison with preprojective algebras also gives the purity of the complex of mixed Hodge modules $\underline{\SA}^{\dR}$, as in \cite[Theorem~6.1]{davison2021purity}.
\end{proof}

\begin{theorem}
\label{theorem:freenessPBWconnections}
Let $C$ be a curve of genus $g\geq 2$.
 \begin{enumerate}
  \item The relative BPS algebra is isomorphic to the free algebra generated by the intersection complexes of the good moduli space:
  \[
   \underline{\BPS}^{\dR}_{\Alg}\cong\Free_{\boxdot-\Alg}\left(\bigoplus_{r\geq 1}\underline{\IC}(\CM_r^{\dR})\right)
  \]
  \item We have a relative PBW isomorphism:
  \[
   \Sym_{\boxdot}\left(\underline{\BPS}_{\Lie}^{\dR}\otimes\HO^*_{\BoC^*}\right)\rightarrow\underline{\SA}^{\dR}
  \]
  where the relative BPS Lie algebra is defined as $\underline{\BPS}_{\Lie}^{\dR}\coloneqq\Free_{\boxdot-\Lie}\left(\bigoplus_{r\geq 1}\underline{\IC}(\CM_r^{\dR})\right)$.
\end{enumerate}
\end{theorem}

\begin{proof}
 The category of connections on a smooth projective curve of genus $\geq 2$ is an Abelian category satisfying the Assumptions 1-6 of \cite{davison2022BPS}. Moreover, it is totally negative in the sense of \cite[\S7]{davison2022BPS} by \S\ref{subsection:Eulerformconnections}. We may therefore apply \cite[Theorem 1.4]{davison2022BPS} to get $(1)$ and \cite[Theorem 1.5]{davison2022BPS} to get $(2)$.
\end{proof}

By taking derived global sections of the isomorphisms in Theorem \ref{theorem:freenessPBWconnections}, we obtain the following.
\begin{corollary}
\label{corollary:absolutedR}
\begin{enumerate}
\item We have the absolute version of the description of the BPS algebra:
  \[
   \rmBPS_{\Alg}^{\dR}\cong\Free_{\Alg}\left(\bigoplus_{r\geq 1}\ICA(\CM_r^{\dR})\right).
  \]
  \item We have the absolute PBW isomorphism
  \[
   \HO^*\SA^{\dR}\cong\Sym\left(\rmBPS_{\Lie}^{\dR}\otimes\HO^*_{\BoC^*}\right).
  \]
 \end{enumerate}
\end{corollary}

\subsection{Genus one}
For completeness, we briefly explain the situation in genus one. Let $C$ be a smooth projective curve of genus one. We let
\[
\begin{matrix}
 \Delta_r&\colon&\CM_{1}^{\dR}&\rightarrow&\CM_r^{\dR}\\
         &      &x&\mapsto&x^{\oplus r}
\end{matrix}
\]
be the small diagonal inside the de Rham moduli space.

\begin{theorem}
 \label{theorem:genusone}
 Let $C$ be a smooth projective curve of genus one. Then, we have the following.
 \begin{enumerate}
  \item We have an isomorphism of algebras
  \[
   \underline{\BPS}_{\Alg}^{\dR}=\Sym_{\boxdot}\left(\bigoplus_{r\geq 1}(\Delta_r)_*\underline{\IC}(\CM_1^{\dR})\right).
  \]
  \item We let $\underline{\BPS}_{\Lie}^{\dR}\coloneqq\bigoplus_{r\geq 1}(\Delta_r)_*\underline{\IC}(\CM_1^{\dR})$. It is a Lie algebra object in $\MHM(\CM^{\dR})$ for the trivial (i.e. vanishing) Lie bracket. We have a PBW isomorphism
  \[
   \underline{\SA}^{\dR}=\Sym_{\boxdot}\left(\underline{\BPS}_{\Lie}^{\dR}\otimes\HO^*_{\BoC^*}\right).
  \]
 \end{enumerate}
\end{theorem}
\begin{proof}
 The category of connections on a genus one curve is totally isotropic (i.e. its Euler form vanishes), \S\ref{subsection:Eulerformconnections}. Therefore, the result follows from \cite[\S6.4.2]{davison2023bps}. It is also possible to adapt the proof of \cite[\S\S4.4,5.4]{davison2023nonabelian} to the case of connections.
\end{proof}

\section{Cohomological Hall algebra for the Hodge moduli stack}
\label{section:CoHAHodge}
\subsection{The Hodge moduli stack}
\label{subsection:Hodgemodulistack}
The categories of $\lambda$-connections for $\lambda\in\BoC$ were imagined by Deligne to interpolate between flat connections and Higgs bundles. The  moduli space of $\bm{\lambda}$-connections gives a way to understand the homeomorphism between the de Rham and the Dolbeault moduli spaces. In this section, we recall its principal features, following \cite{simpson1996hodge}. We only work on smooth projective curves, which had the advantage of simplifying slightly the definitions since the integrability conditions are automatic.

\subsubsection{$\lambda$-connections}
Let $C$ be a smooth projective curve. For $\lambda\in\BoC$, a $\lambda$-connection on $C$ is a pair $(\CF,\nabla)$ of a vector bundle $\CF$ together with a morphim of sheaves
\[
 \nabla\colon \CF\rightarrow \CF\otimes_{\CO_C} \Omega^1_C
\]
satisfying the $\lambda$-twisted Leibniz rule $\nabla(fe)=f\nabla(e)+\lambda e\otimes df$ for any (local) section $e$ of $\CF$ and any local regular function $f$ on $C$. For $\lambda=0$, we obtain Higgs bundles and for $\lambda=1$, usual connections.

Let $p\in C$ be a fixed closed point. A framed $\lambda$-connection of rank $r$ is a triple $(\CF,\nabla,\beta)$ where $\CF$ is a vector bundle of rank $r$, $(\CF,\nabla)$ is a $\lambda$-connection and $\beta\colon \CF_p\rightarrow \BoC^r$ is an isomorphism. We adopt the following convenient terminology: a $\bm{\lambda}$-connection (with bold symbol) is a $\lambda$-connection for some $\lambda\in\BoC$. By \cite[Proposition 4.1]{simpson1996hodge}, which relies on and follows from the formalism developped in \cite{simpson1994moduli,simpson1994moduliII}, there exists a scheme $R^{\Hod}_r(C,p)\rightarrow\BoA^1$ parametrising framed $\bm{\lambda}$-connections and the morphism to $\BoA^1$ remembers the parameter $\lambda$ of the $\bm{\lambda}$-connection. This scheme is acted on by $\GL_r$ by change of framing, and the geometric invariant theory quotient $\CM^{\Hod}_r(C,p)\rightarrow\BoA^1$ is a universal categorical quotient, so that the fiber over $0$ coincides with the Dolbeault moduli space $\CM_r^{\Dol}(C)$ and the fiber over $1$ coincides with the de Rham moduli space $\CM^{\dR}_r(C)$. There is a natural $\BoC^*$-action on $R^{\Hod}_r(C,p)$, given by $t\cdot(E,\nabla,\beta)=(E,t\nabla,\beta)$. It commutes with the $\GL_r$-action. Therefore, it induces a $\BoC^*$-action of $\CM^{\dR}(C)$ covering the natural $\BoC^*$-action on $\BoA^1$.

The stack-theoretic quotient $\FM^{\Hod}_r(C)\coloneqq R^{\Hod}_r(C,p)/\GL_r\rightarrow \BoA^1$ is the stack of $\bm{\lambda}$-connections. It is an Artin stack over $\BoA^1$ sending a $\BoA^1$-scheme $\lambda_S\colon S\rightarrow\BoA^1$ to the groupoid of pairs $(\CF,\nabla)$ where
\begin{enumerate}
 \item $\CF$ is a vector bundle on $C_S=C\times S$,
 \item $\nabla\colon \CF\rightarrow\CF\otimes_{\CO_{C_S}}\Omega^1_{C_S/S}$ is a $\lambda_S$-connection.
\end{enumerate}

We have a commutative diagram
\[
 \begin{tikzcd}
	{\FM_r^{\Hod}} & {} & {\CM_r^{\Hod}} \\
	& {\BoA^1}
	\arrow["{\JH_r^{\Hod}}", from=1-1, to=1-3]
 	\arrow["{\pi_{\FM^{\Hod}_r}}"', from=1-1, to=2-2]
	\arrow["{\pi_{\CM^{\Hod}_r}}", from=1-3, to=2-2]
\end{tikzcd}
\]
and $\JH^{\Hod}_r$ is a good moduli space for $\FM^{\Hod}_r(C)$.

For an arbitrary $\lambda\in\BoC$, the fiber over $\lambda$ of this diagram gives the map
\[
 \JH^{\Hod}_{r,\lambda}\colon\FM^{\Hod}_{r,\lambda}\rightarrow\CM^{\Hod}_{r,\lambda}
\]
from the stack of $\lambda$-connections to the the good moduli space.

Using the $\BoC^*$-action, for any $\lambda\neq\mu\in\BoC\setminus\{0\}$, we have a commutative diagram where the horizontal maps are isomorphisms
\[
 \begin{tikzcd}
	{\FM^{\Hod}_{\lambda}} & {\FM^{\Hod}_{\mu}} \\
	{\CM^{\Hod}_{\lambda}} & {\CM^{\Hod}_{\mu}}
	\arrow["{\JH^{\Hod}_{\lambda}}"', from=1-1, to=2-1]
	\arrow["\cong", from=1-1, to=1-2]
	\arrow["\cong", from=2-1, to=2-2]
	\arrow["{\JH^{\Hod}_{\mu}}", from=1-2, to=2-2]
\end{tikzcd}.
\]

\subsubsection{The sheaf of differential operators}
\label{subsubsection:sheafdiffoperators}
We recall the \emph{deformation to the associated graded} of \cite[p. 86]{simpson1994moduli} used to construct the moduli stack and moduli space of 
$\bm{\lambda}$-connections.

We let $\lambda$ be a coordinate on $\BoA^1$, and we define $\Lambda^{\Hod}$ be the sheaf of $\CO_{C\times\BoA^1}$-algebras over $C\times\BoA^1$ defined as the subsheaf of $\pr_1^*\Lambda^{\dR}$ generated by section of the form $\sum\lambda^is_i$ where $s_i$ are sections of $\Lambda^{\dR}$.

This sheaf of algebra interpolates between $\Lambda^{\dR}$ and $\Lambda^{\Dol}$. Namely, for any $t\neq 0$, $\Lambda^{\Hod}_{C\times\{t\}}\cong\Lambda^{\dR}$ while for $t=0$, $\Lambda^{\Hod}_{C\times\{0\}}\cong \Lambda^{\Dol}$.

The natural filtration of $\Lambda^{\dR}$ by the order of differential operators induces a filtration of $\Lambda^{\Hod}=\bigcup_{l\geq 0}\Lambda^{\Hod}_{\leq l}$ by coherent $\CO_{C\times\BoA^1}$-submodules. The associated graded of $\Lambda^{\Hod}$ with respect to this filtration is isomorphic to $(\pi_{\Tan^*C})_*\CO_{\Tan^*C}\boxtimes \CO_{\BoA^1}$ where $\pi_{\Tan^*C}\colon\Tan^*C\rightarrow C$ is the cotangent projection.

\subsubsection{Preferred trivialization}
\label{subsubsection:preferredtrivialization}
There is a canonical way to trivialize the Hodge moduli space $\CM_{r}^{\Hod}\rightarrow\BoA^1$ in the category of topological spaces over $\BoA^1$. It is explained in \cite[pp. 20-21]{simpson1996hodge} and is called the trivialization \emph{via preferred sections}. It provides us with homeomorphisms $\CM^{\Hod}_r\cong\CM_r^{\dR}\times\BoA^1$ over $\BoA^1$ for any $r\in\BoN$. We refer to \emph{loc.cit.} for more details. The advantage is that it gives compatible trivializations for all $r\in \BoN$, in the sense that, by their canonicity, they are compatible with the direct sum: for any $r,s\in\BoN$, the diagram in the category of topological spaces
\[
 \begin{tikzcd}
	{\CM_r^{\Hod}\times_{\BoA^1}\CM_s^{\Hod}} & {\CM_{r+s}^{\Hod}} \\
	{\CM_r^{\dR}\times\CM_s^{\dR}\times\BoA^1} & {\CM_{r+s}^{\dR}\times\BoA^1}
	\arrow["\oplus", from=1-1, to=1-2]
	\arrow[from=1-2, to=2-2]
	\arrow[from=1-1, to=2-1]
	\arrow["{\oplus\times\id_{\BoA^1}}"', from=2-1, to=2-2]
\end{tikzcd}
\]
commutes, where vertical arrows are given by this trivialization. When considering trivializations of the Hodge moduli space, we will always refer to this preferred trivialization, as the compatibility with direct sum is important for CoHAs considerations.

\subsubsection{Stack of extensions and RHom complex}
In this section, we explain how to obtain the RHom complex over $\FM^{\Hod}\times\FM^{\Hod}$. This complex interpolates between the RHom complexes for the de Rham and the Dolbeault moduli stacks (\S \ref{subsection:resolutionRHomdR} and \S \ref{subsubsection:RHomHiggs} respectively). We give a global presentation as a $3$-term complex of vector bundles, a crucial property for the construction of the CoHA product in \S\ref{subsection:CoHAstructure} following the formalism of \cite{davison2022BPS}. The following lemma is the version over the Hodge moduli stack of Proposition \ref{proposition:resolutionRHomdR}.

\begin{lemma}
\label{lemma:3termHodge}
 The RHom complex over $\FM^{\Hod}\times\FM^{\Hod}$ admits a global resolution by a $3$-term complex of vector bundles. We let $\CC^{\Hod}\coloneqq\RHom[1]$ be (the shift of) such a resolution. Moreover, the complexes $(\imath_{\FM^{\Dol}}\times\imath_{\FM^{\Dol}})^*\CC^{\Hod}$ and $\CC^{\Dol}$ on the one hand and $(\imath_{\FM^{\dR}}\times\imath_{\FM^{\dR}})^*\CC^{\Hod}$ and $\CC^{\dR}$ on the other hand, are quasi-isomorphic.
\end{lemma}
\begin{proof}
We proceed as in \S\ref{subsection:resolutionRHomdR}, with litterally the same proof (except that we work over $\BoA^1$). We let $(\CF,\nabla)$ be a vector bundle with $\bm{\lambda}$-connection over $C$, i.e. a vector bundle on $C\times\BoA^1$ with a $\Lambda^{\Hod}$-module structure (\S\ref{subsubsection:sheafdiffoperators}). We denote by $\pr_i$, $i=1,2$ the projections from $C\times\BoA^1$ on the first and second factors. If $\CO_C(N)$ is a line bundle on $C$, we denote by $\CO_{C\times\BoA^1}(N)$ the pullback $\pr_1^*\CO_C(N)$. We let $N_1\ll0$, so that $(\pr_2)_*\intHom_{\CO_{C\times\BoA^1}}(\CO_{C\times\BoA^1}(N_1),\CF)$ is a (necessarily trivial by contractibility of $\BoA^1$) vector bundle on $\BoA^1$ of rank $P(N_1)$ and we have an epimorphism
\[
 \CO_{C\times\BoA^1}(N_1)\otimes_{\CO_{C\times\BoA^1}}\intHom_{\CO_{C\times\BoA^1}}(\CO_{C\times\BoA^1}(N_1),\CF)\cong\CO_{C\times\BoA^1}(N_1)^{P(N_1)}\xrightarrow{\tilde{f}}\CF.
\]
Thanks to the $\Lambda^{\Hod}$-structure on $\CF$, we obtain an epimorphism
\[
 \Lambda^{\Hod}\otimes_{\CO_{C\times\BoA^1}}\CO_{C\times\BoA^1}(N_1)^{P(N_1)}\xrightarrow{f}\CF.
\]
We let $\CK_1=\ker(f)$. Since $\Lambda^{\Hod}$ is Noetherian, there exists $l\in\BoN$ such that $\ker(f)$ is generated as a $\Lambda^{\Hod}$-Hodge module by $\CK_{1,\leq l}=\ker(f_{| \Lambda^{\Hod}_{\leq l}\otimes_{\CO_C}\CO_{C\times\BoA^1}(N_1)^{P(N_1)}})$. We let $N_2\ll0$ be such that we have an epimorphism $\CO_{C\times\BoA^1}(N_2)\otimes_{\CO_{\BoA^1}}(\pr_2)_*\intHom(\CO_{C\times\BoA^1}(N_2),\CK_{1,\leq l})\cong\CO_{C\times\BoA^1}(N_2)^{Q(N_2)}\xrightarrow{\tilde{g}}\CK_{1,\leq l}$. With the $\Lambda^{\Hod}$-structure on $\CK_{1}$, we obtain an epimorphism
\[
 \Lambda^{\Hod}\otimes_{\CO_{C\times\BoA^1}}\CO_{C\times\BoA^1}(N_2)^{Q(N_2)}\xrightarrow{g}\CK_1.
\]
We obtain an exact sequence
\[
 0\rightarrow\CK\rightarrow\Lambda^{\Hod}\otimes_{\CO_{C\times\BoA^1}}\CO_{C\times\BoA^1}(N_2)^{Q(N_2)}\rightarrow\Lambda^{\Hod}\otimes_{\CO_{C\times\BoA^1}}\CO_{C\times\BoA^1}(N_1)^{P(N_1)}\rightarrow\CF\rightarrow0
\]
defining at the same time $\CK=\ker(g)$. Applying the left-exact functor $(\pr_2)_*\intHom_{\Lambda^{\Hod}}(-,\CG)$, we obtain the sequence of coherent sheaves on $\BoA^1$
\[
 (\pr_2)_*\intHom_{\Lambda^{\Hod}}(\CF,\CG)\rightarrow(\pr_2)_*\intHom_{\Lambda^{\Hod}}(\Lambda^{\Hod}\otimes_{\CO_{C\times\BoA^1}}\CO(N_1)^{P(N_1)},\CG)\rightarrow(\pr_2)_*\intHom_{\Lambda^{\Hod}}(\Lambda^{\Hod}\otimes_{\CO_{C\times\BoA^1}}\CO(N_2)^{P(N_2)},\CG).
\]
By construction, the fiber over $\lambda\in\BoA^1$ of this sequence is quasi-isomorphic to $\RHom_{\Lambda^{\Hod}_{\lambda}}(\CF_{\lambda},\CG_{\lambda})$ where $\CF_{\lambda}$ (resp. $\CG_{\lambda}$) is the restriction to $\FM_{\lambda}^{\Hod}$ of $\CF$ (resp. $\CG$). For fixed $r,s\in\BoN$, one can perform this construction uniformly over $(\CF,\CG)\in\FM_r^{\Hod}\times\FM_s^{\Hod}$ (by boundedness of semistable $\Lambda^{\Hod}$-modules of fixed rank), producing a $3$-term complex of vector bundles over $\FM_r^{\Hod}\times\FM_s^{\Hod}$.
\end{proof}

Given a derived stack $\boldsymbol{\FM}$, we denote by $t_0(\boldsymbol{\FM})$ the classical truncation of $\boldsymbol{\FM}$. Moreover, the Total space of a complex of vector bundles $\CC$ concentrated in degrees $[-1,+\infty)$ on an Artin stack is a derived stack, which we denote by $\Tot(\CC)$ (see \cite[Definition~3.4]{porta2022two}).
\begin{proposition}
 We have $t_0(\Tot(\CC^{\Hod}))\simeq\mathfrak{Exact}^{\Hod}$, the stack of short exact sequence of $\Lambda^{\Hod}$-modules.
\end{proposition}
\begin{proof}
 This can be proven using the same considerations as in \cite[Proposition~3.6]{porta2022two} regarding stacks of short exact sequences and cotangent bundles.
\end{proof}

\subsection{Cohomological Hall algebra structure}
\label{subsection:CoHAstructure}
In this section, we construct the cohomological Hall algebra product on $(\JH^{\Hod})_*\BD\underline{\BoQ}_{\FM^{\Hod}}^{\vir}$, where $\underline{\BoQ}_{\FM^{\Hod}_r}^{\vir}\coloneqq\underline{\BoQ}_{\FM^{\Hod}_r}\otimes\BoL^{(1-g)r^2}$. Consider the diagram
\[
 \begin{tikzcd}
	{\FM^{\Hod}_r\times_{\BoA^1}\FM^{\Hod}_s} & {\mathfrak{Exact}^{\Hod}_{r,s}} & {\FM^{\Hod}_{r+s}} \\
	{\CM^{\Hod}_r\times_{\BoA^1}\CM^{\Hod}_s} && {\CM_{r+s}^{\Hod}}
	\arrow["q^{\Hod}"',from=1-2, to=1-1]
	\arrow["p^{\Hod}",from=1-2, to=1-3]
	\arrow["\oplus"', from=2-1, to=2-3]
	\arrow["{\JH^{\Hod}_r\times_{\BoA^1}\JH^{\Hod}_s}"', from=1-1, to=2-1]
	\arrow["{\JH^{\Hod}_{r+s}}", from=1-3, to=2-3]
\end{tikzcd}
\]
The virtual pullback for the three-term complex of vector bundles $\CC^{\Hod}$ (Lemma \ref{lemma:3termHodge}) defined as in \cite[\S4.4.3]{davison2022BPS} gives a morphism
\[
 v_{r,s}\colon \BD\underline{\BoQ}_{\FM_r^{\Hod}\times_{\BoA^1}\FM_s^{\Hod}}\rightarrow q^{\Hod}_*\BD\underline{\BoQ}_{\mathfrak{Exact}^{\Hod}_{r,s}}\otimes\BoL^{\vrank(\CC^{\Hod})}.
\]
The pushforward for the proper map $p^{\Hod}$ gives a morphism
\[
 w_{r,s}\colon p^{\Hod}_*\BD\underline{\BoQ}_{\mathfrak{Exact}^{\Hod}_{r,s}}\rightarrow\BD\underline{\BoQ}_{\FM^{\Hod}_{r+s}}.
\]
The degree $(r,s)$ component of the cohomological Hall algebra product is given by the Tate twisted composition
\[
 m_{r,s}\coloneqq ((\JH_{r+s}^{\Hod})_*w_{r,s}\circ\oplus_*(\JH_r^{\Hod}\times_{\BoA^1}(\JH_s^{\Hod}))_*v_{r,s})\otimes\BoL^{2(1-g)(r^2+s^2)}.
\]

\begin{theorem}
\label{theorem:CoHAHodgeDeligne}
The complex of mixed Hodge modules $\underline{\SA}^{\Hod}=(\JH^{\Hod})_*\BD\underline{\BoQ}_{\FM^{\Hod}}^{\vir}$ has an algebra structure $m$ in $(\CD^+(\MHM(\CM^{\Hod})),\boxdot_{\BoA^1})$, whose $(r,s)$-graded component is given by $m_{r,s}$. For $\lambda=0$ (resp. $\lambda=1$), $(\imath_{\CM_{\lambda}^{\Hod}})^!\underline{\SA}^{\Hod}$ is isomorphic to the cohomological Hall algebra of the Dolbeault moduli stack $\underline{\SA}^{\Dol}$ recalled in \S\ref{subsection:CoHAandBPSalgebra} (resp. de Rham moduli stack $\underline{\SA}^{\dR}$, constructed in \S\ref{subsection:connections}).
\end{theorem}

\begin{proof}
 The algebra structure is constructed in the discussion preceding the theorem. We use the fact that $\BD\BoQ_{\FM_r^{\Hod}\times_{\BoA^1}\FM_s^{\Hod}}\cong \BD\BoQ_{\FM_r^{\Hod}}\boxtimes_{\BoA^1}\BD\BoQ_{\FM_s^{\Hod}}$. The restriction of $\underline{\SA}^{\Hod}$ to the fiber over $0,1\in\BoA^1$ coincide with the Dolbeault or de Rham cohomological Hall algebras, as the diagram
 \[
\FM^{\Hod}\times_{\BoA^1}\FM^{\Hod}\leftarrow\mathfrak{Exact}^{\Hod}\rightarrow\FM^{\Hod}
 \]
 restricts over $0,1$ to the corresponding (Dolbeault or de Rham) diagram and the restriction of the complex $\CC^{\Hod}$ is the complex $\CC^{\Dol}$ or $\CC^{\dR}$ by construction (Lemma \ref{lemma:3termHodge}).
\end{proof}

\subsection{Structure of the CoHA of the Hodge moduli stack}
\label{subsection:CoHAHodgemodulistack}
\begin{lemma}
\label{lemma:HodgeCoHAnonnegative}
 The complex of constructible sheaves $\SA^{\Hod}\in\CD_{\rmc}^+(\CM^{\Hod},\BoQ)$ is concentrated in relative cohomological degrees $\geq 0$.
\end{lemma}
\begin{proof}
 This follows from the definition of the relative perverse $t$-structure on $\CD^+_{\rmc}(\CM^{\Hod},\BoQ)$. and the fact that for any $\lambda\in\BoC$, by base-change in the diagram
 \[
  \begin{tikzcd}
  {\FM^{\Hod}_{\lambda}} & {\FM^{\Hod}} \\
	{\CM^{\Hod}_{\lambda}} & {\CM^{\Hod}} \\
	{\{\lambda\}} & {\BoA^1}
	\arrow["\imath_{\CM_{\lambda}^{\Hod}}",from=2-1, to=2-2]
	\arrow["\JH_{\lambda}^{\Hod}"',from=1-1, to=2-1]
	\arrow[from=1-1, to=1-2]
	\arrow["\JH^{\Hod}",from=1-2, to=2-2]
	\arrow["\lrcorner"{anchor=center, pos=0.125}, draw=none, from=1-1, to=2-2]
	\arrow[from=2-1,to=3-1]
	\arrow[from=3-1,to=3-2]
	\arrow[from=2-2,to=3-2]
	\arrow["\lrcorner"{anchor=center, pos=0.125}, draw=none, from=2-1, to=3-2]
\end{tikzcd},
 \]
 we get $\imath_{\CM_{\lambda}^{\Hod}}^!(\JH^{\Hod})_*\BD\BoQ_{\FM^{\Hod}}^{\vir}\cong (\JH_{\lambda}^{\Hod})_*\BD\BoQ_{\FM_{\lambda}^{\Hod}}^{\vir}$. If $\lambda=0$, we use that $\imath_{\CM_{\lambda}^{\Hod}}^!(\JH^{\Hod})_*\BD\BoQ_{\FM^{\Hod}}^{\vir}\cong\SA^{\Dol}$ is concentrated is nonnegative perverse degrees, by \cite[Lemma 10.1]{davison2022BPS}. If $\lambda\in \BoC\setminus\{0\}$, we use the $\BoC^*$-equivariance of $\SA^{\Hod}$ which gives $\imath_{\CM_{\lambda}^{\Hod}}^!(\JH^{\Hod})_*\BD\BoQ_{\FM^{\Hod}}^{\vir}\cong \SA^{\dR}$ which, by Theorem \ref{theorem:CoHAdeRham}, is concentrated in nonnegative perverse degrees.

\end{proof}

\begin{corollary}
  The CoHA structure on $\SA^{\Hod}$ induces an algebra structure on the zeroeth relative perverse cohomology $\BPS^{\Hod}\coloneqq\pAH{\BoA^1}{0}(\SA^{\Hod})$.
\end{corollary}
\begin{proof}
 The fact that the algebra structure $m$ on $\SA^{\Hod}$ induces an algebra structure on $\BPS^{\Hod}_{\Alg}$ is a purely formal consequence of the properties of $t$-structures combined with Lemma \ref{lemma:HodgeCoHAnonnegative}. Namely, by adjunction, we have a morphism
 \begin{equation}
  \label{equation:adjunctionmapCoHAHodge}
  \pAH{\BoA^1}{0}(\SA^{\Hod})=\pStau{\BoA^1}{\leq 0}\SA^{\Hod}\rightarrow\SA^{\Hod}
 \end{equation}
and so a morphism
\begin{equation}
\label{equation:HodgeCoHApart}
  \pAH{\BoA^1}{0}(\SA^{\Hod})\boxdot_{\BoA^1} \pAH{\BoA^1}{0}(\SA^{\Hod})\rightarrow \SA^{\Hod}\boxdot_{\BoA^1}\SA^{\Hod}.
\end{equation}
By composing with the multiplication $m$ (Theorem \ref{theorem:CoHAHodgeDeligne}), we obtain a map
\begin{equation}
\label{equation:constructionmultBPS}
 \BPS_{\Alg}^{\Hod}\boxdot_{\BoA^1}\BPS_{\Alg}^{\Hod}\rightarrow\SA^{\Hod}.
\end{equation}
Since $\BPS_{\Alg}^{\Hod}\boxdot_{\BoA^1}\BPS_{\Alg}^{\Hod}\in\Perv(\CM^{\Hod}/\BoA^1)$ ($\oplus$ being finite on fibers, Proposition \ref{proposition:monoidalstructures}), the composition of this map with the adjunction map $\SA^{\Hod}\rightarrow\pStau{\BoA^1}{\geq 1}\SA^{\Hod}$ vanishes and so, the map \eqref{equation:constructionmultBPS} factors through \eqref{equation:adjunctionmapCoHAHodge}. This gives the multiplication $\BPS_{\Alg}^{\Hod}\boxdot_{\BoA^1}\BPS_{\Alg}^{\Hod}\rightarrow\BPS_{\Alg}^{\Hod}$.
\end{proof}

\begin{lemma}
\label{lemma:degreegeqminus2Hodge}
 The complex of mixed Hodge modules $\underline{\SA}^{\Hod}\in\CD^+(\MHM(\CM^{\Hod}))$ is concentrated in cohomological degrees $\geq -1$.
\end{lemma}
\begin{proof}
 Clearly, it suffices to prove that the complex of constructible sheaves $\SA^{\Hod}=\rat(\underline{\SA}^{\Hod})$ is concentrated in perverse cohomological degrees $\geq -1$. Let $r\geq 1$. Since $\SA^{\Hod}_r$ is $\BoC^*$-equivariant on $\BoA^1$, the restriction $\SA^{\Hod}_{\BoA^1\setminus\{0\}}$ is isomorphic to $\SA^{\dR}\boxtimes \BoQ_{\BoC^*}[2]$ under a trivialization $\CM^{\Hod}_{\BoC^*}\cong \CM^{\dR}\times\BoC^*$ (for example, the prefered trivialization \S\ref{subsubsection:preferredtrivialization}). Therefore, if the smallest integer $i\in\BoZ$ such that $\pH{i}(\SA^{\Hod})\neq 0$ is $<-1$, then $\pH{i}(\SA^{\Hod})$ is a perverse sheaf supported on $\CM^{\Dol}=\pi^{-1}_{\CM^{\Hod}}(0)$. Then, we would get that $\SA^{\Dol}=\imath_{\CM_{0}^{\Hod}}^!\SA^{\Hod}$ has summands in negative perverse degrees (actually, in perverse degrees $<-1$) which is a contradiction. This proves the lemma.
\end{proof}

\begin{lemma}
\label{lemma:weightsHodge}
 The complex of mixed Hodge modules $\underline{\SA}^{\Hod}$ has weights $\geq 0$.
\end{lemma}
\begin{proof}
 By Lemma \ref{lemma:coincidenceweightsforMHM} of Appendix \ref{section:purityforMMHM}, it suffices to prove that for any $\lambda\in\BoA^1$, $\imath_{\CM_{\lambda}}^!\underline{\SA}^{\Hod}$ has weights $\geq 0$. If $\lambda=0$, $\imath_{\CM_{\lambda}}^!\underline{\SA}^{\Hod}\cong\underline{\SA}^{\Dol}$ while if $\lambda\neq 0$, $\imath_{\CM_{\lambda}}^!\underline{\SA}^{\Hod}\cong\underline{\SA}^{\dR}$ and the claim follows from \cite[Theorem C]{davison2021purity} and Theorem \ref{theorem:CoHAdeRham} respectively.
\end{proof}

\begin{lemma}
\label{lemma:ICsubobject}
Let $C$ be a smooth projective curve of genus $g\geq 2$. For any $r\geq 1$, we have a canonical morphism
\[
\underline{\IC}(\CM_r^{\Hod}/\BoA^1)\rightarrow\underline{\SA}^{\Hod}  
\]
which for $0=\lambda\in\BoA^1$ induces the canonical morphism
\[
 \underline{\IC}(\CM_{r}^{\Dol})\rightarrow\underline{\SA}^{\Dol}
\]
and for $\lambda=1$ the canonical morphism
\[
 \underline{\IC}(\CM_r^{\dR})\rightarrow\underline{\SA}^{\dR}
\]
(defined as in \cite[\S10.2]{davison2022BPS}).

If $C$ is of genus $1$, we have this morphism but only for $r=1$.
\end{lemma}
\begin{proof}
The morphism $\JH^{\Hod}\colon\FM^{\Hod}\rightarrow\CM^{\Hod}$ is a $\mathbf{G}_{\mathrm{m}}$-gerbe over the open locus $\CM^{\Hod,s}$ of $\CM^{\Hod}$ parametrising simple $\bm{\lambda}$-connections. Moreover, by definition, $\underline{\BoQ}_{\FM^{\Hod}_r}^{\vir}=\underline{\BoQ}_{\FM^{\Hod}_r}\otimes\BoL^{(1-g)r^2}$. Therefore, we have
\[
 (\JH^{\Hod}_*\BD\underline{\BoQ}_{\FM^{\Hod}}^{\vir})_{\CM^{\Hod,s}}\cong \underline{\BoQ}_{\CM^{\Hod,s}}\otimes\BoL^{-\frac{\dim\CM^{\Hod}+1}{2}}\otimes\HO^*_{\BoC^*}.
\]
Note that for $r\in\BoN$, $\dim\CM_r^{\Hod}+1=2(g-1)r^2+4$ is even. Then we have a canonical map
\[
 \underline{\BoQ}_{\CM^{\Hod,s}}\otimes\BoL^{-\frac{\dim\CM^{\Hod}+1}{2}}\rightarrow\underline{\SA}^{\Hod}_{\CM^{\Hod,s}}.
\]
Since by Lemma \ref{lemma:degreegeqminus2Hodge}, $\underline{\SA}^{\Hod}_{\CM^{\Hod,s}}$ is in cohomological degrees $\geq -1$, we obtain a map between mixed Hodge modules
\[
 \underline{\BoQ}_{\CM^{\Hod,s}}\otimes\BoL^{-\frac{\dim\CM^{\Hod}+1}{2}}\rightarrow\CH^{-1}(\underline{\SA}^{\Hod}_{\CM^{\Hod,s}})[1].
\]
Since by Lemma \ref{lemma:weightsHodge}, $\CH^{-1}(\underline{\SA}^{\Hod}_{\CM^{\Hod,s}})[1]$ has nonnegative weights, and $\underline{\BoQ}_{\CM^{\Hod,s}}\otimes\BoL^{-\frac{\dim\CM^{\Hod}+1}{2}}$ is pure of weight zero, this map factors through the inclusion
\[
 W_0(\CH^{-1}(\underline{\SA}^{\Hod}_{\CM^{\Hod,s}})[1])\rightarrow \CH^{-1}(\underline{\SA}^{\Hod}_{\CM^{\Hod,s}})[1]
\]
of the weight zero part to obtain the map
\begin{equation}
\label{equation:extensionmap}
 \underline{\BoQ}_{\CM^{\Hod,s}}\otimes\BoL^{-\frac{\dim\CM^{\Hod}+1}{2}}\rightarrow W_0(\CH^{-2}(\underline{\SA}^{\Hod}_{\CM^{\Hod,s}})[2]).
\end{equation}
Since $W_0(\CH^{-1}(\underline{\SA}^{\Hod})[1])$ is semisimple (as a pure shifted mixed Hodge module), we can uniquely extend the map \eqref{equation:extensionmap} obtained this way to
\[
 \underline{\IC}(\CM^{\Hod}/\BoA^1)\rightarrow W_0(\CH^{-1}(\underline{\SA}^{\Hod})[1]),
\]
which by post-composing with $W_0(\CH^{-1}(\underline{\SA}^{\Hod}_{\CM^{\Hod}})[1])\rightarrow \CH^{-1}(\underline{\SA}^{\Hod}_{\CM^{\Hod}})[1]\rightarrow \underline{\SA}^{\Hod}$ gives the desired map.

The restrictions of this map to the de Rham or Dolbeault stacks are the analogous morphisms from the intersection cohomology of the moduli space to the relative CoHA as they are obtained in the same way \cite[\S10.2]{davison2022BPS}.
\end{proof}
Recall that $\Delta_r\colon\CM_1^{\Hod}\rightarrow\CM_r^{\Hod}$, $x\mapsto x^{\oplus r}$ is the small diagonal of the Hodge moduli space.
\begin{theorem}
\label{theorem:relativBPSalg}
\begin{enumerate}
 \item 
 If $C$ is a smooth projective curve of genus $1$, then the natural map
 \[
  \Sym_{\boxdot_{\BoA^1}}\left(\bigoplus_{r\geq 1}(\Delta_r)_*\IC(\CM_1^{\Hod}/\BoA^1)\right)\rightarrow \BPS_{\Alg}^{\Hod}
 \]
is an isomorphism of algebra objects in $\Perv(\CM^{\Hod}/\BoA^1)$.
\item 
If $C$ is a smooth projective curve of genus $\geq 2$, the natural map
\[
 \Free_{\boxdot_{\BoA^1}-\Alg}\left(\bigoplus_{r\geq 1}\IC(\CM_r^{\Hod}/\BoA^1)\right)\rightarrow\BPS_{\Alg}^{\Hod}
\]
is an isomorphism of algebra objects in $\Perv(\CM^{\Hod}/\BoA^1)$.
\end{enumerate}
\end{theorem}
\begin{proof}
We prove $(2)$. The proof of $(1)$ is similar, relying on \cite[Proposition 5.11]{davison2023nonabelian} for the zero fiber (Dolbeault moduli space) and on Theorem \ref{theorem:genusone} for all other fibers (corresponding to the de Rham stack, using $\BoC^*$-equivariance).

Since $\IC(\CM_r^{\Hod}/\BoA^1)\in\Perv(\CM_r^{\Hod}/\BoA^1)$, the natural morphism $\IC(\CM_r^{\Hod}/\BoA^1)\rightarrow\SA_r^{\Hod}$ provided by Lemma \ref{lemma:ICsubobject} factors through the natural map $\BPS_{\Alg}^{\Hod}\rightarrow\SA^{\Hod}$. Therefore, one can construct a morphism of algebra objects
\[
 \Xi_{\BPS}\colon\Free_{\boxdot_{\BoA^1}}\left(\bigoplus_{r\geq 1}\IC(\CM_r^{\Hod}/\BoA^1)\right)\rightarrow\BPS_{\Alg}^{\Hod}.
\]
The goal is to prove that it is an isomorphism. We let $\CK$ be the cone of $\Xi_{\BPS}$. Then, for $\lambda=0$ (resp. $1$), $\imath_{\CM_{\lambda}^{\Hod}}^!\CK$ is the cone of the morphism $\Free_{\boxdot_{\BoA^1}}\left(\bigoplus_{r\geq 1}\IC(\CM_r^{\Dol})\right)\rightarrow\BPS_{\Alg}^{\Dol}$ (resp. $\Free_{\boxdot_{\BoA^1}}\left(\bigoplus_{r\geq 1}\IC(\CM_r^{\dR})\right)\rightarrow\BPS_{\Alg}^{\dR}$) and so vanishes by \cite[Theorem 1.4]{davison2022BPS} (resp. Theorem \ref{theorem:freenessPBWconnections}). By $\BoC^*$-equivariance, $\imath_{\CM_{\lambda}}^!\CK=0$ for any $\lambda\in\BoA^1$. Therefore, $\CK=0$ and $\Xi_{\BPS}$ is an isomorphism.
\end{proof}

\begin{corollary}
 For any smooth projective curve, $\BPS_{\Alg}^{\Hod}\in\Perv(\CM^{\Hod})[1]$.
\end{corollary}
\begin{proof}
 This is immediate with Theorem \ref{theorem:relativBPSalg}.
\end{proof}

Recall the complex of mixed Hodge modules $\underline{\IC}(\CM_r^{\Hod}/\BoA^1)$ defined in Lemma \ref{lemma:relativeICmhm}.
By Theorem \ref{theorem:relativBPSalg}, it makes sense to define the relative BPS Lie algebra for the Hodge moduli stack as
\[
 \underline{\BPS}^{\Hod}_{\Lie}\coloneqq \Free_{\boxdot_{\BoA^1}-\Lie}\left(\bigoplus_{r\geq 1}\underline{\IC}(\CM_r^{\Hod}/\BoA^1)\right).
\]
We also let
\[
 \underline{\BPS}_{\Alg}^{\Hod}\coloneqq \Free_{\boxdot_{\BoA^1}-\Alg}\left(\bigoplus_{r\geq 1}\underline{\IC}(\CM_r^{\Hod}/\BoA^1)\right).
\]

The following is then straightforward.
\begin{lemma}[PBW isomorphism]
 We have an isomorphism of mixed Hodge modules in $\MHM(\CM^{\Hod})[1]$:
 \[
  \Sym_{\boxdot_{\BoA^1}}\left(\underline{\BPS}_{\Lie}^{\Hod}\right)\cong\underline{\BPS}_{\Alg}^{\Hod}.
 \]
\qed
\end{lemma}

\begin{theorem}
\label{theorem:relativeCoHA}
 Let $C$ be a curve of genus $g\geq 0$. Then, the natural map
\[
 \Xi^{\Hod}\colon\Sym_{\boxdot_{\BoA^1}}(\underline{\BPS}_{\Lie}^{\Hod}\otimes\HO^*_{\BoC^*})\rightarrow\underline{\SA}^{\Hod}
\]
is a quasi-isomorphism of complexes of mixed Hodge modules in $\CD^+(\MHM(\CM^{\Hod}))$.
\end{theorem}
\begin{proof}
 We let $\underline{\CG}\coloneqq \cone(\Xi^{\Hod})$. Then, $\imath_{\CM^{\Dol}}^!\underline{\CG}$ is the cone of the PBW isomorphism for the Dolbeault stack \cite[Theorem 1.5]{davison2022BPS} and so vanishes. Similarly, $\imath_{\CM^{\dR}}^!\underline{\CG}$ is the cone of the PBW isomorphism for the de Rham stack (Theorem \ref{theorem:freenessPBWconnections}), and so vanishes. By $\BoC^*$-equivariance, $\imath_{\CM_{\lambda}^{\Hod}}^!\underline{\CG}\cong 0$ for any $\lambda\in\BoA^1$ and so $\underline{\CG}\cong 0$. Therefore, $\Xi^{\Hod}$ is an isomorphism.
\end{proof}

As a corollary, we can reinforce Lemma \ref{lemma:weightsHodge} (which was crucial in proving Theorem \ref{theorem:relativeCoHA}).

\begin{corollary}
 The complexe of mixed Hodge modules $\underline{\SA}^{\Hod}\in\CD^+(\MHM(\CM^{\Hod}))$ underlying the Hodge CoHA is pure.
\end{corollary}
\begin{proof}
 The direct sum map $\oplus\colon\CM^{\Hod}\times_{\BoA^1}\CM^{\Hod}\rightarrow\CM^{\Hod}$ is finite. Therefore, $\Sym_{\boxdot_{\BoA^1}}(\underline{\BPS}_{\Lie}^{\Hod}\otimes\HO^*_{\BoC^*})$ is pure. Indeed, it suffices that for $r_1,\hdots,r_t\in\BoN$, the external product over $\BoA^1$
 \[
  (\boxtimes_{\BoA^1})_{j=1}^t\underline{\IC}(\CM_{r_j}^{\Hod}/\BoA^1)
 \]
is pure. But this external tensor product is $\underline{\IC}((\prod_{\BoA^1})_{j=1}^t\CM^{\Hod}_{r_j}/\BoA^1)$ and is therefore pure. By Theorem \ref{theorem:relativeCoHA}, $\underline{\SA}^{\Hod}$ is pure.
\end{proof}

\begin{proposition}
\label{proposition:trivializationCoHA}
 Via the trivialization $\CM^{\Hod}\cong \CM^{\Dol}\times\BoA^1$ given by nonabelian Hodge theory, we have
 \[
  \SA^{\Hod}\cong\SA^{\Dol}\boxtimes \BoQ_{\BoA^1}[2],
 \]
 where the product over $\BoA^1$ for $\SA^{\Dol}\boxtimes \BoQ_{\BoA^1}[2]$ is given by
 \[
  m^{\Dol}\boxtimes\id_{\BoQ_{\BoA^1}[2]}\colon (\SA^{\Dol}\boxtimes \BoQ_{\BoA^1}[2])\boxtimes_{\BoA^1}(\SA^{\Dol}\boxtimes \BoQ_{\BoA^1}[2])\cong \SA^{\Dol}\boxtimes\SA^{\Dol}\boxtimes\BoQ_{\BoA^1}[2]\rightarrow \SA^{\Dol}\boxtimes\BoQ_{\BoA^1}[2].
 \]
\end{proposition}
\begin{proof}
 This follows from Lemma \ref{lemma:morphismonfiber} applied to $\SF=\SA^{\Dol}\boxdot\SA^{\Dol}$ and $\SG=\SA^{\Dol}$, and to the multiplication for the Hodge stack $m\in\Hom_{\CD_{\rmc}^+(\CM^{\Hod})}(\SA^{\Hod}\boxdot_{\BoA^1}\SA^{\Hod},\SA^{\Hod})$.
\end{proof}

\begin{corollary}
\label{corollary:reldRdolcoincide}
 The relative cohomological Hall algebras $\SA^{\Dol}$ and $\SA^{\dR}$ coincide via the homeomorphism $\CM^{\Dol}\cong\CM^{\dR}$ given by nonabelian Hodge theory.
\end{corollary}
\begin{proof}
 This is immediate from Proposition \ref{proposition:trivializationCoHA}.
\end{proof}

\begin{corollary}
\label{corollary:abdDoldRcoincide}
 The cohomological Hall algebras $\HO^*\SA^{\Dol}$ and $\HO^*\SA^{\dR}$ are isomorphic.
\end{corollary}
\begin{proof}
 This follows from Proposition \ref{proposition:trivializationCoHA}.
\end{proof}

\begin{proposition}
\label{proposition:Hodgelocconstant}
 The constructible sheaf $(\pi_{\CM}^{\Hod})_*\SA^{\Hod}$ has locally constant and therefore constant cohomology sheaves.
\end{proposition}
\begin{proof}
 This is a consequence of Theorem \ref{theorem:relativeCoHA} since all simple direct summands $\SF$ of $\BPS_{\Lie}^{\Hod}$ are such that $(\pi_{\CM}^{\Hod})_*\SF$ has locally constant cohomology sheaves. Indeed, $\SF=\IC(\CM^{\Hod})$ up to a shift and $\FM^{\Hod}\rightarrow\BoA^1$ is a topologically trivial fibration (\S\ref{subsubsection:preferredtrivialization}).
\end{proof}

\begin{remark}
 Proposition \ref{proposition:Hodgelocconstant} is nicely surprising, as it states that the Borel--Moore homology of the stacks of $\lambda$-connections is constant with $\lambda$, although the stack of $\lambda$-connection does not itself seem to be topologically trivial over $\BoA^1$ (in a sense to be defined), see \cite[Counterexample pp.38-39]{simpson1994moduliII} and \cite[Question p.12]{simpson1996hodge}.
\end{remark}

\section{Nonabelian Hodge isomorphisms}

\subsection{de Rham and Betti}
\label{subsection:deRhamBetti}
In this section, we explain how to enhance the isomorphism between the cohomologies of the de Rham and Betti moduli space induced by the Riemann--Hilbert correspondence to an isomorphism of algebras between the corresponding cohomological Hall algebras.

\subsubsection{Comparison of the (co)homologies}
\begin{theorem}[Simpson]
\label{theorem:simpsondrBetti}
Let $C$ be a genus $g$ curve. We have isomorphisms
\[
 \HO^*(\FM_{r}^{\dR}(C))\cong\HO^*(\FM_{g,r}^{\Betti})
\]
and
\[
 \HO^{\rmBM}_*(\FM_r^{\dR}(C))\cong \HO^{\rmBM}_*(\FM_{g,r}^{\Betti}).
\]
\end{theorem}
\begin{proof}
With the preliminary study of Simpson of these moduli stacks \cite{simpson1994moduliII}, the proof is straightforward. Indeed, the stack $\FM_r^{\dR}$ (resp. $\FM_r^{\Betti}$) is described by Simpson as the global quotient stack $R_r^{\dR}/\GL_r$ (resp. $R_r^{\Betti}/\GL_r$) of the moduli space space of framed connection (resp. framed local systems) and the Riemann--Hilbert correspondence establishes a $\GL_r$-equivariant homeomorphism $\Phi'\colon R_r^{\dR}\rightarrow R_r^{\Betti}$. Translated in terms of equivariant cohomology (resp. equivariant Borel--Moore homology), the statements of the theorem are now obvious.
\end{proof}

\subsubsection{Comparison of the cohomological Hall algebra structures}
\begin{theorem}
\label{theorem:reldRBetticoincide}
 Let $C$ be a smooth projective curve. We let $\Phi\colon\CM^{\dR}\rightarrow\CM^{\Betti}$ be the homeomorphism induced by the Riemann--Hilbert correspondence. Then, we have an isomorphism of algebra objects
 \[
  \Phi_*(\JH^{\dR})_*\BD\BoQ_{\FM^{\dR}}^{\vir}\cong(\JH^{\Betti})_*\BD\BoQ_{\FM^{\Betti}}^{\vir}\in\CD^+_{\rmc}(\CM^{\Betti}).
 \]
 It induces an isomorphism of algebra objects $\HO^*(\SA^{\dR})\cong\HO^*(\SA^{\Betti})$ between the de Rham and Betti CoHAs.
\end{theorem}

\begin{proof}
 The formalism needed was developped in \cite{porta2022two}. We let $\bm{\FM}^{\dR}$ be the derived stack of connections on $C$ and $\bm{\FM}^{\Betti}$ the derived stack of local systems. We have the corresponding stacks of extensions $\bm{\mathfrak{Exact}}^{\dR}$ and $\bm{\mathfrak{Exact}}^{\Betti}$ and the diagrams
 \[
  \bm{\FM}^{\dR}\times\bm{\FM}^{\dR}\xleftarrow{\bm{q}^{\dR}}\bm{\mathfrak{Exact}}^{\dR}\xrightarrow{\bm{p}^{\dR}}\bm{\FM}^{\dR}
 \]
and
\[
  \bm{\FM}^{\Betti}\times\bm{\FM}^{\Betti}\xleftarrow{\bm{q}^{\Betti}}\bm{\mathfrak{Exact}}^{\Betti}\xrightarrow{\bm{p}^{\Betti}}\bm{\FM}^{\Betti},
 \]
and one of the key results of \cite{porta2022two}, see Proposition 7.21 of \emph{loc. cit.} is that the analytifications of these diagrams are isomorphic, via the derived Riemann--Hilbert correspondence $\bm{\Phi}$. This implies that the cotangent complexes $\mathbb{L}_{\bm{q}^{\dR}}$ and $\mathbb{L}_{\bm{q}^{\Betti}}$ are isomorphic (through the Riemann--Hilbert correspondence) and hence,
\[
 \bm{\Phi}^*\mathbb{L}_{\bm{q}^{\Betti}}\cong\mathbb{L}_{\bm{q}^{\dR}}
\]
and in particular, since the complexes $\CC^{\Betti}$ and $\CC^{\dR}$ are the restrictions to the classical truncations $\FM^{\dR}\times\FM^{\dR}$ (resp. $\FM^{\Betti}\times\FM^{\Betti}$) of $\mathbb{L}_{\bm{q}^{\dR}}$ (resp. $\mathbb{L}_{\bm{q}^{\Betti}}$),

\[
 {\Phi}^*\CC^{\Betti}\cong\CC^{\dR}.
\]
From this and the construction of virtual pullbacks using the $3$-term complexes $\CC^{\dR}$ and $\CC^{\Dol}$, if follows that through the Riemann--Hilbert correspondence, the CoHA multiplications coincide: we have an isomorphism of algebra objects
\[
 \Phi_*\SA^{\dR}\cong \SA^{\Betti}.
\]
The theorem follows.
\end{proof}

\begin{remark}
One could use the strategy employed in \cite{davison2022BPS} for the Betti and Dolbeault moduli spaces to compare the Borel--Moore homologies of the Betti and de Rham moduli stacks. This would give an isomorphism of constructible complexes
\[
 \Phi_*(\JH^{\dR})_*\BD\BoQ_{\FM^{\dR}}^{\vir}\cong \Phi_*(\JH^{\Betti})_*\BD\BoQ_{\FM^{\Betti}}^{\vir}.
\]
It is immediate, using Theorem \ref{theorem:reldRBetticoincide}, that this isomorphism coincides with the one given by Theorem \ref{theorem:reldRBetticoincide}.
\end{remark}

\subsection{de Rham and Dolbeault}
\label{subsection:deRhamDolbeault}
The comparison between the de Rham and Dolbeault CoHAs was deduced in \S\ref{subsection:CoHAHodgemodulistack} from the study of the Hodge--Deligne CoHA. We refer to Corollaries \ref{corollary:reldRdolcoincide} and \ref{corollary:abdDoldRcoincide}.

Nevertheless, our results do not give a comparison between the cohomologies. This leads to the following question.
\begin{question}
\label{question:isocohdRDolbeault}
Do we have an isomorphism
\[
 \HO^*(\FM_{g,r,0}^{\dR})\cong\HO^*(\FM_{r,0}^{\Dol}(C))?
\]
\end{question}

\subsection{Betti and Dolbeault}
One of the objectives of \cite{davison2022BPS} was to establish connections between the various versions of the $\mathrm{P}=\mathrm{W}$ conjecture, which concerns the Betti and Dolbeault moduli spaces (namely, to relate it to the $\mathrm{IP}=\mathrm{IW}$ conjecture regarding intersection cohomology (first made in \cite[Question 4.1.7]{de2018perverse}); and to the $\mathrm{SP}=\mathrm{SW}$ conjecture considering the Borel--Moore homologies of the stacks, see \cite{davison2023nonabelian}). Our main result there concerning the Betti and Dolbeault moduli stacks is the following.

\begin{theorem}[{Davison--Hennecart--Schlegel Mejia, \cite[Theorem 1.7]{davison2022BPS}}]
\label{theorem:DHS2022}
We have a canonical isomorphism of vector spaces
\[
 \HO_*^{\rmBM}(\FM_{g,r,d}^{\Betti})\cong \HO^{\rmBM}_*(\FM_{r,d}^{\Dol}(C)).
\]
\end{theorem}

Our method in \cite{davison2022BPS} to prove this isomorphism is of representation theoretic nature and uses the cohomological algebra structures on both sides (after summing over all pairs $(r,d)$ such that $d/r=\mu$ is constant). The cohomological Hall algebra structure does not exist on the cohomology on the Dolbeault and Betti moduli spaces, which makes it not possible to prove an isomorphism in cohomology using our methods (and therefore we cannot answer Question \ref{question:isocohdRDolbeault}).

In this paper, we retrieve Theorem \ref{theorem:DHS2022} as a combination of Theorem \ref{theorem:reldRBetticoincide} and Corollary \ref{corollary:reldRdolcoincide}, but in addition we are able to compare the cohomological Hall algebra structures. Namely, the isomorphism of Theorem \ref{theorem:DHS2022} is an isomorphism of algebras.

\subsection{Affinized BPS Lie algebras}
Given a Lie algebra $\mathfrak{g}$, an affinization of $\mathfrak{g}$ is a Lie algebra $\tilde{\mathfrak{g}}$ which as a vector space is isomorphic to $\mathfrak{g}[u]$, that is polynomials in one variable with coefficients in the original Lie algebra. The Lie bracket on $\tilde{\mathfrak{g}}$ could be (and is, in general) more complicated than the obvious extension of the Lie bracket of $\mathfrak{g}$ to $\mathfrak{g}[u]$. This terminology is convenient for us. In the literature \cite[Chapter 7]{kac1990infinite}, affine Lie algebras often refer to central extensions of the Lie algebra of loops over $\mathfrak{g}$, which are different objects than the ones considered here.

Affinized BPS Lie algebras made their appearance through the example of the tripled Jordan quiver with potential in \cite{davison2022affine}. The ultimate definition of affinized BPS Lie algebras relies on a coproduct on the cohomological Hall algebra. Since the definition of this coproduct has not yet appeared, we formulate our results independently of this coproduct.

\subsubsection{Isomorphism between affinized BPS Lie algebras}
We reformulate Theorem \ref{theorem:actiondetlinebundle} as follows.

\begin{theorem}
\label{theorem:actionChernclassdet}
 Through the nonabelian Hodge homeomorphisms $\Phi\colon\CM^{\dR}\rightarrow\CM^{\Betti}$ and $\Psi\colon\CM^{\dR}\rightarrow\CM^{\Dol}$, the actions of the algebra $\HO^*_{\BoC^*}$ on $\SA^{\dR}$, $\SA^{\Dol}$ and $\SA^{\Betti}$ coincide with each other.
\end{theorem}
\begin{proof}
 We first compare the $\HO^*_{\BoC^*}$-actions through the Riemann--Hilbert correspondence. According to \cite[Theorem~7.1]{simpson1994moduliII}, the (framed) Riemann--Hilbert correspondence gives a $\GL_r$-equivariant homeomorphism of topological spaces $\tilde{\Phi}\colon R^{\dR}\rightarrow R^{\Betti}$. Therefore, working with the equivariant Borel--Moore homology and the commutative square
 \[
  \begin{tikzcd}
	{R^{\dR}} & {R^{\Betti}} \\
	{\CM^{\dR}} & {\CM^{\Betti}}
	\arrow["\Phi", from=2-1, to=2-2]
	\arrow["{\tilde{\Phi}}", from=1-1, to=1-2]
	\arrow["{\JH^{\dR}}"', from=1-1, to=2-1]
	\arrow["{\JH^{\Betti}}", from=1-2, to=2-2]
\end{tikzcd},
 \]
whose horizontal arrows are homeomorphisms, the statement becomes obvious as the determinant map is obtained as the limit as $N\rightarrow\infty$ of the maps
\[
 (R^{\sharp}_r\times_{\GL_r} E_N)\rightarrow (R^{\sharp}_r\times_{\GL_r} E_N)\times E_N/\GL_r
\]
for $\sharp\in\{\dR,\Betti\}$, where $E_N$ is a smooth variety with no cohomology in degrees $2\leq i\leq N$ on which $\GL_r$ acts freely.

We turn to the comparison of determinant line bundles between the de Rham and Dolbeault moduli spaces. Again, we use the Hodge--Deligne moduli stack and its relative CoHA over $\BoA^1$.

By \cite{simpson1994moduliII} (see \S\ref{subsection:Hodgemodulistack}), it can be described as the quotient stack $\FM_r^{\Hod}=R^{\Hod}_r/\GL_r\rightarrow\BoA^1$. The determinant character $\det\colon\GL_r\rightarrow\BoC^*$ gives the map
\[
 \tilde{\det}\colon\FM_r^{\Hod}\rightarrow\FM_r^{\Hod}\times_{\BoA^1}\BoA^1/\GL_r\cong \FM_r^{\Hod}\times\rmB\BoC^*.
\]
where the action of $\GL_r$ on $\BoA^1$ is trivial.

This gives the action map
\[
 \SA^{\Hod}_r\otimes\HO^*_{\BoC^*}\rightarrow\SA_r^{\Hod}
\]
as in \S\ref{subsection:thedetlb}. By taking a generator $u\in\HO^*_{\BoC^*}$ (which sits in cohomological degree $2$) and combining all $r\in\BoN$ together, we have a morphism
\[
 \cdot u \colon \SA^{\Hod}[-2]\rightarrow\SA^{\Hod}.
\]

The restriction of this morphism over $0\in\BoA^1$ (resp. $1\in\BoA^1$) is the action of the first Chern class of the determinant line bundle on $\SA^{\Dol}$ (resp. $\SA^{\dR}$). By Lemma \ref{lemma:morphismonfiber}, the nonabelian Hodge isomorphism between the de Rham and Dolbeault relative CoHas (Corollary \ref{corollary:reldRdolcoincide}) intertwines these actions. This concludes the proof.
\end{proof}
The statement concerning the absolute CoHAs (Corollary \ref{corollary:actiondetlbabsolute}) follows by taking derived global sections.

\section{$\chi$-independence phenomena}
\label{section:chiindependenceph}
In this section, partly conjectural, we introduce new $\chi$-independence questions and results for Dolbeault and Betti CoHAs. In addition to the now classical problem of comparing the cohomology or BPS cohomology when the Euler characteristic varies, we also ask for the (Lie or associative) algebra structures to coincide. Using the nonabelian Hodge theory isomorphisms for CoHAs, one could transfer $\chi$-independence results (at the cost of forgetting the mixed Hodge structures) between Betti, de Rham and Dolbeault sides.

\subsection{CoHA product for character varieties in the $\ell$-adic setting}

\subsubsection{Classical construction}
\label{subsubsection:classicalmultladic}
In this section, we explain briefly how to define the CoHA for character varieties in $\ell$-adic Borel--Moore homology. The definitions work in the exact same way as for Borel--Moore homology in the analytic topology.

Let $\mu=d/r$ with $\gcd(r,d)=1$. We $\zeta_r\coloneqq\exp\left(\frac{2\pi i}{r}\right)$ be a primitive $r$th root of unity. We let $K=\BoQ[\zeta_r]$ and $\overline{K}=\overline{\BoQ}$ be its algebraic closure. We are interested in the stack of representations of the twisted fundamental group algebra $K[\pi_1(\Sigma_g),\zeta_r^d]\coloneqq K[x_i,y_i\colon 1\leq i\leq g]/\langle\zeta_r^d\prod_{i=1}^gx_iy_ix_i^{-1}y_i^{-1}=1\rangle$. This is a stack over $K$, which we denote by $\FM_{r,d,K}^{\Betti}$. It admits a good moduli space (defined in terms of GIT quotient) $\JH_K\colon \FM_{r,d,K}^{\Betti}\rightarrow\CM_{r,d,K}^{\Betti}$. We may base-change all stacks and schemes to $\overline{K}$. In this way, we obtain the good moduli space $\JH_{\overline{K}}\colon\FM_{r,d,\overline{K}}^{\Betti}\rightarrow\CM_{r,d,\overline{K}}^{\Betti}$.

The multiplication on $\JH_{\overline{K}}\BD(\BoQ_{\ell})_{\FM_{r,d,\overline{K}}^{\Betti}}$ may be constructed as for preprojective algebras of quivers, see \cite[Appendix by Ben Davison, \S4.1]{ren2017cohomological}. The moment map is replaced in this context by the \emph{multiplicative} moment map
\[
 \begin{matrix}
  \mu&\colon& \GL_{r,\overline{K}}^{2g}&\rightarrow&\GL_{r,\overline{K}}\\
  &&(x_i,y_i)_{1\leq i\leq g}&\mapsto& \prod_{i=1}^gx_iy_ix_i^{-1}y_i^{-1}.
 \end{matrix}
\]
By taking derived global sections, this gives the CoHA product on $\HO^{\rmBM}_*((\FM_{\mu,\overline{K}})_{\et},\BoQ_{\ell})$. 

The formalism of virtual pullbacks in $\ell$-adic Borel--Moore homology is developed in great details in \cite[\S3]{olsson2015borel}. The $\ell$-adic derived categories of algebraic stacks have been introduced in \cite{behrend2003derived}. All stacks appearing in the construction are quotient stacks, allowing an approach to the $\ell$-adic derived category via simplicial schemes, for example. We also refer to the series of papers \cite{laszlo2008six,laszlo2008sixb,laszlo2009perverse}. See \S\ref{subsection:elladic}.

\subsubsection{Multiplication in terms of the $3$-term complex}
In the case of the stacks of representations over the field of complex numbers $\BoC$, the multiplication can be defined using a $3$-term complex of vector bundles over the product $\FM\times\FM$ of the stack considered, as in \cite[\S9.1]{davison2022BPS}. It is explained in detail in \cite[Appendix A]{davison2022BPS} that the multiplications obtained in this way and in the more classical way \S\ref{subsubsection:classicalmultladic} coincide. As noticed in \cite{davison2022BPS} after Corollary 6.3, the same proof adapts to multiplicative preprojective algebras and twisted fundamental group algebras of Riemann surfaces. When working over $K=\BoQ[\zeta_r]$, the $3$-term complex of \cite[\S2.5.2]{davison2022BPS} (for preprojective algebras), see also \cite[Appendix C]{davison2022BPS}, is defined over $K$. Therefore, the formalism of \cite{davison2022BPS} combined with virtual pullbacks in the $\ell$-adic derived categories \cite{olsson2015borel} give the CoHA product on $(\JH_{\overline{K}})_*\BD(\BoQ_{\ell})_{\FM_{d/r},\overline{K}}^{\vir}$. By the arguments of \cite[Appendix A]{davison2022BPS}, it coincides with the CoHA product of \S\ref{subsubsection:classicalmultladic}.

\subsection{Galois conjugation and $\chi$-independence for character varieties}

Galois conjugate algebraic varieties may fail to have the same topology, but remarkably they have isomorphic cohomologies (with $\BoC$-coefficients).

\begin{definition}
 Let $K$ be a number field and $X$ an algebraic variety over $K$. For any embedding $\sigma\colon K\rightarrow \BoC$, we may form the algebraic variety $X_{\sigma}\coloneqq X\times_{\Spec(K)}\Spec(\BoC)$. If $\sigma'\colon K\rightarrow\BoC$ is another embedding, the varieties $X_{\sigma}$ and $X_{\sigma'}$ are called \emph{Galois conjugate}.
\end{definition}

Let $X$ be an algebraic variety defined over an algebraically closed field $K$. There are various comparisons result regarding $\ell$-adic cohomology when we extend the scalars from $K$ to a bigger algebraically closed field $L$ \cite[\S4.2.9]{beilinson2018faisceaux}. There are also comparison results between $\ell$-adic and singular cohomology when $X$ is defined over $\BoC$ (Artin's comparison theorem, see also \cite[\S6.1]{beilinson2018faisceaux} for more general statements).

\begin{proposition}
 Let $\mu=d/r$ and $\mu'=d'/r$ with $\gcd(r,d)=\gcd(r,d')=1$. Then, the cohomological Hall algebras $\HO^*(\SA_{\mu}^{\Betti},\BoC)$ and $\HO^*(\SA_{\mu'}^{\Betti},\BoC)$ are isomorphic.
\end{proposition}

\begin{proof}
 We let $r,d,d'$ as in the statement of the proposition and $\zeta_r=\exp\left(\frac{2\pi i}{r}\right)$ a primitive $r$th root of unity. We let $K\coloneqq\BoQ[\zeta_r]$ and $\overline{\BoQ}$ its algebraic closure. The convolution diagrams giving the CoHA product on $\HO^*(\SA_{\mu}^{\Betti})$ and $\HO^*(\SA_{\mu'}^{\Betti})$ are Galois conjugate. Therefore, they both come from the same convolution diagram over $\overline{\BoQ}$ by extending the scalars to $\BoC$ using different embeddings $\overline{\BoQ}\rightarrow\BoC$ and so by the comparison results mentioned above, we have isomorphisms of algebras
 \[
  \HO^{\rmBM}_{*}\left(\left(\FM_{1/r,\overline{\BoQ}}^{\Betti})\right)_{\et},\overline{\BoQ}_{\ell}\right)\cong \HO^{\rmBM}_*(\FM_{d/r}^{\Betti},\overline{\BoQ}_{\ell})
 \]
for any $d\in\BoZ$ coprime with $r$, where $ \HO^{\rmBM}_{*}\left(\left(\FM_{1/r,\overline{\BoQ}}^{\Betti})\right)_{\et},\overline{\BoQ}_{\ell}\right)$ is defined using the $\ell$-adic derived category $\CD^+_{\rmc}(\FM_{1/r,\overline{\BoQ}}^{\Betti},\overline{\BoQ}_{\ell})$ and $\HO^{\rmBM}_*(\FM_{d/r}^{\Betti},\overline{\BoQ}_{\ell})$ using the analytic constructible derived category $\CD^+_{\rmc}(\FM_{d/r}^{\Betti},\overline{\BoQ}_{\ell})$. To come back to complex coefficients, we choose an isomorphism $\overline{\BoQ}_{\ell}\cong\BoC$.
\end{proof}

\subsection{$\chi$-independence for Higgs bundles}

We start by recalling the $\chi$-independence result for the BPS cohomology of the moduli space of Higgs bundles.

Let $r,d,d'\in\BoZ$ be such that $r>0$. Recall the BPS sheaves $\underline{\BPS}_{\Lie,r,d}^{\Dol}$ and $\underline{\BPS}_{\Lie,r,d'}^{\Dol}$ on $\CM_{r,d}^{\Dol}$ and $\CM_{r,d'}^{\Dol}$ respectively. We let $\mathtt{h}\colon \CM_{\bullet,\bullet}^{\Dol}=\bigsqcup_{r,d\in\BoZ}\CM_{r,d}^{\Dol}\rightarrow B_r$ be the Hitchin morphism.

\begin{theorem}[Kinjo--Koseki {\cite{kinjo2021cohomological}}]
\label{theorem:kinjokosekichiindep}
 There is an isomorphism $\mathtt{h}_*\underline{\BPS}_{\Lie,r,d}^{\Dol}\cong \mathtt{h_*}\underline{\BPS}_{\Lie,r,d'}^{\Dol}$.
\end{theorem}

We formulate the following set of conjectures regarding the interactions of the CoHA structure with the $\chi$-independence.

Let $(r,d)\in\BoZ$. Then, $\underline{\BPS}^{\Dol}_{\Lie,\BoZ_{\geq 1}\cdot (r,d)}\coloneqq\bigoplus_{l\in\BoZ_{\geq 1}}\underline{\BPS}_{\Lie,lr,ld}^{\Dol}\subset \underline{\BPS}^{\Dol}_{\Lie,d/r}$ is a sub-Lie algebra. We let $\underline{\BPS}_{\Alg,\BoZ_{\geq 1}\cdot(r,d)}^{\Dol}\coloneqq\Sym_{\boxdot}\left(\underline{\BPS}_{\Lie,\BoZ_{\geq 1}\cdot(r,d)}^{\Dol}\right)\subset \underline{\BPS}_{\Alg,d/r}^{\Dol}$ (an equality and inclusion of mixed Hodge modules respectively). By the inclusion of Lie algebras above, and since the BPS algebra is the enveloping algebra of the BPS Lie algebra, this inclusion gives $\underline{\BPS}_{\Alg,\BoZ_{\geq 1}\cdot(r,d)}^{\Dol}$ an algebra structure, indentifying it with the enveloping algebra of $\underline{\BPS}_{\Lie,\BoZ_{\geq 1}\cdot(r,d)}^{\Dol}$.

\begin{conjecture}
\label{conjecture:chiindependence}
 \begin{enumerate}
 \item \label{item:isoLieconj}Let $r,d,d'\in\BoZ$ with $r\geq 1$. Then, by Theorem \ref{theorem:kinjokosekichiindep}, we have an isomorphism of mixed Hodge modules
 \[
  \mathtt{h}_*\underline{\BPS}^{\Dol}_{\Lie,\BoZ_{\geq 1}\cdot (r,d)}
\cong \mathtt{h}_*\underline{\BPS}^{\Dol}_{\Lie,\BoZ_{\geq 1}\cdot (r,d')}.
 \]
This is an isomorphism of Lie algebras.
\item \label{item:isoalgconj}The $\chi$-independence induces an isomorphism of algebras
\[
 \mathtt{h}_*\underline{\BPS}^{\Dol}_{\Alg,\BoZ_{\geq 1}\cdot (r,d)}
\cong \mathtt{h}_*\underline{\BPS}^{\Dol}_{\Alg,\BoZ_{\geq 1}\cdot (r,d')}
\]
\item \label{item:symissubalg} For any $(r,d)$ with $r>0$,
\[
 \SA^{\Dol}_{\BoZ_{\geq 1}\cdot (r,d)}\coloneqq\Sym_{\boxdot}\left(\underline{\BPS}^{\Dol}_{\Lie,\BoZ_{\geq 1}\cdot (r,d)}\otimes\HO^*_{\BoC^*}\right)\subset \SA_{d/r}^{\Dol}
\]
is a sub-algebra object.
\item \label{item:chiindfullcoh} For any $r,d,d'$ with $r>0$, assuming \eqref{item:symissubalg} we have an isomorphism of algebra objects
\[
\mathtt{h}_*\SA_{\BoZ_{\geq 1}\cdot (r,d)}^{\Dol}\cong \mathtt{h}_*\SA_{\BoZ_{\geq 1}\cdot (r,d')}^{\Dol}\in\CD^+(\MHM(B_r))
\]
induced by the $\chi$-independence Theorem \ref{theorem:kinjokosekichiindep}.
\end{enumerate}
\end{conjecture}
The conjectures \eqref{item:isoLieconj} and \eqref{item:isoalgconj} are about finding generators of the Lie algebra $\mathtt{h}_*\underline{\BPS}^{\Dol}_{\Lie,\BoZ_{\geq 1}\cdot (r,d)}$. Indeed, assuming the genus of the curve is $\geq 2$, then both sides are free (Lie) algebras (by Shirshovś theorem, stating that a sub-Lie algebra of a free Lie algebra is itself a free Lie algebra), but we only have a good description of generators when $(r,d)$ is coprime (thanks \cite{davison2022BPS}). Once a description of generators for $(r,d)$ non-coprime is found, one has to check that the chi-independence isomorphisms induce isomorphisms between the generating subobjects. This remark immediately leads to the following proposition (since generators for  free Lie algebra give generators for its enveloping algebra), using the chi-independence for the IC-complexes \cite[Corollary 14.9]{davison2022BPS}.

\begin{proposition}
 \begin{enumerate}
 \item Conjecture \ref{conjecture:chiindependence} \eqref{item:isoLieconj} and Conjecture \ref{conjecture:chiindependence} \eqref{item:isoalgconj} equivalent.
  \item Conjecture \ref{conjecture:chiindependence} \eqref{item:isoLieconj} and \eqref{item:isoalgconj} are true for $(r,d)$ coprime.
 \end{enumerate}
\end{proposition}

\begin{remark}
 We can formulate Conjecture \ref{conjecture:chiindependence} for the Betti CoHAs instead. The difference is that the $\chi$-independence is not know for the Betti moduli stack when $\gcd(r,d)\neq\gcd(r,d')$. This would be a preliminary step for parts \eqref{item:isoLieconj}, \eqref{item:isoalgconj} and \eqref{item:chiindfullcoh} of Conjecture \ref{conjecture:chiindependence} (which would give an isomorphism between the complexes of mixed Hodge modules considered). Part \eqref{item:symissubalg} does not require the $\chi$-independence for its formulation. 
\end{remark}

\subsection{$\chi$-independence and module structure}
There is one additional structure one may consider on the cohomological Hall algebras, that of a module over the ring of tautological classes. The $\chi$-independence results and conjectures suggest that these module structures should agree. In this section, we present results going in this direction.

\subsubsection{Module structure over the ring of tautological classes: Dolbeault}
\label{subsubsection:moduleovertautclassesDol}
Let $C$ be a smooth projective curve, and $\mathfrak{Coh}_{r,d}(C)$ be the stack of coherent sheaves of rank $r$ and degree $d$ on $C$. We let $\CF_{r,d}$ be the universal coherent sheaf on $\mathfrak{Coh}_{r,d}(C)\times C$. We choose a basis $1,\gamma_1,\hdots,\gamma_{2g},[\pt]\in\HO^*(C)$, where $1\in\HO^0(C)$ is the unit, $\gamma_i\in\HO^1(C)$ and $[\pt]$ is the Poincar\'e dual of the class of a point. We define the cohomology classes $a_i,b_{i,j},f_{i}\in\HO^*(\mathfrak{Coh}_{r,d}(C))$ using the K\"unneth isomorphism:
\[
 c_i(\CF_{r,d})=a_i\otimes 1+\sum_{i=1}^{2g}b_{i,j}\otimes \gamma_i+f_i\otimes[\pt]\in\HO^{2i}(\mathfrak{Coh}_{r,d}(C)\times C).
\]
Heinloth described the cohomology algebra $\HO^*(\mathfrak{Coh}_{r,d}(C))$ in \cite{heinloth2012cohomology}.

\begin{theorem}[Heinloth]
\label{theorem:heinloth}
 The algebra $\HO^*(\mathfrak{Coh}_{r,d}(C))$ is freely generated as a graded commutated algebra by the elements $a_i,b_{i,j},f_i$ where $\deg(a_i)=2i$, $\deg(b_{i,j})=2i-1$ and $\deg(f_i)=2i-2$.
\end{theorem}
We denote by $\mathbb{H}_{r,d}$ the ring of tautological classes of rank $r$ and degree $d$.

Let $r,d\in\BoZ$, $r>0$. The Borel--Moore homology of the Dolbeault moduli stack $\HO^{\rmBM}_*(\FM_{r,d}^{\Dol})$ is acted on by the ring of tautological classes $\mathbb{H}_{r,d}$. This action can be described as follows. We let
\[
 \phi\colon\FM_{r,d}^{\Dol}\rightarrow\mathfrak{Coh}_{r,d}(C)\times\FM_{r,d}^{\Dol}
\]
be the map sending a Higgs pair $(\CF,\theta)$ to $(\CF,(\CF,\theta))$. Using the fact that the composition of this map with the second projection is the identity and that $\mathfrak{Coh}_{r,d}(C)$ is a smooth stack of dimension $(g-1)r^2$, we obtain $\phi^*(\BoQ_{\mathfrak{Coh}_{r,d}(C)}\boxtimes\BD\BoQ_{\FM_{r,d}^{\Dol}})\cong\BD\BoQ_{\FM_{r,d}^{\Dol}}$. By adjunction $(\phi^*,\phi_*)$ and pushing down to the Dolbeault moduli space $\CM_{r,d}^{\Dol}$ (and taking into account the shift by the virtual dimension), we obtain a morphism
\[
 \mathbb{H}_{r,d}\otimes \SA_{r,d}^{\Dol}\rightarrow\SA_{r,d}^{\Dol},
\]
which is by definition the \emph{action by tautological classes}. By pushing down to the point, we obtain the action on the Borel--Moore homology $\mathbb{H}_{r,d}\otimes\HO^{\rmBM}_*(\FM_{r,d}^{\Dol})\rightarrow\HO^{\rmBM}_*(\FM_{r,d}^{\Dol})$.

The following result is very easy given Heinloth's theorem.
\begin{lemma}[$\chi$]
 For $r,d,d'\in\BoZ$, $r>0$ such that $\gcd(r,d)=\gcd(r,d')$, we have an isomorphism $\mathbb{H}_{r,d}\rightarrow\mathbb{H}_{r,d'}$, $a_i\mapsto a_i$, $b_{i,j}\mapsto b_{i,j}$ and $f_i\mapsto f_i$.
\end{lemma}
\begin{proof}
 This follows immediately from Heinloth's theorem (Theorem \ref{theorem:heinloth}).
\end{proof}
From now on, we will write $\mathbb{H}_{r}=\mathbb{H}_{r,d}$ for any $d\in\BoZ$ using these isomorphisms.

\subsubsection{Module structure over the ring of tautological classes: Betti}
For the Betti moduli space, for convenience, we define the ring of tautological classes differently. See \cite[\S1.2]{de2021cohomology} for an equivalent description of tautological classes.

Let $r,d\in\BoZ$ with $r>0$. The stack of $\zeta_r^d$-twisted representations of $\pi_1(C,p)$ is realised as a quotient $\FM_{r,d}^{\Betti}=X_{r,d}^{\Betti}/\GL_r$. Therefore, we have a map $\FM_{r,d}^{\Betti}\times C\rightarrow\pt/\GL_r$ which corresponds to the tautological $\GL_r$-bundle $\CT_{r,d}$ over $\FM_{r,d}^{\Betti}\times C$. The ring of tautological classes is defined to be the supercommutative ring freely generated by the classes $c_j(\gamma,\CT_{r,d})$, where we write the Total Chern class $c(\CT_{r,d})=\sum_{j\in\BoZ,\{\gamma\}\subset\HO^*(C,\BoQ_{\ell})}c_j(\gamma,\CT_{r,d})\otimes \gamma$ once a basis $\{\gamma\}\subset \HO^*(C,\BoQ_{\ell})$ is chosen. The same procedure used in \S\ref{subsubsection:moduleovertautclassesDol} gives the actions by tautological classes
\[
 \mathbb{H}_{r,d}\otimes\SA_{r,d}^{\Betti}\rightarrow\SA_{r,d}^{\Betti}, \quad \mathbb{H}_{r,d}\otimes\HO^{\rmBM}_*(\FM_{r,d}^{\Betti})\rightarrow\HO^{\rmBM}_*(\FM_{r,d}^{\Betti}).
\]

\subsubsection{$\chi$-independence of the module structure}

\begin{proposition}
\label{proposition:chiindmodstructureBetti}
 Let $r,d,d'\in\BoZ$ be such that $\gcd(r,d)=\gcd(r,d')$. Then, the isomorphism $\HO^{\rmBM}_*(\FM_{r,d}^{\Betti})\cong\HO^{\rmBM}_*(\FM_{r,d'}^{\Betti})$ given by Galois conjugation is a morphism of $\mathbb{H}_{r,d}$-modules.
\end{proposition}
\begin{proof}
 This follows from \cite[Proof of Proposition 2.1]{de2021cohomology}, where it is explained that Galois conjugation is compatible with the tautological bundles.
\end{proof}

\begin{conjecture}[$\chi$-independence of the module structure for the Dolbeault stacks]
\label{conjecture:Dolchiindmodstructure}
Let $r,d\in\BoZ$, $r>0$. Then, for any $l\geq 1$, $(\SA_{\BoZ_{\geq 1}\cdot(r,d)}^{\Dol})_{(lr,ld)}$ is stable under the $\mathbb{H}_{lr}$-action and if $d'\in\BoZ$ is such that $\gcd(r,d)=\gcd(r,d')$, then the isomorphism $\mathtt{h}_*(\SA_{\BoZ_{\geq 1}\cdot(r,d)}^{\Dol})_{(lr,ld)}\cong \mathtt{h}_*(\SA_{\BoZ_{\geq 1}\cdot(r,d)}^{\Dol})_{(lr,ld')}$ induced by the $\chi$-independence is an isomorphism of $\mathbb{H}_{lr}$-modules.
\end{conjecture}

\begin{remark}
 For $r,d,d'\in\BoZ$ such that $\gcd(r,d)=\gcd(r,d')=1$, by \cite{de2021cohomology}, Galois conjugation on the Betti side induces an isomorphism $\HO^*(\CM_{r,d}^{\Dol})\cong\HO^*(\CM_{r,d}^{\Dol})$ preserving renormalised tautological classes. By the fact $\FM_{r,d}^{\Dol}$ is a $\mathbf{G}_{\mathrm{m}}$-gerb over $\CM_{r,d}^{\Dol}$ in this case and so is a smooth stack, one can identify its Borel--Moore homology with its cohomology and one can deduce a version of Conjecture \ref{conjecture:Dolchiindmodstructure} where the isomorphism is given by Galois conjugation (rather than $\chi$-independence). It is natural to expect that Galois conjugation and $\chi$-independence give the same isomorphism, which does not appear to be obvious.
\end{remark}

Proposition \ref{proposition:chiindmodstructureBetti} solves gives an answer for the $\chi$-independence of the module structure under a condition on the ranks and degrees. We may formulate a conjectural generalisation as follows.

\begin{conjecture}[$\chi$-independence of the module structure for the Betti stack]
 Let $r,d\in\BoZ$, $r>0$. Then, for any $l\geq 1$, the homogeneous component $(\SA_{\BoZ_{\geq 1}\cdot(r,d)}^{\Betti})_{(lr,ld)}$ is stable under the $\mathbb{H}_{lr}$-action and if $d'\in\BoZ$ is such that $\gcd(r,d)=\gcd(r,d')$, then there is an isomorphism $\HO^*((\SA_{\BoZ_{\geq 1}\cdot(r,d)}^{\Betti})_{(lr,ld)})\cong \HO^*((\SA_{\BoZ_{\geq 1}\cdot(r,d)}^{\Betti})_{(lr,ld')})$ of $\mathbb{H}_{lr}$-modules.
\end{conjecture}

\appendix

\section{Curves of genus zero}
We briefly explain the case of genus zero curves (the projective line) separately. In this case, the content of the nonabelian Hodge correspondence is rather empty and the geometry is quite trivial -- we explain the situation for completeness.

Let $C=\BoP^1$ be the projective line. Then, for $r\in\BoN$, the Betti, de Rham and Dolbeault moduli spaces are isomorphic to a point $\pt$ (the projective line is simply connected). The moduli stacks are
\[
 \FM_r^{\Betti}\cong\FM_r^{\dR}\cong\FM_r^{\Dol}\cong\pt/\GL_r,
\]
\[
 \CM_r^{\Betti}\cong\CM_r^{\dR}\cong\CM_r^{\Dol}\cong\pt.
\]
For convenience, we let $\FM_r\coloneqq \pt/\GL_r$ and $\CM_r\coloneqq \pt$ for any $r\geq 1$.
The following is essentially contained in \cite[\S5.3]{davison2023nonabelian} (for Betti and Dolbeault) and the de Rham version is obtained by the same methods.
\begin{theorem}[\cite{davison2023nonabelian}]
Let $C=\BoP^1$ be the projective line.
 \begin{enumerate}
  \item $\underline{\BPS}_{\Lie}^{\dR}\cong \underline{\BPS}_{\Lie}^{\Dol}\cong\underline{\BPS}_{\Lie}^{\Betti}\cong \underline{\BoQ}_{\CM_1}$,
  \item $\underline{\BPS}_{\Alg}^{\dR}\cong \underline{\BPS}_{\Alg}^{\Dol}\cong\underline{\BPS}_{\Alg}^{\Betti}\cong \Sym(\underline{\BoQ}_{\CM_1})\cong\bigoplus_{r\geq 0}\underline{\BoQ}_{\CM_r}$,
  \item $\rmBPS_{\Lie}^{\dR}\cong\rmBPS_{\Lie}^{\Betti}\cong\rmBPS_{\Lie}^{\Dol}\cong\BoQ$ is a one dimensional (necessarily Abelian) Lie algebra,
  \item $\rmBPS_{\Alg}^{\dR}\cong\rmBPS_{\Alg}^{\Betti}\cong\rmBPS_{\Alg}^{\Dol}\cong\BoQ[x]$ is a polynomial algebra in one variable.
  \item We have isomorphisms between cohomological Hall algebras
  \[
   \HO^{\rmBM}_*(\FM^{\Dol})\cong\HO^{\rmBM}_*(\FM^{\Betti})\cong\HO^{\rmBM}_*(\FM^{\dR})
  \]
and isomorphisms of vector spaces
\[
 \HO^{*}(\FM^{\Dol},\BoQ^{\vir})\cong\HO^*(\FM^{\Betti},\BoQ^{\vir})\cong\HO^*(\FM^{\dR},\BoQ^{\vir}).
\]
\end{enumerate}
\end{theorem}

\section{Virtual pullbacks}
In this paper, we made use of the virtual pullbacks for 3-term complexes between categories of mixed Hodge modules defined in \cite{davison2022BPS}. Working with mixed Hodge modules was crucial for us to show that the intersection cohomology of the Hodge moduli space has a canonical morphism to the (complex of mixed Hodge modules underlying the) relative CoHA $\underline{\SA}^{\Hod}$. Then, the structural results Theorems \ref{theorem:CoHAdeRhamintro}, \ref{theorem:CoHAderhamstructure}, \ref{theorem:CoHAHodge} and \ref{theorem:CoHAHodgestructure} are proven in the categories of mixed Hodge modules.

There is a different approach to define pullbacks for a quasi-smooth morphism of derived Artin stacks. It gives a construction of the virtual pullback in Borel--Moore homology or in the constructible derived categories.

We refer to \cite[\S\S3,4]{porta2022non} for description of this virtual pullback (in the setting of non-archimedean geometry), following \cite{khan2019virtual} in the algebraic setting. We could use their definition for the virtual pullback at the cost of forgetting the enriched mixed Hodge module structures on the cohomological Hall algebras considered.

In the setting of \cite{davison2022BPS}, which is adapted to the study of cohomological Hall algebras, the virtual pullback can be constructed \emph{classically} since the relevant cotangent complexes have tor-amplitude $[-1,1]$ and this allowed the author together with Ben Davison and Sebastian Schlegel Mejia to define the CoHA structures at the level of mixed Hodge modules. The construction is adapted from the virtual pullbacks defined in \cite{kapranov2019cohomological} to define CoHAs of surfaces.

\section{Purity for (monodromic) mixed Hodge modules}
\label{section:purityforMMHM}
\subsection{Purity of mixed Hodge modules}
\label{subsection:purityMHM}
The goal of this section is to give an alternative perspective on the purity of mixed Hodge modules, using $*$ or $!$ restrictions to points and the purity of mixed Hodge structures. This will be convenient to give bounds on weights of mixed Hodge modules for which we have a good understanding of the restrictions to points (which is the case for example for families of stacks). This leads to the notion of \emph{punctual weights}. We eventually show that there is no difference between the standard notions of weights (using the weight filtration, part of the datum defining a mixed Hodge module) and the notion of punctual weights (that we define here) for complexes of mixed Hodge modules.

If $Y\subset X$ is a locally closed subset, we uniformly denote by $\imath_Y\colon Y\rightarrow X$ the locally closed immersion. If $x\in X$ is a closed point, we write $\imath_x=\imath_{\{x\}}$.

We work with bounded complexes of mixed Hodge modules, and we explain which statements are valid for unbounded complexes (in one direction) in \S\ref{subsection:unbounded}.

We turn \cite[Proposition 5.1.9]{beilinson2018faisceaux}, valid for $\ell$-adic sheaves, into a definition for complexes of mixed Hodge modules.

All varieties considered are defined over the field of complex numbers.

\begin{definition}[Punctual weights and purity]
\label{definition:punctualweights}
 Let $X$ be an algebraic variety and $w\in\BoZ$ an integer. We say that a complex of mixed Hodge modules $\SF\in\CD^{\rmb}(\MHM(X))$ 
 \begin{enumerate}
  \item has punctual weights $\geq w$ if for any $x\in X$ and any $i\in\BoZ$, the mixed Hodge structure $\CH^i(\imath_x^!\SF)$ has weights $\geq w+i$.
  \item has of punctual weights $\leq w$ if for any $x\in X$ and any $i\in\BoZ$, the mixed Hodge structure $\CH^i(\imath_x^*\SF)$ has weights $\leq w+i$.
  \item pure of punctual weight $w$ if it has punctual weights $\geq w$ and $\leq w$.
 \end{enumerate}
\end{definition}

\begin{proposition}[Weights and distinguished triangles]
\label{proposition:weightstriangles}
Let $\SF\rightarrow\SG\rightarrow\SH\rightarrow$ be a distinguished triangle in $\CD^{\rmb}(\MHM(X))$. Then,
 \begin{enumerate}
  \item If $\SG,\SH$ have punctual weights $\leq w$, then $\SF$ has punctual weights $\leq w$,
  \item If $\SF,\SH$ have punctual weights $\leq w$, then $\SG$ has punctual weights $\leq w$,
  \item If $\SF,\SG$ have punctual weights $\leq w$, then $\SH$ has punctual weights $\leq w+1$.
 \end{enumerate}
We have the dual statements:
 \begin{enumerate}
  \item If $\SG,\SH$ have punctual weights $\geq w$, then $\SF$ has punctual weights $\geq w-1$,
  \item If $\SF,\SH$ have punctual weights $\geq w$, then $\SG$ has punctual weights $\geq w$,
  \item If $\SF,\SG$ have punctual weights $\geq w$, then $\SH$ has punctual weights $\geq w$.
 \end{enumerate}
\end{proposition}
\begin{proof}
The proofs are all similar, and could be deduced from one another by translation of triangles and Verdier duality. Therefore we only prove the statement $(1)$ of the first series.
 
Let $x\in X$. We have a distinguished triangle
\[
 \imath_x^*\SF\rightarrow\imath_x^*\SG\rightarrow\imath_x^*\SH\rightarrow
\]
which induces a long exact sequence in cohomology
\[
\CH^{j-1}(\imath_x^*\SH)\rightarrow \CH^j(\imath_x^*\SF)\rightarrow\CH^{j}(\imath_x^*\SG)\rightarrow\CH^j(\imath_x^*\SH).
\]
One can now use that $\CH^j(\imath_x^*\SG)$ and $\CH^{j-1}(\imath_x^*\SH)$ are mixed Hodge structures with respective weights $\leq j+w$ and $\leq j-1+w$ to deduce that $\CH^j(\imath_x^*\SF)$ is a mixed Hodge structure with weights $\leq j+w$.
\end{proof}

\begin{proposition}[Stability properties for (punctual) weights]
\label{proposition:stabilitypropertiesweights}
Let $\SF\in\CD^{\rmb}(\MHM(X))$ be a complex of mixed Hodge modules. The following statements hold for the usual notion of weights and for the notion of punctual weights.
 \begin{enumerate}
   \item If $\SF$ has (punctual) weights $\leq w$, then $\BD\SF$ has (punctual) weights $\geq w$.
  \item If $\SF$ has (punctual) weights $\leq w$, then $f_!\SF$ and $f^*\SF$ have (punctual) weights $\leq w$,
  \item If $\SF$ has (punctual) weights $\geq w$, then $f^!\SF$ and $f_*\SF$ have (punctual) weights $\geq w$,
  \item If $\SF$ (resp. $\SG$) has (punctual) weights $\leq w$ (resp. $\leq w'$), $\SF\otimes\SG$ has punctual weight $\leq w+w'$,
  \item If $\SF$ (resp. $\SG$) has (punctual) weights $\geq w$ (resp. $\leq w'$), then $\intHom(\SF,\SG)$ has (punctual) weights $\geq -w+w'$,
 \end{enumerate}
\end{proposition}
\begin{proof}
We only prove the statements for the notion of punctual purity as their analogue for the classical notion are known.
\begin{enumerate}
 \item The Verdier duality $\BD$ exchanges the functors $\imath_x^!$ and $\imath_x^*$ and if $V$ is a complex of mixed Hodge structures with weights $\leq w$ or $\geq w$, its dual has weights $\geq w$ or $\leq w$.
 \item We let $f\colon Y\rightarrow X$. Let $y\in Y$. then, $\imath_x^*f^*\SF\cong\imath_{f(x)}^*\SF$ and by definition, it has weights $\leq w$ if $\SF$ has punctual weights $\leq w$.
 
We let $f\colon X\rightarrow Y$. We refrain from proving the statement concerning the weights of $f_!\SF$ as it will follow from Corollary \ref{corollary:coincidenceweights}, stating that the punctual notion of weights coincide with the natural one. We leave it to the reader to check that we do not use the compatibility of punctual weights and the pushforward with proper supports to prove Corollary \ref{corollary:coincidenceweights}.
 \item This follows from $(2)$ by duality $(1)$.
 \item For any $x\in X$, $\imath_x^*(\SF\otimes\SG)\cong (\imath_x^*\SF)\otimes(\imath_x^*\SG)$ and so we are reduced to the properties of the weights of the tensor products of mixed Hodge structures.
 \item This follows from $(4)$ and $(1)$ since $\intHom(\SF,\SG)=(\BD\SF)\otimes\SG$.
\end{enumerate}
\end{proof}

\begin{lemma}
\label{lemma:punctualweights}
 Let $\SF\in\CD^{\rmb}(\MHM(X))$ be a complex of mixed Hodge modules.
\begin{enumerate}
 \item If $\SF$ has weights $\geq w$, then $\SF$ has punctual weights $\geq w$.
 \item If $\SF$ has weights $\leq w$, then $\SF$ has punctual weights $\leq w$.
 \item If $\SF$ is pure of weight $w$, then, $\SF$ is punctually pure of weight $w$.
\end{enumerate}
\end{lemma}
\begin{proof}
 This comes from the classical permanence properties of the bounds for weights of mixed Hodge modules under the functors $\imath^!, \imath^*$.
 
 $(3)$ is clearly implied by $(1)$ and $(2)$.
\end{proof}

What we aim for is a converse to the statements of Lemma \ref{lemma:punctualweights}.

For mixed Hodge modules, the notion of purity is defined globally in terms of the weight filtration. In constrast, for $\ell$-adic sheaves, it is first defined in terms of punctual purity (see \cite[\S 5.1.5]{beilinson2018faisceaux} and the references therein) and the weight filtration is obtained as a property \cite[Th\'eor\`eme 5.3.5]{beilinson2018faisceaux}.

The following is the analogue of \cite[Proposition 5.1.9]{beilinson2018faisceaux} for mixed Hodge modules and gives a description of purity in terms of punctual purity.
\begin{lemma}
\label{lemma:coincidenceweightsforMHM}
 Let $X$ be an algebraic variety and $M\in\MHM(X)$ a mixed Hodge module. Let $w\in\BoZ$. Then,  $M$ has weights $\geq w$ (resp. $\leq w$) if and only if for any $x\in X$, the complex of mixed Hodge structures $\imath_x^!M$ (resp. $\imath_x^*M$) has weights $\geq w$ (resp. $\leq w$), i.e. the punctual weights of $M$ are $\geq w$ (resp. $\leq w$).
\end{lemma}
\begin{proof}
 The direct implication is straightforward since $!$-pullbacks preserve lower bounds of weights (Proposition \ref{proposition:stabilitypropertiesweights}).
 
 For the reverse implication, we let $0\subset M_{w'}\subset \hdots\subset M$ be the weight filtration of $M$. We let $U$ be a smooth open subvariety in the support of $M_{w'}$ such that $\imath_U^!M$ has locally constant cohomology mixed Hodge modules $\CH^j(\imath_U^!M)$, $j\in\BoZ$, i.e. $\CH^j(\imath_U^!M)$ is an admissible variation of mixed Hodge structures on $U$.
 
 By left $t$-exactness of $\imath_U^!$, we have a monomorphism of mixed Hodge modules
\[
 \imath_U^!M_{w'}=\CH^{0}(\imath_U^!M_{w'})\rightarrow\CH^0(\imath_U^!M).
\]
Let $x\in U$. By applying $\imath_x^!$, we obtain a monomorphism of $(2\dim U)$-shifted mixed Hodge modules
\[
 \imath_x^!M_{w'}\rightarrow \imath_x^!\CH^0(\imath_U^!M)=\CH^{2\dim U}(\imath_x^!M)[-2\dim U]
\]
where the last equality follows from the fact that for a locally closed inclusion $f\colon Y\rightarrow X$ between smooth algebraic varieties, $f^!$ induces a functor $\VHS(X)\rightarrow \VHS(Y)[-2(\dim X-\dim Y)]$. The complex of mixed Hodge structures $\imath_x^!M_{w'}$ is pure of weight $w'$ (by Lemma \ref{lemma:punctualweights}) while by assumption on $M$, $\CH^{2\dim U}(\imath_x^!M)[-2\dim U]$ has weights $\geq w$. We must then have $w'\geq w$ and so $M$ has weights $\geq w$.

The statement regarding upper bounds for weights follows by Verdier duality and Proposition \ref{proposition:stabilitypropertiesweights} (1).
\end{proof}
The previous lemma implies that for mixed Hodge modules, the intrinsic notion of weights (defined using the weight filtration) and the notion of weights defined using punctual weights coincide. Therefore, we will just say ``weights" to mean weights or punctual weights in the sequel when we consider a mixed Hodge module. We shall extend in Corollary \ref{corollary:coincidenceweights} this comparison to complexes of mixed Hodge modules.

\begin{lemma}[Augmentation of Lemma \ref{lemma:coincidenceweightsforMHM}]
\label{lemma:augmentation}
 A mixed Hodge module $\SF\in\MHM(X)$ is of weights $\leq w$ (resp. $\geq w$) if and only if any irreducible subvariety $Y$ of $X$ has a dense open subset $U$ such that $\imath_U^*\SF$ (resp. $\imath_U^!\SF$) has punctual weights $\leq w$ (resp. $\geq w$).
\end{lemma}
\begin{proof}
The direct implication is immediate by definition since by Proposition \ref{proposition:stabilitypropertiesweights}, $\imath_U^*$ preserves upper bounds for weights and $\imath_U^!$ preserves lowear bounds for weights.
 
The reverse implication follows from Lemma~\ref{lemma:coincidenceweightsforMHM} by taking $Y=\{x\}$ for varying $x\in X$.
\end{proof}

\begin{proposition}
\label{proposition:puritycomplexcohomology}
 Let $\SF\in\CD^{\rmb}(\MHM(X))$. Then $\SF$ has punctual weights $\leq w$ (resp. $\geq w$) if and only if each mixed Hodge module $\CH^i\SF$ has weights $\leq w+i$ (resp. $\geq w+i$). 
\end{proposition}
\begin{proof}
 The statement with lower bounds for weights is dual to the statement with upper bounds. Therefore, we only concentrate on upper bounds. We first prove the reverse implication. For any $j\in\BoZ$, we have a distinguished triangle
 \[
  \CH^j(\SF)\rightarrow\tau^{\geq j}\SF\rightarrow\tau^{>j}\SF\rightarrow.
 \]
 We can therefore prove that $\tau^{\geq j}\SF$ has punctual weights $\leq w$ by descending induction on $j\in\BoZ$ using the fact that $\tau^{\geq j}\SF=0$ for $j$ big enough using Proposition \ref{proposition:weightstriangles}.
 
 We now prove the direct implication. For $j$ big enough, $\CH^j\SF=0$ and so the statement is trivial. We prove by descending induction on $j$ that $\CH^j\SF$ has (punctual) weigths $\leq w+j$. Let $n\in\BoZ$. We assume that for any $j>n$, $\CH^j\SF$ has weights $\leq w+j$. Therefore, $\tau^{>n}\SF$ has punctual weights $\leq w$. We prove now that $\CH^n\SF$ has weights $\leq w+n$.
 
 By Proposition \ref{proposition:weightstriangles} and the distinguished triangle $\tau^{\leq n}\SF\rightarrow\SF\rightarrow\tau^{>n}\SF\rightarrow$, $\tau^{\leq n}\SF$ has punctual weights $\leq w$. We have a natural adjunction morphism $\tau^{\leq n}\SF\rightarrow\CH^n\SF[-n]$. It is an isomorphism after applying $\tau^{\geq n}$ or, equivalently, $\CH^n$. It fits in a distinguished triangle
 \begin{equation}
  \label{equation:proofdisttriang}
  \tau^{<n}\SF\rightarrow\tau^{\leq n}\SF\rightarrow\CH^n\SF[-n]\rightarrow
 \end{equation}
If $\CH^n\SF$ has weights $>w+n$, there is a surjection $\CH^n\SF\rightarrow\SG$ where $\SG$ is simple and pure of weight $w'+n>w+n$. We let $x$ in the interior of the support of $\SG$ that is not contained in smaller dimensional supports of simple constituents of $\tau^{\leq n}\SF$. We let $k$ be the dimension of the support of $\SG$. Then, $\imath_x^*\tau^{<n}\SF$ is in cohomological degrees $<n-k$ and $\imath_x^*\CH^n\SF$ is in cohomological degrees $\leq n-k$ with $\HO^k(\imath_x^*\CH^n\SF)$ non pure, with some weight $w'+n-k$. The long exact sequence in cohomology associated to the triangle
\[
   \imath_x^*\tau^{<n}\SF\rightarrow\imath_x^*\tau^{\leq n}\SF\rightarrow\imath_x^*\CH^n\SF[-n]\rightarrow
\]
gives an exact sequence $\HO^{n-k}(\imath_x^*\tau^{\leq n}\SF)\rightarrow \HO^{n-k}(\imath_x^*\CH^n\SF[-n])\rightarrow 0$. This proves that $\HO^{n-k}(\imath_x^*\tau^{\leq n}\SF)$ has some weight $w'+n-k>w+n-k$, and so $\tau^{\leq n}\SF$ has not punctual weights $\leq w$ which is a contradiction.
\end{proof}

\begin{corollary}
\label{corollary:coincidenceweights}
 The notions of weights for mixed Hodge modules obtained from the weight filtration on the cohomology and from the notion of punctual weights coincide.
\end{corollary}
\begin{proof}
 This is a combination of Lemma \ref{lemma:coincidenceweightsforMHM} and Proposition \ref{proposition:puritycomplexcohomology}.
\end{proof}

\subsection{Unbounded complexes}
\label{subsection:unbounded}
We explain that some statements remain valid for complexes that are unbounded below or above. To define the unbounded below or above derived categories of mixed Hodge modules, we proceed as in \cite[\S2.3]{davison2021purity}. Namely, for any $n\in\BoZ$, we let $\CD^{\rmb,\geq n}(\MHM(X))$ be the category of complexes of mixed Hodge modules $\CF$ on $X$ such that $\CH^i(\CF)=0$ for $i<n$ and $i\gg0$. Then, for $n\leq m$, we have a truncation functor $\tau_{n,m}\colon \CD^{\rmb,\geq n}(\MHM(X))\rightarrow\CD^{\rmb,\geq m}(\MHM(X))$. For $l\leq n\leq m$, we have a canonical isomorphism of functors $\tau_{l,m}\cong \tau_{l,n}\circ\tau_{n,m}$. We may define $\CD^-{\MHM(X)}$ by the limit of the direct system of categories defined.

We also define $\CD^+(\MHM(X))$ by the dual procedure. The Verdier duality functor exchanges the categories $\CD^+(\MHM(X))$ and $\CD^-(\MHM(X))$.

We have the following functors between these categories, following from the classical $t$-exactness properties of the functors. If $f\colon X\rightarrow Y$ is a morphism between algebraic varieties for which the fibers have dimension $\leq d$, $f^*[d]$ and $f_![d]$ are right $t$-exact while $f_*[-d]$ and $f^![-d]$ are left $t$-exact. In particular, we have the corresponding $t$-exactness properties for locally closed immersions.

\begin{enumerate}
 \item Definition \ref{definition:punctualweights} makes sense for complex unbounded below or above.
 \item Proposition \ref{proposition:weightstriangles} holds for complexes unbounded below or above.
 \item Proposition \ref{proposition:stabilitypropertiesweights} holds for complexes unbounded below or above.
 \item Lemma \ref{lemma:punctualweights} holds for complexes unbounded below or above.
 \item Lemmas \ref{lemma:coincidenceweightsforMHM} and \ref{lemma:augmentation} are only about mixed Hodge modules on $X$ and not complexes.
 \item In Proposition \ref{proposition:puritycomplexcohomology}, the statement concerning upper bounds for weights remains valid for complexes that are bounded above. The statement regarding lower bounds for weights remains valid for complexes that are bounded below. For the first statement, this comes from the fact that the proof proceeds by induction, using that the complexes considered have vanishing cohomology sheaves in high enough degrees.
\end{enumerate}

\subsection{Purity of monodromic mixed Hodge modules}
\subsubsection{Monodromic mixed Hodge modules}

Following the approach of \cite{davison2020cohomological}, the category of monodromic mixed Hodge modules on $X$, denoted by $\MMHM(X)$, is constructed from the category of mixed Hodge modules $\MHM(X\times\BoA^1)$. Namely, it is constructed as a quotient $\CB_X/\CC_X$ of two Serre subcategories of $\MHM(X\times\BoA^1)$. More precisely,
\begin{enumerate}
 \item $\CB_X$ is the category of mixed Hodge modules $\SF\in\MHM(X\times\BoA^1)$ such that for any $x\in X$, the total cohomology of $\imath_{\{x\}\times\mathbf{G}_{\mathrm{m}}}^*\SF$ is an admissible variation of mixed Hodge structures on $\mathbf{G}_{\mathrm{m}}$,
 \item $\CC_X$ is the full subcategory of $\CB_X$ of mixed Hodge modules $\SF\in\MHM(X\times\BoA^1)$ such that for any $x\in X$, the total cohomology of $\imath_{\{x\}\times\BoA^1}\SF$ is an admissible variation of mixed Hodge structures on $\BoA^1$.
\end{enumerate}

The derived category of monodromic mixed Hodge modules is $\CD^{\rmb}(\MMHM(X))$. It can also be constructed from $\CD^{\rmb}(\MHM(X\times\BoA^1))$ using the Abelian categories $\CB_X$ and $\CC_X$ and their derived categories.

If $f\colon X\rightarrow Y$ is a morphism between algebraic varieties, we have functors $f^*, f_*, f^!, f_!$ between the corresponding derived categories.

The tensor product is built as follows. For $\SF, \SG\in\CD^{\rmb}(\MMHM(X))$, we let
\[
 \SF\otimes\SG\coloneqq (X\times\BoA^1\times\BoA^1\xrightarrow{\id\times +} X\times\BoA^1)_*(\pr_{1,2}^*\SF\otimes\pr_{1,3}^*\SG)
\]
where $+\colon \BoA^1\times\BoA^1\rightarrow\BoA^1$ is the sum morphism, and $\pr_{1,j}\colon X\times\BoA^1\times\BoA^1\rightarrow X\times\BoA^1$ is the projection on the $1$st and $j$th factors for $j=2,3$.

For $X,Y$ two scheme over $Z$, we also defined the external tensor product $\boxtimes_Z\colon \CD^{\rmb}(\MMHM(X))\times\CD^{\rmb}(\MMHM(Y))\rightarrow\CD^{\rmb}(\MMHM(X\times_ZY))$ as follows
\[
 \SF\boxtimes_{Z}\SG\coloneqq(X\times_ZY\times\BoA^1\times\BoA^1\xrightarrow{\id\times +} X\times_ZY\times\BoA^1)_*(\pr_{1,3}^*\SF\otimes\pr_{2,4}^*\SG).
\]

Monodromic mixed Hodge modules extend the category of mixed Hodge modules, in the sense that there is a fully faithful embedding
\[
 \imath_*\colon\MHM(X)\rightarrow\MMHM(X).
\]
At the level of derived categories, we have the fully faithful embedding
\[
 \imath_*\colon\CD^{\rmb}(\MHM(X))\rightarrow\CD^{\rmb}(\MMHM(X)).
\]

We now explain how to obtain a forgetful functor
\[
 \forg^{\mon}_X\colon\MMHM(X)\rightarrow\Perv(X)
\]
extending the functor $\rat_X$.

We define
\[
\begin{matrix}
 \Psi_X&\colon&\MMHM(X)&\rightarrow&\MMHM(X)\\
       &      &\SF     &\mapsto    &(X\times\BoA^2\xrightarrow{\id\times +}X\times\BoA^1)_*[\SF\boxtimes(\mathbf{G}_{\mathrm{m}}\rightarrow\BoA^1)_!\BoQ_{\mathbf{G}_{\mathrm{m}}}[1]]
\end{matrix}
\]
and
\[
\begin{matrix}
 \Theta_X&\colon&\MMHM(X)&\rightarrow&\MHM(X\times\mathbf{G}_{\mathrm{m}})\\
         &      &\SF     &\mapsto    &(X\times\mathbf{G}_{\mathrm{m}}\rightarrow X\times\BoA^1)^*\Psi_X(\SF).
\end{matrix}
\]
The functor $\Theta_X$ becomes an equivalence of categories when we restrict the target to the full subcategory of monodromic objects in $\MHM(X\times\mathbf{G}_{\mathrm{m}})$ (i.e. of mixed Hodge modules $\SF\in\MHM(X\times\mathbf{G}_{\mathrm{m}})$ such that for any $x\in X$, the cohomology mixed Hodge modules of the pullback $(\{x\}\times\mathbf{G}_{\mathrm{m}})^*\SF$ are admissible variations of Hodge structures on $\mathbf{G}_{\mathrm{m}}$).

Then, we let $\forg_X^{\mon}\coloneqq \rat_Xe^*\Theta_X[-1]\colon \MMHM(X)\rightarrow\Perv(X)$, where $e\colon X\rightarrow X\times\mathbf{G}_{\mathrm{m}}$, $x\mapsto (x,1)$.

\subsubsection{Purity}
The discussion of \S\ref{subsection:purityMHM} can be extended readily to monodromic mixed Hodge modules. In this case, the $*$ and $!$ restrictions of a monodromic mixed Hodge module to a point are monodromic mixed Hodge structures on the point, and punctual purity or bounds on weigths refer to properties for these monodromic mixed Hodge structures. For more details, we refer to \cite[\S2.1]{davison2020cohomological}. We only mention this for the sake of future use as monodromic mixed Hodge modules do not appear in the body of this paper.
\printbibliography
\end{document}